%% file: main.tex
\documentclass[11pt]{amsart}
\input{package}
\input{newtheorem}

\input{convention}

\title{Perturbation and Pruning of Nondegenerate $\ZT$ Harmonic $1$-forms}

\begin{document}

\author{Jiahuang Chen} 

\address{AMSS}
\email{chenjiahuang@amss.ac.cn}

\begin{abstract}
  We prove that for any nondegenerate $\ZT$ harmonic $1$-form, there exists a metric perturbation producing a new nondegenerate $\ZT$ harmonic $1$-form whose ordinary zero set is discrete. As an application, we show that for generic smooth nondegenerate $\ZT$ harmonic $1$-forms, the leaf spaces are $\ZZ$-trees. Moreover, we show that if a $3$-dimensional rational homology sphere admits a smooth nondegenerate $\mathbb{Z}/2$-harmonic $1$-form, then there exists another nondegenerate $\mathbb{Z}/2$-harmonic $1$-form whose singular locus has exactly two connected components.
\end{abstract}

\maketitle

\section{Introduction}
Let $(M^n,g)$ be a oriented Riemannian $n$-manifold with metric $g$, where $n\ge 2$. A $\ZT$ harmonic $1$-form $v$ on $(M,g)$ is a bounded $1$-form that is $2$-valued and harmonic, whose zero set are closed subset of $M$ with Hausdorff dimension at most $n-2$. In particular, when $n=2$ and $(M,g)$ is a Riemann surface, every $\ZT$ harmonic $1$-form $v$ is the real part of the square root of some holomorphic quadratic differential $q$, i.e. $\re\sqrt{q}$. 

The concept of $\ZT$ harmonic $1$-forms was introduced by C.\,Taubes in \cite{TaubesPSL2C}, to study the limit behavior of degenerating sequences of $\SL$-connections on $2$- or $3$-dimensional closed manifolds. In particular, Taubes' work provides a compactification of flat $\SL$-connections, in where the ideal points are $\ZT$ harmonic $1$-forms. 

On the other hand, $\ZT$ harmonic $1$-forms have also appeared in the study of other gauge-theoretic equations \cite{taubes20,Taubes22,HaydysWalpuski,WalpuskiZhang,SamanLi} and in the geometry of calibrated submanifolds \cite{Hebranchdefor}. However, despite their importance, few properties of $\ZT$ harmonic $1$-forms are known, particularly concerning the regularity and topological restrictions of their zero sets. Assuming smoothness of the zero set and nonvanishing leading coefficients, S.\,Donaldson established the local deformation theory for $\ZT$ harmonic $1$-forms \cite{donaldsondeformation2019}. See also \cite{parker2025deformations}. To illustrate, we give precise definitions:

\begin{definition}
    A $\ZT$ harmonic $1$-form on an oriented manifold $M$ (not necessarily compact) is given by a closed subset $\Zl\subsetneq M$ and a two-valued section $v$ on $T^\ast M\setminus\Zl$, satisfying the following:
\begin{enumerate}[label=(\roman*)]
    \item For each $x\notin\Zl$,there exists an neighborhood $U\subset M\setminus\Zl$ of $x$ such that $v=\pm w$, where $v$ is a $1$-form and $dw=d\star_g w=0$.
    \item For each $x\in\Zl$, on any neighborhood $U$ of $x$, $v\ne\pm w$ for any $1$-form $w$ on $U\setminus\Zl$.
    \item For every precompact open subset $U$ of $M$, we have \[\int_{U\setminus\Zl}|v|^2+|\nabla v|^2<+\infty.\]
    \item There exists a constant $C$ and $\alpha>0$ such that for every $x\in\Zl$ and every $r>0$ less than the injectivity radius of $M$ at $x$, we have $\int_{B_r(x)}|v|^2<Cr^{n+\alpha}$. 
\end{enumerate}    
\end{definition}
 The closed set $\Zl$ is called the \textbf{singular locus} of $v$.
\begin{definition}
    A $\ZT$ harmonic $1$-form $(v,\Zl)$ on $(M,g)$ is said to be \textbf{smooth} if the singular locus $\Zl$ is a smooth co-oriented submanifolds with codimension $2$ in $M$.
    A smooth $\ZT$ harmonic $1$-form $(v,\Zl)$ is said to be \textbf{nondegenerate} if there is a tubular neighborhood $N$ of $\Zl$ and a positive constant $c$ such that $|v|(x)\ge c\cdot\mathrm{dist}(x,\Zl)^{1/2}$. 
    
    A point $y\in M\setminus\Zl$ at which $|v|(y)=0$ is called an \textbf{ordinary zero} of $v$. 
\end{definition}

We note that in general the singular locus $\Zl$ need not be smooth, although it is expected that there are some generic regularity results for $\Zl$. To the best of the author's knowledge, the best results on the regularity of $\Zl$ are given in \cite{taubes2014, Boyu17}. On the other hand, a number of $\ZT$ harmonic $1$-forms in $\RR^3$, whose singular loci has graphical singularities, are constructed in \cite{TWI,TWII,ChenHe}. Also, an index theorem was established in \cite{IndexHMT} when $\Zl$ is a graph. 

Nevertheless, throughout this article, all $\ZT$ harmonic $1$-forms are assumed to be \textbf{smooth}, without further mention.

The nondegeneracy condition for $\ZT$ harmonic spinors was studied in \cite{Takahashi,TakahashiIndex}. In \cite{donaldsondeformation2019}, Donaldson introduced the nondegeneracy for $\ZT$ harmonic $1$-forms, to apply a version of Nash-Moser implicit function theorem in proving a local deformation theorem. This approach has shed light on the study of $\ZT$ harmonic $1$-forms, inspiring a series subsequent works \cite{Hebranchdefor,HeParker,salm,parker2025deformations,parker2024gluing}.

\vspace{1ex}
We now highlight two observations: (i) \textit{$\ZT$ harmonic $1$-forms generalize quadratic differential to higher dimensions}; (ii) \textit{nondegeneracy is a key assumption in the study of $\ZT$ harmonic $1$-forms}. These two observations lead to the following two questions, respectively: 
\begin{question}\label{QUESTION}
    \begin{itemize}
        \item What topological constraints are imposed on the zero set of a $\ZT$ harmonic $1$-form? (Recall that the number of zeros of a quadratic differential is determined by the genus of the Riemann surface.)
        \item How general are smooth nondegnerate $\ZT$ harmonic $1$-forms? Are there always some smooth nondegnerate $\ZT$ harmonic $1$-forms in the ideal boundary of Taubes compactification of the $\SL$ character variety of a $3$-manifold $M$?
    \end{itemize}
\end{question}

In this paper, we establish results related to these questions, but does not attempt to fully resolve them (which is rather ambitious).

Let $\M$ be the space of Riemannian metrics on a closed oriented manifold $M^n~(n\ge 2)$ and $\mathcal{S}$ be the space of smooth codimension $2$ submanifolds in $M$. Suppose there is a nondegenerate $\ZT$ harmonic $1$-form $v_0$ with (smooth) singular locus $\Zl_0$ under the metric $g_0$ on $M$. Our first main result is:
\begin{theorem}[Perturbation]\label{main1}
    In a neighborhood $\mathcal{U}$ of $g_0$ in $\M$, there exists a dense open set $\mathcal{U}'\subset\mathcal{U}$ with the following property: for any $g\in\mathcal{U}'$, there exists a unique $\Zl$ in a neighborhood of $\Zl_0$ in $\mathcal{S}$, such that there exists a nondegenerate $\ZT$ harmonic $1$-form $v$ under the metric $g$, with singular locus $\Zl$ and discrete ordinary zero set.
\end{theorem}
    
This theorem simplifies the structure of the ordinary zero set of a nondegenerate $\ZT$ harmonic $1$-form and serves as a $\ZT$ version of the transversality theorem for harmonic $1$-forms proved in \cite{hondatrans}. The main difficulty in proving Theorem~\ref{main1} is the lack of Sard's lemma in the Fr\'echet category, which is different with \cite{hondatrans}, where the setting is within the Banach category. In this article, we use a gluing argument to circumvent this difficulty.

\vspace{1ex}
The simplification of the zero set, will also simplify the leaf space of $\ZT$ harmonic $1$-forms as we now explain.

The pullback $\tilde{v}$ of a $\ZT$ harmonic $1$-form $v$ to the universal cover space $\tilde{M}$ defines a singular foliation $\ker\tilde{v}$ on $\tilde{M}$. The leaf space $\T_{\tilde{M},\tilde{v}}$ is always an $\RR$-tree (see Section \ref{Sec4} for precise definitions), endowed with an isometric $\pi_1(M)$-action \cite{HeWentworthZhang}.

Consider the $2$-fold branched covering $p:\hat{M}\to M$ branched along the singular locus $\Zl$ of $v$. We identify the preimage $p^{-1}(\Zl)$ with $\Zl$. Then $\hat{v}:=p^\ast v$ is a harmonic $1$-form on the manifold $\hat{M}$ under the metric $\hat{g}=p^\ast g$, which has conical singularities along $\Zl$. Moreover, let $\tau$ be the canonical involution on $\hat{M}$, then the $1$-form $\hat{v}$ is anti-invariant under the action of $\tau^\ast$. Thus $[\hat{v}]\in H^1_-(\hat{M};\RR)=\{[\sigma]\in H^1(\hat{M};\RR):\tau^\ast[\sigma]=-[\sigma]\}$. Denote the cohomology space $H^1_-(\hat{M};\RR)$ by $H$. (Since we only deform the singular locus in a small neighborhood in $\mathcal{S}$, the topology of the branched cover does not change, and we retain the notation $\hat{M}$.) 

We will show that:
\begin{theorem}\label{main2}
    In a neighborhood $\mathcal{U}$ of $(g_0,[\hat{v}_0])$ in $\M\times H$, there is a dense open set $\mathcal{U}'\subset\mathcal{U}$, such that, for any $(g,h)\in\mathcal{U}'$, there exists a unique $\Zl$ in a neighborhood of $\Zl_0$ in $\mathcal{S}$, and a nondegenerate $\ZT$ harmonic $1$-form $v$ with respect to the metric $g$ satisfying:
    \begin{enumerate}[label=(\roman*)]
        \item $v$ has singular locus $\Zl$.
        \item $[\hat{v}]=h$ in $H^1(\hat{M};\RR)$.
        \item $\T_{\tilde{M},\tilde{v}}$ is a $\ZZ$-tree.
    \end{enumerate}
\end{theorem}

Let's explain the relation between Theorem \ref{main2} and Question \ref{QUESTION}. In \cite{MorganShalenI}, an algebraic compactification, referred to as the Morgan-Shalen compactification, was defined for the $\SL$ character variety of a $2$- or $3$-manifold $M$. Roughly speaking, the ideal points of the Morgan-Shalen compactification are $\RR$-trees equipped with $\pi_1(M)$-actions. It was proposed in \cite{TaubesPSL2C} and proved in \cite{HeWentworthZhang} that the Taubes compactification is related to the Morgan-Shalen compactification. The $2$-dimensional case was established earlier in \cite{DDW00}.

Indeed, one can associate an $\RR$-tree $\T$ with a $\pi_1(M)$-equvariant harmonic maps $u:\tilde{M}\to \T$, and the gradient field of $u$ then defines a $\ZT$ harmonic $1$-form on $M$. This correspondence yields a continuous surjection from the Morgan-Shalen compactification to the Taubes compactification. See \cite{DDW00,HeWentworthZhang} for details. 

Moreover, it is a fact that, roughly speaking, $\ZZ$-trees are dense in the ideal boundary of the Morgan-Shalen compactification \cite[Theorem II.4.3]{MorganShalenI}. Thus if a smooth nondegenerate $\ZT$ harmonic $1$-form is contained in the ideal boundary of the Taubes compactification, as in the second item in Question \ref{QUESTION}, it would follow that such a form can be perturbed so that its leaf space becomes a $\ZZ$-tree.

Given this, Theorem~\ref{main2} may therefore be viewed as a supporting evidence for the existence of nondegenerate $\ZT$ harmonic $1$-forms in the ideal boundary of the Taubes compactification.

\vspace{2ex}
Next we focus on $3$-manifolds. In this case, a singular locus $\Zl$ is a link in $M$. Some examples of nondegenerate $\ZT$ harmonic $1$-forms on rational homology spheres are found in \cite{Z3He,HeParker}. 

It is natural to conjecture that, given a link $\Zl$ in a rational homology sphere $M$, the existence of a $\ZT$ harmonic $1$-form 
with $\Zl$ as its singular locus imposes topological restrictions on $\Zl$. This is related to the first item in Question~\ref{QUESTION}.

A.\,Haydys \cite{Haydys} proved that if there is a $\ZT$ harmonic $1$-form on $M$, with singular locus $\Zl$, then the Alexander polynomial $\Delta_{\Zl}(t)$ must vanish at $t=-1$, implying that $\Zl$ has at least two connected components. We prove that this lower bound is sharp in the following sense:

\begin{theorem}[Pruning]\label{main3}
    Suppose that there is a nondegenerate $\ZT$ harmonic $1$-form $v$ with singular locus $\Zl$ on a rational homology sphere $(M,g)$. Then there exists another nondegenerate $\ZT$ harmonic $1$-form on $M$ with respect to another metric $g'$, whose singular locus is a sublink $\Zl'$ of $\Zl$, consisting of exactly two connected components.   
\end{theorem}

The major tool used in the proofs of Theorem \ref{main1} and \ref{main2}, is the construction from Calabi's intrinsic characterization \cite{Calabi,intrinsic} of harmonic $1$-forms. For other applications of this method, see \cite{Yan}. More general results will be established in a forthcoming work \cite{ChenHeYan}.

\vspace{1ex}

\begin{flushleft}
    \textbf{Acknowledgements.}
    ~The author wish to express his gratitude to his advisor Siqi He for his insights and suggestions. He would also like to thank Dashen Yan for many helpful discussions.  The author is supported by National Key R\&D Program of China (No.2023YFA1010500).
\end{flushleft}

\vspace{1ex}
\begin{flushleft}
\textbf{Notations.} 
~Here is a summary of various notations used through the article:
\end{flushleft}
\begin{itemize}
    \item $M$ denotes a closed oriented smooth manifold. We write $(M,g)$ to refer to a Riemannian manifold $M$ with smooth metric $g$. 
    \item We use $\Zl$ to denote a co-oriented smooth submanifold of $M$ of codimension $2$. And $E\to M\setminus\Zl$ will always denote a real line bundle, with nontrivial monodromy along any small loop surrounding $\Zl$. 
    
    \item For the existence of $E$, we impose certain topological conditions on $\Zl$. Let $N\to \Zl$ be the normal (complex) line bundle of $\Zl$ in $M$. It suffices to assume that the first Chern class $c_1(N)$ is even in $H^2(\Zl;\ZZ)$. This is equivalent to the condition that $w(\Zl)=\iota^\ast w(M)$, where $\iota$ is the inclusion $\iota:\Zl\to M$, and $w(-)$ is the total Stiefel-Whitney class of a manifold. 

    In this article, we always assume codimension $2$ submanifolds of $M$ under consideration satisfy the above topological conditions.

    \item The real line bundle $E\to M\setminus\Zl$ is equipped with the canonical flat metric.

    \item Consider the $2$-fold branched covering $p:\hat{M}\to M$ with branched set $\Zl$. We use $\tau$ to denote the canonical involution on $\hat{M}$.

    \item We use $\dist_g(-,-)$ to denote the distance on $(M,g)$. 
\end{itemize}

\section{Preliminaries}

In this section, we gather some standard background material that will be used in the sequel. To begin with, we introduce a flat model and define multivalued functions and $1$-forms on it. We then extend these notions to general manifolds and define various norms for such multivalued objects. Finally, we give a more practical definition of nondegenerate $\ZT$ harmonic $1$-forms. This section serves as a brief review of the first two sections of \cite{donaldsondeformation2019}. For recent developments in the analysis of $\ZT$ objects, see \cite{Walpuski-Bera}.

\subsection{The flat model}
Consider the Euclidean space $\RR^n$, which we identify with $\CC\times\RR^{n-2}$. For a point $x\in\CC\times\RR^{n-2}$, we denote its coordinate by $(z,t)$, where $z\in\CC$ and $t=(t_i)_{1\le i\le n-2}\in\RR^{n-2}$. We further write $z$ in polar coordinates as $z=re^{i\theta}=(r,\theta)$. There is a $2$-fold branched covering $p:\CC\times\RR^{n-2}\to \CC\times\RR^{n-2}$, defined by $(r,\theta,t)\mapsto (r,2\theta,t)$. 

The pullback of the flat metric $g_{\mathrm{flat}}$ of $\CC\times\RR^{n-2}$ under $p$, which we henceforth denote by $\hat{g}_{\mathrm{flat}}$, is conical, with singularities along $\{0\}\times\RR^{n-2}$. Note that $\hat{g}_{\mathrm{flat}}$ is invariant under the involution $\tau$ which defined by $(z,t)\mapsto (-z,t)$.   

In coordinate $(z,t)=(r,\theta,t_i)$, we have $\hat{g}_{\mathrm{flat}}(r,\theta,t_i)=dr^2+4r^2d\theta^2+\sum_idt_i^2$.

\subsection{Multivalued functions and $1$-forms}
Firstly, let's discuss on the flat model. Consider a flat real line bundle $E\to\CC^\ast\times\RR^{n-2}$, with nontrivial monodromy along any loop surrounding $\{0\}\times\RR^{n-2}$. Here, the base space $\CC\times\RR^{n-2}$ is equipped with the flat metric $g_{\mathrm{flat}}$.

A section $f$ of $E$ can be lifted to be a function $\hat{f}$ on the manifold $(\CC^\ast\times\RR^{n-2},\hat{g}_{\mathrm{flat}})$, such that $\hat{f}(-z,t)=-\hat{f}(z,t)$. We say such a function is odd. Since $p(-z,t)=p(z,t)$, $f$ can be regarded as a $2$-valued function on $(\CC^\ast\times\RR^{n-2},g_{\mathrm{flat}})$, and the action of Laplacian on such a function $\Delta_{g_\mathrm{flat}} f$ is well-defined. Note that $\Delta_{g_\mathrm{flat}} f=0$ if and only if $\Delta_{\hat{g}_{\mathrm{flat}}}\hat{f}=0$.

Next, consider a manifold $(M,g)$ together with a codimension $2$ submanifold $\Zl$. Analogously, for the flat real line bundle $E\to M\setminus\Zl$, a section $f$ of $E$ is considered as a $2$-valued function on $M\setminus\Zl$, and corresponds to an odd function $\hat{f}$ on $\hat{M}\setminus\Zl$. We will say that $f$ is a $2$-valued function on $M$ with singular locus $\Zl$, and define $\Delta_g f$ accordingly. 

Similarly, sections of the twisted cotangent bundle $E\otimes T^\ast(M\setminus\Zl)$ is called $2$-valued $1$-forms on $M\setminus\Zl$. Their lift to $\hat{M}\setminus\Zl$ are called odd $1$-forms on $\hat{M}\setminus\Zl$. As above, we say a section $v$ of $E\otimes T^\ast(M\setminus\Zl)$ is a $2$-valued $1$-form on $M$ with singular locus $\Zl$, and its Laplacian $\Delta_g v$ is defined in the same manner.

\subsection{Section spaces}\label{sectionspaces}
In this subsection, we consider section spaces of the flat real line bundle $E\to M\setminus\Zl$, or equivalently, $2$-valued function spaces on $(\hat{M}\setminus\Zl,\hat{g})$, as defined in \cite{donaldsondeformation2019}. 

Firstly, let's discuss on the flat model. Fix an exponent $\alpha\in (0,\frac{1}{2})$. Consider a $2$-valued function $f$ on $\CC^\ast\times\RR^{n-2}$ and let $\hat{f}$ be the corresponding odd function on $(\CC^\ast\times\RR^{n-2},\hat{g}_{\mathrm{flat}})$.

We define the H\"older norm of a $2$-valued function $f$ to be \[\|\hat{f}\|_{,\alpha}=\sup_{x\ne x'}\frac{|\hat{f}(x)-\hat{f}(x')|}{|x-x'|^{\alpha}},\] where $x$ and $x'$ are any two distinct points on $\CC^\ast\times\RR^{n-2}$. The H\"older norm for $2$-valued functions on $M$ with singular locus $\Zl$ is defined similarly. 

Let $\underline{\mathcal{T}}$ be the set of $n$ commuting vector fields $r\frac{\partial}{\partial r},\,\frac{\partial}{\partial\theta},\frac{\partial}{\partial t_i}$, on the flat space $(\CC\times\RR^{n-2},g_{\mathrm{flat}})$. Note that the first two vector fields have norm $r$ and vanishes along $\{0\}\times \RR^{n-2}$. 

For $k\ge 1$, we define:
\begin{itemize}
    \item $\underline{\mathcal{T}}_k=\{L_1L_2\cdots L_k:L_i\in\underline{\mathcal{T}},1\le i\le k\}$.
    \item $\mathcal{T}_k=\mathrm{span}_{\RR}\cup_{j\le k}\underline{\mathcal{T}}_j$.
\end{itemize}

\begin{definition}
    We define the $\mathcal{D}^{k,\alpha}$ norm of a $2$-valued function $f$ to be \[\|f\|_{\mathcal{D}^{k,\alpha}}=\max_{0\le j\le k,\,L\in\underline{\mathcal{T}}_j}\|Df\|_{,\alpha}.\]
The $\mathcal{D}^{k,\alpha}$ norm of the corresponding odd function $\hat{f}$ is defined similarly.
\end{definition}

If an odd function $f$ has finite H\"older norm, then it's clear that $f$ can be extended continuously to the whole space $\CC\times\RR^{n-2}$, so that $f(0,t)\equiv 0$.

The definition of the $\mathcal{D}^{k,\alpha}$ norm naturally extends to $2$-valued functions with singular locus $\mathcal{Z}$ on a general manifold $M$. The construction proceeds via the standard localization approach sketched as follows:
\begin{enumerate}[label=(\roman*)]
    \item Choose an open cover $\{U_i\}$ of $M$ adapted to $\mathcal{Z}$.
    \item Select corresponding collections of vector fields $\mathcal{T}_k$ and $\underline{\mathcal{T}}_k$ in each neighborhood (retaining the same notation for simplicity).
    \item Employ a partition of unity $\{\rho_i\}$ subordinate to $\{U_i\}$ to piece together local definitions.
\end{enumerate}
Standard arguments show that the resulting function space is independent of the particular choices made in this process. 

One can also consider the sheaf of local $\D^{k,\alpha}$ sections, denoted by $\D^{k,\alpha}_{\mathrm{loc}}$, on $M$. The space of global sections is denoted by $\D^{k,\alpha}(M)$, which is clearly a Banach space, although its norm depends on the choice of:
(i) The vector field collections $\mathcal{T}_k$ and $\underline{\mathcal{T}}_k$, (ii) The partition of unity, (iii) The local coordinate charts.
However, different choices yield equivalent norms.

On the flat model $(\CC\times\RR^{n-2},g_{\mathrm{flat}})$, another sheaf $\mathcal{E}^{k+2,\alpha}_{\mathrm{loc}}$ was introduced in \cite{donaldsondeformation2019}, which consists of $2$-valued functions $f\in\D^{k+2,\alpha}$ such that $\Delta_{g_{\mathrm{flat}}} f\in\D^{k,\alpha}$. Subtleties arise when extending this definition to a general manifold $(M,g)$. One need to introduce the following notation:

\begin{definition}{\cite[Definition 3.1]{donaldsondeformation2019}}\label{normal}
    A normal structure on $\Zl\subset M$ consists of the following data:
    \begin{enumerate}
        \item A normal bundle $N\subset TM|_{\Zl}$, complementary to $T\Zl$.
        \item A Euclidean structure on $N$.
        \item A $2$-jet along $\Zl$ of diffeomorphisms from the total space of $N$ to $M$, extending the canonical $1$-jet.
    \end{enumerate}
\end{definition}

Any metric $g$ induces such a structure. Locally near $\mathcal{Z}$, we can identify an open ball $V$ with an open set in $\mathbb{C}\times\mathbb{R}^{n-2}$ (non-isometrically), so that $V\cap\mathcal{Z}\subset \{0\}\times\mathbb{R}^{n-2}$. The normal structure yields an orthonormal frame $\{\underline{Z}_1,\underline{Z}_2,Y_1,\ldots,Y_{n-2}\}$ on $V$ satisfying $\frac{1}{2}\sum_{i=1}^2 \mathrm{div}(\underline{Z}_i)\underline{Z}_i = \mu$ along $\mathcal{Z}\cap V$, where $\mu$ is the mean curvature of $\Zl$. This gives the decomposition $\Delta_g = \Delta_{g_{\mathrm{flat}}} + 2\mu + \mathcal{L}$ on $V$, with $\Delta_{g_{\mathrm{flat}}}$ the model Laplacian and $\mathcal{L}\in\mathcal{T}_2$ (see \cite[Section 3]{donaldsondeformation2019}).

For $f \in \mathcal{D}^{k+2,\alpha}_{\mathrm{loc}}(V)$ with $\Delta_{g_{\mathrm{flat}}} f \in \mathcal{D}^{k,\alpha}_{\mathrm{loc}}(V)$, the presence of the first-order term $2\mu$ in $\Delta_g$ means $\Delta_g f$ need not belong to $\mathcal{D}^{k,\alpha}_{\mathrm{loc}}(V)$. Therefore, the sheaf $\mathcal{E}^{k+2,\alpha}_{\mathrm{loc}}$ cannot be defined simply as those $f$ satisfying both $f \in \mathcal{D}^{k+2,\alpha}_{\mathrm{loc}}$ and $\Delta_g f \in \mathcal{D}^{k,\alpha}_{\mathrm{loc}}$, since such a definition would obstruct the reduction to flat local models in our analysis.

To address this issue, we must modify the Laplacian $\Delta_g$ to eliminate the first-order term $2\mu$. On $(M,g)$, we define an positive smooth function $W$, such that $\nabla W=-\mu$ along $\Zl$ while $W|_{\Zl}\equiv 1$. Then on the small open set $V$, $\tilde{\Delta}_g:=W\Delta_g W^{-1}=\Delta_{g_{\mathrm{flat}}}+\mathcal{L}$ for some $\mathcal{L}\in\mathcal{T}_2$. 

The restriction of the modified operator $\tilde{\Delta}_g$ on $V$, can be extended to $\CC\times\RR^{n-2}$ and becomes what is referred to as an ``admissible operator".
\begin{definition}
    A differential operator $\tilde{\Delta}$ over $(\CC\times\RR^{n-2},g_{\mathrm{flat}})$ is called admissible if:
    \begin{enumerate}
        \item $\tilde{\Delta}=\Delta_{g|_{\mathrm{flat}}}+\mathcal{L}$ with $\mathcal{L}\in\mathcal{T}_2$.
        \item $\tilde{\Delta}=W^{-1}\Delta_g W$, where $W$ is a smooth positive function and $g$ is a smooth Riemannian metric on $\CC\times\RR^{n-2}$.
        \item $\tilde{\Delta}=\Delta_{g|_{\mathrm{flat}}}$ outside a compact set in $\CC\times\RR^{n-2}$.
    \end{enumerate}
\end{definition}

With this setup, we have:

\begin{lemma}{\cite[Proposition 2.6]{donaldsondeformation2019}}\label{admissiblesheaf}
    Consider an admissible operator $\tilde{\Delta}$ over $(\CC\times\RR^{n-2},g_{\mathrm{flat}})$, if $f\in\D^{1,\alpha}_{\mathrm{loc}}$ and $\tilde{\Delta}f\in\D^{k,\alpha}_{\mathrm{loc}}$, then $f\in\D^{k+2,\alpha}_{\mathrm{loc}}$. As a result, $f\in\mathcal{E}^{k+2,\alpha}_{\mathrm{loc}}$ if and only if $f\in\D^{k+2,\alpha}_{\mathrm{loc}}$ and $\tilde{\Delta}f\in\D^{k,\alpha}_{\mathrm{loc}}$ for an admissible operator $\tilde{\Delta}$.
\end{lemma}

We now define the sheaf $\E^{k+2,\alpha}_{\mathrm{loc}}$ to consist of $2$-valued functions $f$ on $M$ with singular locus $\Zl$, satisfying $f\in\D^{k+2,\alpha}_{\mathrm{loc}}$ and $\tilde{\Delta}_g f\in\D^{k,\alpha}_{\mathrm{loc}}$. The space of global sections $\E^{k+2,\alpha}(M)$ is equipped with the norm $\|f\|_{\E^{k+2,\alpha}}:=\|\tilde{\Delta}_gf\|_{\D^{k,\alpha}}$, giving it a Banach space structure that, like $\mathcal{D}^{k,\alpha}(M)$, depends on various choices. On the other hand, the sheaf $\mathcal{E}^{k+2,\alpha}_{\mathrm{loc}}$ itself depends only on the normal structure on $\mathcal{Z}$.

\begin{lemma}{\cite[Proposition 3.4]{donaldsondeformation2019}}\label{isobetweenBanachweight}
   The map $\tilde{\Delta}_g:\E^{k+2,\alpha}(M)\to \D^{k,\alpha}(M)$ is an isomorphism between Banach spaces, for any $k\ge 0$. As a result, $\Delta_g:W\E^{k+2,\alpha}(M)\to \D^{k,\alpha}(M)$ is an isomorphism, since $W\D^{k,\alpha}(M)=\D^{k,\alpha}(M)$.
\end{lemma}

For an odd function $\hat{f}$ on $\hat{M}$ that lifts from a $2$-valued function $f$ on $M$ with singular locus $\mathcal{Z}$, we say $\hat{f}$ is of class $\mathcal{D}^{k,\alpha}$ or $\mathcal{E}^{k,\alpha}$ when $f$ belongs to the corresponding function class. The spaces of global odd functions are correspondingly denoted by $\mathcal{D}^{k,\alpha}(\hat{M})$ and $\mathcal{E}^{k,\alpha}(\hat{M})$.

\subsection{$L^2$ harmonic $1$-forms}\label{L2harm1form}

Fix a submanifold $\mathcal{Z} \subset M$ of codimension 2. The branched covering $p:\hat{M} \to M$ induces a conical metric $\hat{g}$ on $\hat{M}$. The Hodge decomposition theorem holds for $(\hat{M}, \hat{g})$ \cite{Teleman,ShuguangWang}.

In particular, for any cohomology class $[h] \in H^1_{-}(\hat{M}; \mathbb{R}) := {[h] \in H^1(\hat{M}; \mathbb{R}) : \tau^*[h] = -[h]}$, there exists a unique 1-form $\hat{v}$ on $\hat{M} \setminus \mathcal{Z}$ satisfying:
\begin{enumerate}[label=(\roman*)]
\item $\|\hat{v}\|_{L^2} < +\infty$.
\item $\tau^*\hat{v} = -\hat{v}$.
\item $d\hat{v} = d\star_{\hat{g}}\hat{v} = 0$.
\item $[\hat{v}] = [h]$ in $H^1_{-}(\hat{M}; \mathbb{R}) \cong H^1_{-}(\hat{M}\setminus\mathcal{Z}; \mathbb{R})$.
\end{enumerate}

\begin{definition}
An \emph{odd $L^2$ harmonic 1-form} representing $[h]$ is a 1-form $\hat{v}$ on $\hat{M}$ satisfying conditions (i)--(iv) above.
\end{definition}

\begin{definition}
A section $v$ of $E \otimes T^*(M \setminus \mathcal{Z})$ is a \emph{2-valued harmonic 1-form with singular locus $\mathcal{Z}$} if:
\begin{enumerate}[label=(\roman*)]
\item $\int_M |v|^2 < +\infty$;
\item For any open ball $B\subset M\setminus\Zl$, $v=\pm w$, where $w$ is some harmonic $1$-form on $(B,g|_B)$.
\end{enumerate}
\end{definition}

There are canonical bijections between each two of: (i) 2-valued harmonic 1-forms on $(M, g)$ with singular locus $\mathcal{Z}$, (ii) odd $L^2$ harmonic 1-forms on $(\hat{M}, \hat{g})$, (iii) cohomology classes in $H^1_{-}(\hat{M}; \mathbb{R})$.

For brevity, we denote by $v(g, \mathcal{Z}, [h])$ a 2-valued harmonic 1-form in metric $g$ with singular locus $\mathcal{Z}$ whose lift $\hat{v}$ represents $[h] \in H^1_{-}(\hat{M}; \mathbb{R})$.


Let $U$ be a tubular neighborhood of $\mathcal{Z}$, identified via a diffeomorphism $\zeta: U \to U' \subset N$ with a neighborhood of the zero section in the normal bundle $N \to \mathcal{Z}$ (which carries a natural almost complex structure). Under this identification, each point $x \in U$ corresponds to a normal vector $\zeta(x) \in N$. 

If $\sigma$ is a section of the dual bundle $N^{-1}$, then $\langle\sigma,\zeta\rangle=:\sigma\zeta$ is a complex-valued function defined on $U$. Similarly, given a section $\sigma$ of the square root $N^{-1/2}$ of the dual line bundle $N^{-1}$, $\sigma\zeta^{1/2}$ is a $2$-valued complex function defined on $U$, equivalently, is a section of the complexified bundle $E^\CC|_{U}$. The same construction applies to sections $\sigma_p$ of $N^{-p}$ yielding the (multivalued) function $\sigma_p\zeta^{p}$, for any positive integer or half-integer $p$. For more details, see \cite[Section 3]{donaldsondeformation2019}.

A $2$-valued harmonic $1$-form is exact in the tubular neighborhood $U$ of $\Zl$, namely, $v|_{U}=df$ on $U$ for some $2$-valued harmonic function $f$ on $(U,g|_U)$. With the notions introduced in the previous subsection, and by Lemma \ref{admissiblesheaf}, we have $f\in W\E^{k,\alpha}_{\mathrm{loc}}(U)$ for any $k\ge 0$. 
\begin{lemma}{\cite[Proposition 2.4]{donaldsondeformation2019}}\label{expansionforfunctionlem}
    On a tubular neighborhood of $\Zl$, we have 
    \begin{align}\label{expansionforfunction}
        f(x)=\re\left(A(t)\zeta^{\frac{1}{2}}+B(t)\zeta^{\frac{3}{2}}\right)(x)+E(x).
    \end{align}
    Here, $A$ and $B$ are sections of $N^{-1/2}$ and $N^{-3/2}$ respectively. $E$ is the error term, satisfying $|E(p)|=\mathcal{O}(r^{5/2})$, where $r=\dist(x,\Zl)$.
\end{lemma}

As a result, any $2$-valued harmonic $1$-forms has an expansion near the singular locus $\Zl$:
\begin{lemma}\label{expansionlem}
    Consider a $2$-valued harmonic $1$-form $v=v(g,\Zl,[h])$, then on a tubular neighborhood of $\Zl$, we have 
    \begin{align}\label{expansion}
        v(x)=\re\left(d\left(A(t)\zeta^{\frac{1}{2}}+B(t)\zeta^{\frac{3}{2}}\right)(x)\right) +E_v(x).
    \end{align}
    Here, $A$ and $B$ are sections of $N^{-1/2}$ and $N^{-3/2}$ respectively. $E_v$ is the error term, satisfying $|E(x)|=\mathcal{O}(r^{3/2})$, where $r=\dist(x,\Zl)$.
\end{lemma}

\begin{definition}
    When $A(t)\equiv0$, the corresponding $2$-valued harmonic $1$-form $v$ is called a $\ZT$ harmonic $1$-form. If, in addition, $B(t)$ is nonvanishing everywhere on $\Zl$, then $v$ is called a \textbf{nondegenerate} $\ZT$ harmonic $1$-form.
\end{definition}

Recall that the sheaf $W\E^{k+2,\alpha}_{\mathrm{loc}}$ depends on the normal structure of $\Zl\subset M$, and hence so does the expansions \eqref{expansionforfunction} and \eqref{expansion}. The regularity of $A$- and $B$- terms may change as the metric varies.

Let's fix a singular locus $\Zl_0$ with a given normal structure and consider the Fr\'echet manifold $\M_0$ consisting of smooth Riemannian metrics $g$ on $M$, that induces the same normal structure on $\Zl_0$. Given a metric $g\in\M_0$ and a cohomology class $[h]\in H:=H^1_-(\hat{M};\RR)$, there exists a unique $2$-valued harmonic $1$-form $v=v(g,\Zl_0,[h])$. And $v$ has an expansion as \eqref{expansion}, although the metric varies in $\M_0$. 

With this setup, the coefficients $A(t)$ and $B(t)$ in \eqref{expansion}, can be regarded as maps: 
\begin{align}\label{ABunderline}
    \underline{A}:\M_0\times H\to C^\infty(\Zl_0, N^{-\frac{1}{2}}),\,\ \underline{B}:\M_0\times H\to C^\infty(\Zl_0, N^{-\frac{3}{2}}),
\end{align}
where $\underline{A}(g,[h])$ and $\underline{B}(g,[h])$ give the corresponding coefficient functions for the harmonic form $v(g,\mathcal{Z}_0,[h])$ at each metric $g \in \mathcal{M}_0$ and cohomology class $[h] \in H$.

\section{Perturb nondegenerate $\ZT$ harmonic $1$-forms}

In this section, we present the proof of Theorem \ref{main1}, which proceeds in several steps. 

In Section \ref{3.1}, we introduce an appropriate Green's function that will be essential for our analysis. This Green's function is then applied in Section \ref{3.2} to establish a transversality result (Theorem \ref{singulartrans}), showing that through a metric perturbation while preserving the singular locus $\mathcal{Z}$, we can ensure the corresponding $2$-valued harmonic 1-form has only discrete zeros on $M\setminus\mathcal{Z}$. The proof follows arguments similar to \cite{hondatrans}. However, this may introduce non-zero $A$-terms in the expansion \eqref{expansion}.

On the other hand, in Section \ref{3.3}, we recall Donaldson's deformation theorem (Theorem \ref{deformation}), which allows metric perturbations that maintain vanishing $A$-terms.  While one might attempt to directly combine this with our transversality result, a significant technical obstacle arises: Donaldson's theorem is formulated in the Fréchet category, where Sard's lemma and the parametric transversality theorem—key tools in the proof of Theorem~\ref{singulartrans}—fail to hold.

To overcome this issue, in Section~\ref{3.4} we employ a grafting argument.  
Our goal is to construct a nondegenerate $\mathbb{Z}/2$-harmonic $1$-form with isolated ordinary zeros. 
To achieve this, we cut the ``bad'' part of the nondegenerate $\ZT$ harmonic $1$-form (where ordinary zeros may accumulate), and then graft onto the remained part the ``good" part (which is away from the singular locus, having isolated zeros) of the $2$-valued harmonic $1$-form given in Section \ref{3.2}. 

This construction yields a new $2$-valued $1$-form together with a new metric on $M$ under which the form is harmonic. 
The resulting $2$-valued $1$-form has the desired properties, thus proving the density part of Theorem \ref{main1}.
Finally, we establish the openness part using Donaldson's deformation theorem.

\subsection{Green's function}\label{3.1}
Consider the flat model $(\CC\times\RR^{n-2},\hat{g}_{\mathrm{flat}})$. In \cite{kahlercone,donaldsondeformation2019}, the Green's function of the singular Laplacian $\Delta_{\hat{g}_{\mathrm{flat}}}$ has been constructed:

\begin{theorem}
    There exists a function $G(x,x')$ defined on $(\CC\times \RR^{n-2})^2$ away from the diagonal, satisfying the following properties:
    \begin{enumerate}[label=(\roman*)]
        \item For odd functions $\rho\in \E^{2,\alpha}_{\mathrm{loc}}$ such that $\Delta \rho$ is compactly supported:
        \begin{align}\label{greenflat}
            \rho(x)=\int_{\CC\times\RR^{n-2}}G(x,x')\Delta_{\hat{g}_{\mathrm{flat}}}\rho(x')dx'.
        \end{align}
        
        \item $G(x,x')=G(x',x)$.
        \item Fix $x$, $G(x,x')$ is an odd function that is in $\E^{k,\alpha}_{\mathrm{loc}}$ for any $k\ge 0$ as long as $x'\ne x$. And it satisfies $\Delta_{\hat{g}_{\mathrm{flat}}}G(x,x')=0$.
        \item For $x\in \CC^\ast\times\RR^{n-2}$, if $\dist(x,x')<\dist(x,\Zl)/2$, we have $G(x,x')=c_n\dist(x,x')^{2-n}+f_x(x')$, where $c_n=((n-2)\omega_{n-1})^{-1}$ (with $\omega_n$ is the volume of unit sphere $S^{n-1}$), and $f_x$ is a harmonic function defined locally near $x$. 
    \end{enumerate}
\end{theorem}

Using this model Green's function, we can construct a parametrix near the singular locus $\Zl\times\Zl\subset \hat{M}\times \hat{M}$. For the nonsingular part $(\hat{M}\setminus\Zl)\times(\hat{M}\setminus \Zl)$, the parametrix is defined as usual. Now glue them together by partition of unity, we then obtain a global parametrix defined near the diagonal in $\hat{M}\times\hat{M}$. Applying the standard iteration method as in \cite{Aubin}, we can finally construct a Green's function $G_g$ on the conical manifold $(\hat{M},\hat{g})$:

\begin{theorem}\label{greenM}
    There exists a function $G_g(x,x')$ defined on $\hat{M}\times\hat{M}$ away from the diagonal, satisfying the following properties:
    \begin{enumerate}[label=(\roman*)]
        \item For all odd functions $\rho\in W\E^{2,\alpha}(\hat{M},\Zl)$ such that $\Delta_{\hat{g}}\rho\in \D^{,\alpha}(\hat{M},\Zl)$:
        \begin{align}\label{greendelta}
            \rho(x)=\int_{\hat{M}}G_g(x,x')\Delta_{\hat{g}}\rho(x')dx'.
        \end{align}
        
        \item $G_g(x,x')=G_g(x',x)$.
        \item Fix $x$, $G_g(x,x')$ is an odd function that is in $\E^{k,\alpha}_{\mathrm{loc}}$ for any $k\ge 0$ as long as $x'\ne x$. And it satisfies $\Delta_{\hat{g},x'}G_g(x,x')=0$.
        \item For $x\in \hat{M}\setminus\Zl$ and $x'$ sufficiently close to $x$, then we have $G_g(x,x')=c_n\dist(x,x')^{2-n}+G_{g,1}(x,x')$, where $c_n=((n-2)\omega_{n-1})^{-1}$, and $G_{g,1}$ satisfying:
        \begin{align*}
            &|G_{g,1}(x,x')| \le C, && \text{if } n = 3, \\
            &|G_{g,1}(x,x')| = \mathcal{O}(|\log(\dist(x,x'))|), && \text{if } n = 4, \\
        &|G_{g,1}(x,x')| = \mathcal{O}(\dist(x,x')^{4-n}), && \text{if } n > 4.
        \end{align*}
        Here, $C$ is some constant.
        \item $G_{g}(\tau(x),x')=-G_{g}(x,x')$.
    \end{enumerate}
\end{theorem}

\subsection{A transversailty result}\label{3.2}
K. Honda proved a transversailty theorem for harmonic $1$-forms on closed manifolds in \cite{hondatrans}. In this subsection, we generalize his result to the case of $2$-valued harmonic $1$-forms. A key ingredient in \cite{hondatrans} is the following well-known infinite dimensional parametric transversality theorem: 
\begin{lemma}\label{translem}
    Let $X$ be a Banach manifold, and let $N_1$ and $N_2$ be finite dimensional manifolds (without boundary). Let $f:X\times M_1\to M_2$ be a $C^k$-map for $k$ sufficiently large. If $f$ is transverse to a submanifold $Z$ of $M_2$, then there is a dense subset $U$ of $X$, such that for all $x\in U$, $f_x=f(x,-):M_1\to M_2$ is transverse to $Z$.   
\end{lemma}

For an sufficiently large $\ell$, let $\mathcal{M}^{\ell,\alpha}$ denote the Banach manifold of $C^{\ell,\alpha}$ H\"older Riemannian metrics on manifold $M$. Fix the singular locus $\Zl_0\subset M$, a reference metric $g_0$ of $M$ and a class $[h]\in H=H^1_-(\hat{M};\RR)$. Let $\M^{\ell,\alpha}_0$ be the subset of $\M^{l,\alpha}$, consisting of those metrics that induce the same normal structure on $\Zl_0$ as the one induced by $g_0$.

By definition of the normal structure, fixing a normal structure is equivalent to imposing algebraic conditions on the $1$-jet of metrics, along $\Zl_0$. Therefore, the space $\mathcal{M}_0^{\ell,\alpha}$ forms a Banach submanifold of $\mathcal{M}^{\ell,\alpha}$.

We can define the following evaluation map:
\begin{align}
    \label{evm} ev:\mathcal{M}_0^{\ell,\alpha}\times(\hat{M}\setminus\Zl_0)\to T^\ast (\hat{M}\setminus\Zl_0): (g,x)\mapsto \hat{v}_g(x),
\end{align}
where $\hat{v}_g$ is the unique $L^2$ harmonic $1$-form representing $[h]$, with respect to the conical metric $\hat{g}$. Let $v_g=v_g(g,\Zl_0,[h])$ be the corresponding $2$-valued harmonic $1$-form on $M$ with singular locus $\Zl_0$.

Following \cite{hondatrans}, we will show that $ev$ is transverse to the zero section $Z=Z(\hat{M}\setminus\Zl_0)\subset T^\ast (\hat{M}\setminus\Zl_0)$. The next result then follows as a corollary of Lemma \ref{translem}.

\begin{proposition}\label{singulartrans}
    Fix the singular locus $\Zl_0$ and a class $[h]\in H^1_-(\hat{M};\RR)$. There is a dense open set $\mathcal{U}\subset\mathcal{M}_0^{\ell,\alpha}$, such that for any $g\in \mathcal{U}$, the ordinary zero set of the $2$-valued harmonic $1$-form $v_g=v_g(g,\Zl_0,[h])$, is a discrete subset of the open manifold $M\setminus\Zl_0$. 
\end{proposition}

Note that the ordinary zero set of $v_g$ may have limit points in $M$, and any such limit points must lie in $\Zl_0$.

\begin{definition}
    Let $x$ be an ordinary zero of a $2$-valued $1$-form $v$. In a small neighborhood $B$ of $x$, $v=\pm df$, where $f$ is a well-defined function on $B$. When there is such a neighborhood $B$ of $p$, that $f$ is a Morse function on it, we say that $x$ is of Morse-type. If all the ordinary zeros of $v$ are of Morse type, we say that  $v$ is \textbf{transverse}. In particular, the set of ordinary zeros of $v$ is a discrete subset of $M\setminus\Zl$, where $\Zl$ is the singular locus of $v$.
\end{definition}

Let's compute the derivative of $ev$. The singular locus $\Zl_0$, metric $g_0$ and odd cohomology class $[h]$ are fixed as above. Let $v_0=v_0(g_0,\Zl_0,[h])$ be the corresponding $2$-valued harmonic $1$-form, and $\hat{v}_0$ be its lift to $\hat{M}$. Consider a $C^{\ell,\alpha}$ symmetric tensor field that is compactly supported on $M\setminus\Zl_0$, say $\gamma\in C_c^{\ell,\alpha}(\mathrm{Sym}^2T^\ast (M\setminus\Zl_0))$. Then $g_t=g_0+t\gamma$ is a smooth path in $\M_0^{\ell,\alpha}$ with velocity $\gamma$. Denote by $v_t=v_t(g_t,\Zl_0,[h])$ the corresponding path of $2$-valued harmonic $1$-forms, and $\hat{v}_t$ the corresponding path of $L^2$ harmonic $1$-forms. Then we have $v_t-v_0=df_t$, where $f_t$ is a path of $2$-valued functions on $M$ with singular locus $\Zl_0$. Differentiating the equation $\Delta_{g_t}v_t=0$ with respect to $t$ at $t=0$, we have:
\begin{align}
    0=\Delta_{g_0}\frac{d}{dt}\Big|_{t=0}df_t+\left(\frac{d}{dt}\Big|_{t=0}\Delta_{g_t}\right)v_0.
\end{align}
Since $dv_0=0$, it follows that $\left(\frac{d}{dt}|_{t=0}\Delta_{g_t}\right)v_0=d(\frac{d}{dt}|_{t=0}\delta_{g_t})v_0$, where $\delta_{g_t}$ is the formal adjoint of $d$ with respect to the metric $g_t$. 

Thus we have 
    $\Delta_{g_0}\dot{f}+\left(\frac{d}{dt}|_{t=0}\delta_{g_0+t\gamma}\right)v_0=0$,
where $\dot{f}=\frac{d}{dt}|_{t=0}f_t$. Let $\hat{f}$ denote the odd function on $\hat{M}$ corresponding to $\dot{f}$. It follows that $\hat{f}$ solves:
\begin{align}\label{tderiv}
    \Delta_{\hat{g}_0}\hat{f}+\left(\frac{d}{dt}\Big|_{t=0}\delta_{g_0+t\gamma}\right)\hat{v}_0=0.
\end{align}
The partial derivative of $ev$ is then given by $\partial_g ev(g_0,x)(\gamma,0)=\frac{d}{dt}|_{t=0}\hat{v}_t(x)=d\hat{f}(x)$. 

\begin{lemma}\label{partialderivativeofevissurj}
    Let $x\in\hat{M}\setminus\Zl_0$ be a zero of $\hat{v}_0$. Then the partial derivative $\partial_gev(g_0,x)(-,0):T_{g_0}\M^{\ell,\alpha}_0\times\{0\}\to T^\ast_x(\hat{M}\setminus\Zl_0)$ is surjective. 
\end{lemma}
\begin{proof}
    Since $C_c^{\ell,\alpha}(\mathrm{Sym}^2T^\ast M)\subset T_{g_0}\M^{\ell,\alpha}_0$, it is sufficient to show that $\partial_gev(g_0,x)(-,0):C_c^{\ell,\alpha}(\mathrm{Sym}^2T^\ast M)\to T^\ast_x(\hat{M}\setminus\Zl_0)$ is surjective.
    
    Fix $\gamma\in C_c^{\ell,\alpha}(\mathrm{Sym}^2T^\ast M)$, if $\hat{f}$ solves \eqref{tderiv}, then by $(i)$ of Theorem \ref{greenM}, we have \begin{align}
        \hat{f}(x)&=\int_{\hat{M}}G_{g_0}(x,x')\left(\frac{d}{dt}\Big|_{t=0}\star_{g_t}d\star_{g_t}\right)\hat{v}_0(x')dx' \nonumber\\
                  &=\int_{\hat{M}}G_{g_0}(x,x')\left(\star_{g_0}d\frac{d}{dt}\Big|_{t=0}\star_{g_t}\right)\hat{v}_0(x')dx' \nonumber\\
                  &=\int_{\hat{M}}\langle d_{x'}G_{g_0}(x,x'),\star_{g_0}(D_\gamma\star_{g_0}) \hat{v}_0(x')\rangle_{x'}dx',
    \end{align}
    where $D_\gamma\star_{g_0}=\frac{d}{dt}|_{t=0}\star_{g_0+t\gamma}$. Thus $d\hat{f}(x)=\int_{\hat{M}}\langle d_xd_{x'}G_{g_0}(x,x'),\star(D_\gamma\star) \hat{v}_0(x')\rangle_{x'}dx'$. By literally repeating the computation in \cite[Proposition 2.13]{hondatrans}, one can show that the map \[i_{\hat{v}_0(y)}:\mathrm{Sym}^2T_{p(y)}^\ast{M}\to T^\ast_y(\hat{M}\setminus\Zl_0):\gamma(p(y))\mapsto\star_{g_0} (D_\gamma\star_{g_0})\hat{v}_0(y),\] is surjective for any $y\in \hat{M}\setminus\Zl_0$ such that $\hat{v}_0(y)\ne 0$. Here, $p(y)\in M\setminus\Zl_0$. 
    
    The remaining steps follow from $(iv)$ of Theorem \ref{greenM}, \cite[Proposition 2.10]{hondatrans}, and the proof of \cite[Theorem 2.19]{hondatrans}. We choose a sequence $y_i\to x$ such that $\hat{v}(y_i)\ne 0$, and then study the limit of the image of partial derivative $\partial_gev(g_0,x)(-,0)$ along this sequence. Finally, we can show that $\partial_gev(g_0,x)(-,0)|_{C_c^{\ell,\alpha}(\mathrm{Sym}^2T^\ast M)}$ is surjective.
\end{proof}

\begin{corollary}
    The map $ev$ defined in \eqref{evm} is transverse to the zero section $Z(\hat{M}\setminus\Zl_0)\subset T^\ast(\hat{M}\setminus\Zl_0)$.
\end{corollary}

This concludes Theorem \ref{singulartrans}.

\subsection{Tame estimates and the deformation theorem}\label{3.3}
Let $\M$ be the Fr\'echet manifold consisting of smooth Riemannian metrics on $M$. Fix a metric $g_0\in \M$, the singular locus $\Zl_0$ and a cohomology class $[h]$. Let $\M_0$ be the Fr\'echet submanifold of $\M$, consisting of metrics that induces on $\Zl_0$ the same normal structure as that induced by $g_0$.

We first recall the fundamental notion of tame estimates from \cite{Hamilton}.
\begin{definition}\label{defoftame}
    Consider graded Fr\'echet spaces \footnote{Indeed, the spaces involved in this article are all assumed to be “tame”; however, this notion will not play a role in our discussion. For details, see \cite[II.1.3]{Hamilton}.} $\mathcal{F}_1$ and $\mathcal{F}_2$. For simplicity, denote the norms on both spaces by the same notion $\|-\|_l$ for $l\ge 0$. For $f\in\mathcal{F}_i$, $\|f\|_0\le \|f\|_1\le \|f\|_2\le\cdots$

    Let $S:U\subset\mathcal{F}_1\to\mathcal{F}_2$ be a map between Fr\'echet spaces. We say that $S$ satisfies \textbf{tame estimates} if there exist constants $l_0 \ge 0$ and $k \ge 0$ 
    such that for every $l \ge l_0$, there is a constant $C_l$ (depending on $l$) satisfying \[\|S(f)\|_l\le C_l(1+\|f\|_{l+k}).\]

    A map $S$ is called \textbf{smooth tame}, if it is smooth and all its derivatives satisfy tame estimates.
\end{definition}

There is a graded sequence of norms on the Fr\'echet space $C^\infty(\mathrm{Sym}^2T^\ast M)$, namely the H\"older norms. For any $\gamma\in C^\infty(\mathrm{Sym}^2T^\ast M)$, we define $\|\gamma\|_l=\sum_{i=0}^l\|\gamma\|_{C^{i,\alpha}}$. One can define a metric $d$ on $C^\infty(\mathrm{Sym}^2T^\ast M)$ by \[d(\gamma_1,\gamma_2)=\sum_{i=0}^\infty 2^{-i}\frac{\|\gamma_1-\gamma_2\|_i}{\|\gamma_1-\gamma_2\|_i+1},\]
this metric induces a topology on $C^\infty(\mathrm{Sym}^2T^\ast M)$, and hence on Fr\'echet manifolds $\M_0\subset\M\subset C^\infty(\mathrm{Sym}^2T^\ast M)$. 

A neighborhood of $g_0$ on $\M$, can be chosen to be an open ball $B_{\varepsilon}(0):=\{g\in\M: d(g-g_0,0)<\varepsilon\}$ for some small $\varepsilon>0$. Similar for $\M_0$.

Since $2^{-i}\to 0$ as $i\to \infty$, there exists constants $l_0>0$ and $\delta<0$, such that for any $l\ge l_0$, if $\gamma\in C^\infty(\mathrm{Sym}^2T^\ast M)$ satisfies $\|\gamma\|_l<\delta$, then $\gamma\in B_{\varepsilon}(0)$. Conversely, fix $l>0$, for any $\epsilon>0$, we can choose $\varepsilon$ sufficiently small, such that for any $\gamma\in B_{\varepsilon}(0)$, we have $\|\gamma\|_{l}<\epsilon$. 

\vspace{2ex}

In the following discussion, we will make use of the Fr\'echet spaces \begin{align*}\E^{\infty,\alpha}(M)=\cap_{l\ge 2}\E^{l,\alpha}(M),\text{ and }C^{\infty}(N^{j})=\cap_{l\ge 0} C^{l,\alpha+\frac12}(N^j),\end{align*} where $N$ is the normal bundle of $\Zl_0$ equipped with Euclidean structure induced by $g_0$, and $j=1,-\frac{1}{2},-\frac{3}{2}$. The graded structures on $\E^{\infty,\alpha}(M)$ is given by the norms $\|-\|_{\E^{l,\alpha}(M)}$ defined before Lemma \ref{isobetweenBanachweight}. And the graded structure on spaces $C^{\infty}(N^{j})$ are given by the corresponding H\"older norms in the obvious way. 

\vspace{2ex}
We now aim to control the variation of $2$-valued $1$-forms as the metric varies; specifically, we seek tame estimates for these forms, their first two leading coefficients, and the associated error terms. 
Such estimates are established in~\cite[Section~4]{donaldsondeformation2019}.

Suppose $g_1=g_0+\gamma\in \M_0$, where $\gamma$ is a tensor field in $C^\infty(\mathrm{Sym}^2T^\ast M)$. Let $v_0=v_0(g_0,\Zl_0,[h])$ and $v_1=v_1(g_1,\Zl_0,[h])$. Then $v_1-v_0=df$ for some $2$-valued function $f$ with singular locus $\Zl_0$. In a tubular neighborhood $U$ of $\Zl_0$, both $v_0$ and $v_1$ are exact, hence we may write $v_i=df_i$ on $U$. Here, each $f_i$ is a $2$-valued harmonic function with respect to the metric $g_i$ for $i=0,1$, and both of them has singular locus $\Zl_0$. 

As a result, $f=f_1-f_0\in W\E^{l,\alpha}(M)$ for any $l\ge 2$. On the other hand, $\Delta_{g_0}f=\delta_{g_0}v_1$ and $\Delta_{g_1}f=-\delta_{g_1}v_0$, and the right-hand sides are sections in $\D^{l,\alpha}(M)$ for all $l\ge 0$. By definition of $\|-\|_{\E^{l+2,\alpha}}$, we have $\|W^{-1}f\|_{\E^{l+2,\alpha}}=\|W^{-1}\delta_{g_1}v_0\|_{\D^{l,\alpha}}$. The following tame estimates was proved in \cite[Section 4]{donaldsondeformation2019}:

\begin{lemma}\label{tameestforform}
     Assume $g_1$ lies in a sufficiently small neighborhood $\mathcal{U}_1$ of $g_0$ in $\M_0$. There exists a constant $k>0$ independent of $g_1$, such that for each $l\ge 0$, there is a constant $C_l$ satisfying $\|W^{-1}f\|_{\E^{l+2,\alpha}}\le C_l(1+\|\gamma\|_{l+k})$, where $\gamma=g_1-g_0$ and $\|-\|_{l+k}$ denotes the graded norm defined by H\"older norm on the Fr\'echet space of smooth symmetric tensor fields.
\end{lemma}

As a corollary of Lemma \ref{tameestforform} and \cite[Theorem II.3.1.1]{Hamilton}, one have:
\begin{lemma}\label{reallytame}
    Fix the reference metirc $g_0$, the map $\M_0\to \E^{\infty,\alpha}(M):g_1\mapsto W^{-1}f$ is smooth tame.
\end{lemma}

This yields:
\begin{lemma}{\cite[Section 4]{donaldsondeformation2019}}\label{coefficientsaretame}
The maps $\underline{A}:\M_0\times H\to C^\infty(N^{-1/2})$ and $\underline{B}:\M_0\times H\to C^\infty(N^{-3/2})$ are smooth and tame.
\end{lemma}

By Lemma \ref{expansionforfunctionlem}, the function $f$ has an expansion:
\[f(x)=\re\left(\tilde{A}(t)\zeta^{\frac{1}{2}}+\tilde{B}(t)\zeta^{\frac{3}{2}}\right)+\tilde{E}(x),\]
where $|\tilde{E}(x)|\le C\dist(x,\Zl)^{\frac{5}{2}}$. Thus $\underline{A}(g_1,[h])-\underline{A}(g_0,[h])=\tilde{A}(t)$ and $\underline{B}(g_1,[h])-\underline{B}(g_0,[h])=\tilde{B}(t)$. If we denote the error terms in \eqref{expansion} by $E_i$ for $v_i$ respectively, we then have $E_1-E_0=d\tilde{E}$.

The following estimates can be derived from Lemma \ref{coefficientsaretame} and \cite[Lemma III.1.2.4]{Hamilton}.

\begin{lemma}
There exists a constant $k>0$, such that for each $g_1\in\mathcal{U}_1$ and $l\ge 0$, there is a constant $C_l$ satisfying:
\begin{align}\label{tameestforAB}
    &\|\underline{A}(g_1,[h])-\underline{A}(g_0,[h])\|_{C^{l,\alpha+\frac{1}{2}}}+\|\underline{B}(g_1,[h])-\underline{B}(g_0,[h])\|_{C^{l,\alpha+\frac{1}{2}}}\le C_l(\|g_1-g_0\|_{l+k}).
\end{align}
\end{lemma}
Moreover, one can obtain an estimate for the error term, by combining Lemma \ref{reallytame} with the proof of $(2)$ in \cite[Proposition 2.4]{donaldsondeformation2019}:
\begin{lemma}
    There exists a $k>0$, such that in a tubular neighborhood $U$ of $\Zl$, we have 
\begin{align}\label{EstforEterm}
        |E_1(x)-E_0(x)|\le C\dist(x,\Zl_0)^{\frac{3}{2}}\|g_1-g_0\|_{k},
\end{align}
for any $g_1\in\mathcal{U}$ and some constant $C$.
\end{lemma}
We emphasize that all the norms for $A$- and $B$-terms and error terms here are induced by the reference metric $g_0$.

\vspace{2ex}
We now begin to recall Donaldson's deformation theorem for nondegenerate $\ZT$ harmonic $1$-forms. Let $\mathcal{S}$ be the Fr\'echet manifold that consists of smooth codimension $2$ submanifolds in $M$. Given the pair $(g_0,\Zl_0)\in \M\times\mathcal{S}$ as above, we have:

\begin{lemma}{\cite[Proposition 4.2]{donaldsondeformation2019}}\label{tamenessofDiff}
    There exists a neighborhood $\mathcal{U}\subset\M\times\mathcal{S}$ of $(g_0,\Zl_0)$ such that we can find a smooth \textbf{tame} map $\Psi:\mathcal{U}\to \mathrm{Diff}(M)$, where $\mathrm{Diff}(M)$ is the Fr\'echet Lie group of smooth diffeomorphism of $M$. Th map $\Psi$ has the property that, for each pair $(g,\Zl)\in \M$, the diffeomorphism $\Psi(g,\Zl)$ sends $\Zl_0$ to $\Zl$, and carries the normal structure of $\Zl_0$ under $g_0$ to that of $\Zl$ under $g$.
\end{lemma}

Note that $\Psi(g,\Zl)^\ast g\in\M_0$. By Lemma \ref{coefficientsaretame}, we can define smooth tame maps between Fr\'echet manifolds as follows:
\begin{definition}
    \begin{align}
        &\bA:\mathcal{U}\times H\to C^\infty(\Zl_0,N^{-\frac{1}{2}}):(g,\Zl,[h])\mapsto \underline{A}\left(\Psi(g,\Zl)^\ast g,[h]\right),\\
        &\bB:\mathcal{U}\times H\to C^\infty(\Zl_0,N^{-\frac{3}{2}}):(g,\Zl,[h])\mapsto \underline{B}\left(\Psi(g,\Zl)^\ast g,[h]\right),
    \end{align}
    where $H=H^1_-(\hat{M};\RR)$ and $\underline{A}$ and $\underline{B}$ are defined as in \eqref{ABunderline}.
\end{definition}

The following is Donaldson's deformation theorem:
\begin{theorem}{\cite[Theorem 1.1]{donaldsondeformation2019}}\label{deformation}
    Given a smooth metric $g_0$, a singular locus $\Zl_0$ and cohomology class $[h_0]$, such that $v_0=v_0(g_0,\Zl_0,[h_0])$ is a nondegenerate $\ZT$ harmonic $1$-form. There is a neighborhood $\mathcal{U}_1$ of $(g_0,[h_0])$ in $\mathcal{M}\times H$, and a neighborhood $\mathcal{U}_2$ of $\Zl_0$ in $\mathcal{S}$, such that for any $(g,[h])\in\mathcal{U}_1$, there exists a unique $\Zl\in\mathcal{U}_2$ such that the corresponding $2$-valued harmonic $1$-form $v=v(g,\Zl,[h])$ is a nondegenerate $\ZT$ harmonic $1$-form, i.e. $\bA(\Zl,g,[h])\equiv 0$ and $\bB(\Zl,g,[h])$ does not vanish everywhere. Moreover, the map $\underline{\Zl}:\mathcal{U}_1\to\mathcal{U}_2:(g,[h])\mapsto \Zl$ is smooth and tame.
\end{theorem}
The neighborhood $\mathcal{U}$ of $(g_0,\Zl_0)$ can be chosen to be identified with a product $\mathcal{U}_1\times \mathcal{U}_2$, where $\mathcal{U}_1$ is a neighborhood of $0$ in the Fr\'echet space of smooth tensor fields $C^\infty(\mathrm{Sym}^2(T^\ast M))$, and $\mathcal{U}_2$ is a neighborhood of $0$ in the Fr\'echet space $C^\infty(N)$ of smooth section of the normal bundle $N\to\Zl_0$, which equipped with Euclidean structure induced by $g_0$. Moreover, $\mathcal{U}_1$ and $\mathcal{U}_2$ can be chosen as metric open balls with respect to the natural metrics on the corresponding Fr\'echet manifolds.

\subsection{Grafting of $2$-valued harmonic $1$-forms}\label{3.4}

To fix notations, we set:
\begin{itemize}
    \item $U_{\epsilon_0,\epsilon_1}:=\{x\in M:\epsilon_0<\dist_{g_0}(x,\Zl_0)<\epsilon_1\}$, where $\epsilon_1>\epsilon_0\ge 0$.
    \item $U_{\epsilon}:=\{x\in M:\dist_{g_0}(x,\Zl_0)<\epsilon\}$, where $\epsilon>0$.
    \item $\hat{U}_{\epsilon_0,\epsilon_1}$ and $\hat{U}_{\epsilon}$ are preimages of $U_{\epsilon_0,\epsilon_1}$ and $U_{\epsilon}$ in $\hat{U}$, respectively.
\end{itemize}

\begin{lemma}\label{forbiddenband}
    Assume $v_0=v_0(g_0,\Zl_0,[h])$ is a nondegenerate $\ZT$ harmonic $1$-form. In a sufficiently small neighborhood $\mathcal{U}_1$ of $g_0$ in $\M_0$, there are constants $0<\epsilon_0<\epsilon_1$ depending only on $\mathcal{U}_1$, such that for any $g\in\mathcal{U}_1$, the corresponding $2$-valued harmonic $1$-form $v=v(g,\Zl_0,[h])$ has a positive lower bound $c_0$ on $U_{\epsilon_0,\epsilon_1}$.
\end{lemma}
\begin{proof}
    By Lemma \ref{expansionlem}, for $\epsilon>0$ sufficiently small, the $2$-valued harmonic $1$-form $v$ admits an expansion in a tubular neighborhood $U_{\epsilon}$ of $\Zl_0$:
    \[v(x)=d\re\left(A(t)\zeta^{\frac{1}{2}}+B(t)\zeta^{\frac{3}{2}}\right)+E(x),\]
    where $A=\underline{A}(g,[h])$ and similarly for $B$. For simplicity we set $r=\dist_{g_0}(x,\Zl_0)$. It follows that
    \begin{align}\label{abslower}
        |v(x)|\ge&\left|\re \left(B(t)\zeta^{\frac{1}{2}}d\zeta\right)\right|-\left(\left|\re\left)A(t)\zeta^{-\frac{1}{2}}d\zeta\right)\right|+\left|\re\left(dA(t)\zeta^{\frac{1}{2}}\right)\right|\right)\\
        &-\left(\left|dB(t)\zeta^{\frac{3}{2}}\right|+|E(x)|\right)\nonumber\\
        \ge&|B(t)|r^{\frac{1}{2}}-\left(r^{-\frac{1}{2}}+r^{\frac{1}{2}}\right)\|\underline{A}(g,[h])\|_{C^{1}}\nonumber\\
        &-r^{\frac{3}{2}}\left(\|\underline{B}(g,[h])\|_{C^{1}}+C\|g-g_0\|_k\right)-|E_0(x)|.\nonumber
    \end{align}

    We now aim to find a neighborhood $\mathcal{U}_1$ of $g_0$, and constants $\epsilon_0$ and $\epsilon_1$, such that for any $g\in\mathcal{U}_1$ and $x\in U_{\epsilon_0,\epsilon_1}$, the right-hand side of the last inequality in \eqref{abslower} has a positive lower bound $c_0$. This reduces to choosing these data suitable so that, for $\epsilon_0<r<\epsilon_1$, the following two inequalities hold:
    \begin{align}
        &\frac{1}{2}|B(t)|r^{\frac{1}{2}}>\left(r^{-\frac{1}{2}}+r^{\frac{1}{2}}\right)\|\underline{A}(g,[h])\|_{C^{1}}+c_0,\label{epsilon1}\\
        &\frac{1}{2}|B(t)|r^{\frac{1}{2}}>r^{\frac{3}{2}}\left(\|\underline{B}(g,[h])\|_{C^{1}}+C\|g-g_0\|_k\right)-|E_0(p)|.\label{epsilon2}
    \end{align}
    By assumption, $B_0=\underline{B}(g_0,[h])$ is nonvanishing, hence there exists a constant $c > 0$ such that $|B_0(t)|>c$ for all $t\in\Zl_0$.
    From the estimate \eqref{tameestforAB}, we can choose a metric ball $\mathcal{U}_1$ centered at $g_0$ with radius $\varepsilon > 0$ (with respect to the natural metric on the Fréchet manifold $\M_0$, as discussed after Definition~\ref{defoftame}), such that $|B(t)| > \frac{c}{2} > 0$ for all $g \in \mathcal{U}_1$.
    
    On the other hand, since $A_0=\underline{A}(g_0,[h])\equiv 0$, hence the radius $\varepsilon$ of $\mathcal{U}_1$ can be chosen so that $\|\underline{A}(g, [h])\|_{C^1} \ll \frac{c}{16}$ for all $g \in \mathcal{U}_1$. Moreover, by setting $\epsilon_0\ge 8C_1\varepsilon/c(1-C_1\varepsilon)>0$, where $C_1$ is the constant in \eqref{tameestforAB} and and assuming $C_1 \varepsilon$ is sufficiently small, we obtain $\frac{1}{8}|B(t)|r^{1/2}>r^{-1/2}\|\underline{A}(g,[h])\|_{C^1}$. With these choices, 
    \[\frac{1}{2}|B(t)|r^{\frac{1}{2}}-\left(r^{-\frac{1}{2}}+r^{\frac{1}{2}}\right)\|\underline{A}(g,[h])\|_{C^{1}}>\frac{1}{8}cr^{\frac{1}{2}},\] hence we can choose some $0<c_0<c\epsilon_0^{1/2}/2$.

    Now the \eqref{epsilon2} determines the constant $\epsilon_1$. Indeed, we require 
    \[r<\frac{1}{4}|B(t)|/\left(\|\underline{B}(g,[h])\|_{C^{1}}+C\|g-g_0\|_k\right),\,~r^{\frac{1}{2}}>4|E_0(x)|/|B(t)|.\]
    For the first inequality, it suffices to take $\epsilon_1<\frac{1}{4}c/(C_1\varepsilon/(1-C_1\varepsilon)+\|B_0\|_{C^1}+C\varepsilon)$. Since $|E_0(x)|=\mathcal{O}(r^{3/2})$, we may also choose $\epsilon_1$ appropriately to satisfy the second inequality, and such that $\epsilon_1>\epsilon_0>0$. 
\end{proof}

\begin{lemma}\label{transdegzerodiscrete}
    If $v_0=v_0(g_0,\Zl_0,[h])$ is a transverse and nondegenerate $\ZT$ harmonic $1$-form, then the ordinary zero set is discrete in $M$.  
\end{lemma}
\begin{proof}
    In the proof of Lemma \ref{forbiddenband}, we note that the constant $\epsilon_0$ depends on the $A$-term. In particular, when $v_0$ is a nondegenerate $\ZT$ harmonic $1$-form, we can prove that $v_0(g_0,\Zl_0,[h])$ is nonvanishing on $U_{0,\epsilon_1}$. Recall that the ordinary zero set of $v_0$ is discrete on $M\setminus \Zl_0$, and any limit point must lie on $\Zl_0$. But since $v_0$ has no ordinary zeros near $\Zl_0$, there can be no limit points of ordinary zeros. Hence, the zero set is discrete in $M$.
\end{proof}

The classical Schauder estimates imply the following:
\begin{lemma}\label{Holderconvergeonforbiddenband}
    Fix $\epsilon_1 > \epsilon_0 > 0$. For any $l \ge 2$, if a sequence of $2$-valued functions $\{f_i\}_{i}$ with singular locus $\Zl_0$ satisfies
$\|f_i\|_{\E^{l,\alpha}} \to 0$, then
\[
    \|f_i\|_{C^{l,\alpha}(U_{\epsilon_0,\epsilon_1})} \to 0.
\]
Here, \[\|f_i\|_{C^{l,\alpha}(U_{\epsilon_0,\epsilon_1})}:=\|\hat{f}_i\|_{C^{l,\alpha}(\hat{U}_{\epsilon_0,\epsilon_1})},\] where $\hat{f}_i$ is the lift of $f_i$ to $\hat{M}$, and $C^{l,\alpha}$ denotes the standard H\"older norm induced by $g_0$.
\end{lemma}

Now we have all the ingredients needed to construct transverse and nondegenerate $\ZT$ harmonic $1$-forms. 

Fix a pair $(g_0,\Zl_0)$, such that $v_0=v_0(g_0,\Zl_0,[h])$ is nondegenerate. We also fix the cohomology class $[h]$ throughout the rest of this subsection.

\vspace{1ex}
We begin the grafting process:

\begin{large}
    \textbf{Step 1.}
\end{large}
 By Proposition \ref{singulartrans}, for any $\ell\in\NN$ sufficiently large and any small neighborhood $V$ of $g_0$ in $\M_0^{\ell,\alpha}$, we can find a metric $g_1$ in $V$, such that $v_1(g_1,\Zl_0,[h])$ is a transverse $2$-valued $1$-form. Since smooth metric are dense in $\M_0^{\ell,\alpha}$, we may assume $g_1$ is smooth, i.e. $g_1\in\M_0$. 

Moreover, by the discussion after Definition \ref{defoftame}, $g_1$ can be chosen to lie in any prescribed neighborhood $\mathcal{U}$ of $g_0$ in $\M_0$, provided $\ell$ is large enough. Thus we have:
\begin{lemma}\label{transversecloseenough}
    For any $\delta_1>0$, there exists a metric $g_1\in \M_0$, such that $v_1=v_1(g_1,\Zl_0,[h])$ is transverse and $d(g_1-g_0,0)<\delta_1$,  where $d$ is the metric on the Fr\'echet space $C^{\infty}(Sym^2T^\ast M)$ defined as in the discussion after Definition \ref{defoftame}. 
\end{lemma}

As shown in Lemma \ref{forbiddenband} and \ref{transdegzerodiscrete}, we have the following result:
\begin{lemma}\label{gluingband}
    For the $2$-valued harmonic $1$-forms $v_1=v_1(g_1,\Zl_0,[h])$, there exists two constants $\epsilon_1>\epsilon_0>0$, such that $|v_1(x)|\ge c_0$, for all points $x$ in the region $U_{\epsilon_0,\epsilon_1}$. This region will be referred to as the \textbf{cambium}. The constants $\epsilon_0$ and $\epsilon_1$ are uniform with respect to $\delta_1$, or equivalently, independent of $g_1$ that sufficiently close to $g_0$ in $\M_0$. Moreover, the original nondegenerate $\ZT$ harmonic $1$-form $v_0(g_0,\Zl_0,[h])$ is nonvanishing on $U_{0,\epsilon_1}$.
\end{lemma}


\begin{large}
    \textbf{Step 2.}
\end{large}
We graft $v_1(g_1,\Zl_0,[h])$, which is transverse but not necessarily a $\ZT$ harmonic $1$-form, onto $v_0(g_0,\Zl_0,[h])$, which is nondegenerate but not necessarily transverse, along the cambium $U_{\epsilon_0,\epsilon_1}$, to obtain a new $2$-valued $1$-form $v_2$ on $M$, with the same singular locus $\Zl_0$. 

Consider a smooth cut-off function $\chi\ge 0$ on $M$, such that $\chi(x)=1$ if $\dist_{g_0}(x,\Zl_0)\le \epsilon_0$, $\chi(x)=0$ if $\dist_{g_0}(x,\Zl_0)\ge \epsilon_1$, and $|d\chi|_{g_0}\le 1/(\epsilon_1-\epsilon_0)$ on $M$.

In a tubular neighborhood $U$ of $\Zl_0$, $v_i(g_i,\Zl_0,[h])~(i=0,1)$ are exact, so we can write $v_i=df_i$ on $U$, where each $f_i$ is a $2$-valued function with singular locus $\Zl_0$, harmonic with respect to $g_i$ on $U$. Suppose $U_{\epsilon_0,\epsilon_1}\subset U$, we define
\begin{align}
    v_2(x)=\begin{cases}
        d\left(\chi f_0+(1-\chi)f_1\right)(x),& x\in U_{\epsilon_1},\\
        v_1(x),&x\notin U_{\epsilon_1}.
    \end{cases}
\end{align}

\begin{lemma}\label{nonvanishing1form}
    With appropriate choices of $g_1$, $v_2$ is nonvanishing on $U_{0,\epsilon_1}$.
\end{lemma}
\begin{proof}
    On $U$, we have $v_2=v_0+(1-\chi)(v_1-v_0)+(f_0-f_1)d\chi$. When $x\in U_{\epsilon_0,\epsilon_1}$, \begin{align}\label{v3lower}
        |v_2(x)|\ge& |v_0(x)|-|v_0-v_1|(x)-\frac{1}{\epsilon_1-\epsilon_0}|f_0-f_1|(x).
    \end{align}
    Recall that $v_0-v_1=df$ and $f_0-f_1=f$ for some $2$-valued function $f\in W\E^{\infty,\alpha}$, as discussed before Lemma \ref{tameestforform}. Since $|v_0(x)|\ge c_0$ on $U_{\epsilon_0,\epsilon_1}$, by Lemma \ref{reallytame} and \ref{Holderconvergeonforbiddenband}, if $g_1$ is sufficiently close to $g_0$, then $|df|$ and $|f|$ can be made arbitrarily small, so that the right-hand side of \eqref{v3lower} remains positive.

    On the other hand, when $0<\dist_{g_0}(x,\Zl_0)\le \epsilon_0$, by definition of $\chi$, we have $v_2=v_0$. By Lemma \ref{gluingband}, $v_2$ has no zeros in this region.
\end{proof}

\begin{large}
    \textbf{Step 3.}
\end{large}
Next, we glue the $(n-1)$-forms $\star_{g_0}v_0$ and $\star_{g_1}v_1$ together, along the cambium $U_{\epsilon_0,\epsilon_1}$, to obtain a candidate for $\star_{g_2} v_2$, where $g_2$ is a new metric to be determined. 

Consider the tubular neighborhood $U$ of $\Zl_0$. There is a $2$-fold branched covering $p:\hat{U}\to U$, branched along $\Zl_0$. The pullbacks $p^\ast\star_{g_i}v_i$ are well-defined closed $(n-1)$-forms on $\hat{U}$. On the other hand, we have a homotopic equivalence $\hat{U}\simeq \Zl_0$, hence $H^{n-1}(\hat{U};\RR)\cong H^{n-1}(\Zl_0;\RR)\cong 0$. Consequently, the forms $p^\ast\star_{g_i}v_i=\star_{\hat{g}_i}\hat{v}_i$ are exact. Let $w_i$ be the $(n-2)$-form on $\hat{U}$, such that $\star_{\hat{g}_i}\hat{v}_i=dw_i$ for $i=0,1$. 

For some estimates required in Lemma~\ref{wedgepositiveonbandlem}, we need an explicit expression for $w_i$ on the region $U_{\epsilon_0,\epsilon_1}$. Thus, we proceed to construct $w_i$ directly from $\star_{\hat{g}_i}\hat{v}_i$ and give its local coordinate representation. 

We choose local coordinates as follows. The region $\hat{U}_{\epsilon_0,\epsilon_1}$ is diffeomorphic to $SN^{1/2}\times(\epsilon_0,\epsilon_1)$, where $SN^{1/2}$ is the unit circle bundle of the normal bundle $N^{1/2}$ of $\Zl_0$ in $\hat{U}$. There is a natural involution $\tau$ on $\hat{U}_{\epsilon_0,\epsilon_1}$, which induces an orientation-preserving involution on $SN^{1/2}$. 

Let $x=(x_s)_{1\le s\le n-2}$ be local coordinates on $\Zl_0$, $\theta$ be the coordinate of the fiber of $SN^{1/2}\to\Zl_0$, and $r$ be the radial coordinate on $(\epsilon_0,\epsilon_1)$. Then $(x,\theta,r)$ forms a system of local coordinates on $U_{\epsilon_0,\epsilon_1}$.

Recall that $\hat{v}_i=d\hat{f}_i$, where $\hat{f}_i$ is the lift of $f_i$ to $\hat{U}$. In the chosen coordinates, we have $\hat{v}_i(x,\theta,r)=\frac{\partial \hat{f}_i}{\partial x_s}(x,\theta,r)dx_s+\frac{\partial \hat{f}_i}{\partial \theta}(x,\theta,r)d\theta+\frac{\partial \hat{f}_i}{\partial r}(x,\theta,r)dr$.

Then 
\begin{align}
    \star_{\hat{g}_i}\hat{v}_i(x,\theta,r)=&\frac{\partial \hat{f}_i}{\partial x_s}(x,\theta,r)\star_{\hat{g}_i}dx_s+\frac{\partial \hat{f}_i}{\partial \theta}(x,\theta,r)\star_{\hat{g}_i}d\theta+\frac{\partial \hat{f}_i}{\partial r}(x,\theta,r)\star_{\hat{g}_i}dr
\end{align}
Recall that $\star_{\hat{g}_i}dx_s=(-1)^{j-1}\hat{g}_i^{sj}\widehat{dx_j}\wedge d\theta\wedge dr+(-1)^{n-2}\hat{g}_i^{s\theta}dx\wedge dr+(-1)^{n-1}\hat{g}_i^{sr}dx\wedge d\theta$, where $\widehat{dx_j}=dx_1\wedge\cdots\wedge dx_{j-1}\wedge dx_{j+1}\wedge\cdots\wedge dx_{n-1}$ is the $(n-2)$-form obtained by omitting $dx_j$, and $\hat{g}_i^{sj}=\langle dx_s,dx_j\rangle_{\hat{g}_i},\hat{g}_i^{s\theta}=\langle dx_s,d\theta\rangle_{\hat{g}_i}$ and $\hat{g}_i^{sr}=\langle dx_s,dr\rangle_{\hat{g}_i}$. Similar expressions hold for $\star_{\hat{g}_i} d\theta$ and $\star_{\hat{g}_i} dr$.

Therefore, we may further write 
\begin{align}
    \star_{\hat{g}_i}\hat{v}_i(x,\theta,r)=&F^j_{i,1}(x,\theta,r)\widehat{dx_j}\wedge d\theta\wedge dr+F_{i,2}(x,\theta,r)dx\wedge dr+F_{i,3}(x,\theta,r)dx\wedge d\theta, 
\end{align}
where $F^j_{i,1},F_{i,2},F_{i,3}$ are smooth functions depending on derivatives of $\hat{f}_i$ and metric $\hat{g}_i$. For example, 
\begin{align*}
    F^j_{i,1}(x,\theta,r)=&(-1)^{j-1}\frac{\partial \hat{f}_i}{\partial x_s}\hat{g}_i^{sj}+(-1)^{n-2}\frac{\partial \hat{f}_i}{\partial \theta}\hat{g}_i^{s\theta}+(-1)^{n-1}\frac{\partial \hat{f}_i}{\partial r}\hat{g}_i^{sr}.
\end{align*}
The remaining coefficients $F_{i,2}$ and $F_{i,3}$ have analogous expressions. Since $\hat{f}_i(x,\theta+\pi,r)=-\hat{f}_i(x,\theta,r)$, we have $F^j_{i,1}(x,\theta+\pi,r)=-F^j_{i,1}(x,\theta,r)$. This anti-symmetry under $\theta \mapsto \theta + \pi$
also holds for $F_{i,2}$ and $F_{i,3}$.

Let $\epsilon_2=(\epsilon_0+\epsilon_1)/2$. Let $p:U_{\epsilon_0,\epsilon_1}\to SN^{1/2}$ be the projection and $\iota:SN^{1/2}\to U_{\epsilon_0,\epsilon_1}:(x,\theta)\mapsto (x,\theta,\epsilon_2)$ be the inclusion. As in the classical Poincar\'e lemma
in de Rham cohomology (see, e.g.,~\cite{BottTu}), we have the following:
\begin{align}\label{poincare}
    &\star_{\hat{g}_i}\hat{v}_i(x,\theta,r)-p^\ast\iota^\ast\star_{\hat{g}_i}\hat{v}_i(x,\theta,r)\\
    =&d\left(\left( \int_{\epsilon_2}^rF^j_{i,1}(x,\theta,u)du\right)\widehat{dx_j}\wedge d\theta +\left( \int_{\epsilon_2}^rF_{i,2}(x,\theta,u)du\right)dx\right).\nonumber
\end{align}
 For the second term of the left-hand side of \eqref{poincare}, note that $p^\ast\iota^\ast\star_{\hat{g}_i}\hat{v}_i(x,\theta,r)=F_{i,3}(x,\theta,\epsilon_2)dx\wedge d\theta$. Therefore, integrating over the fiber, we obtain 
\begin{align*}
    p^\ast\iota^\ast\star_{\hat{g}_i}\hat{v}_i(x,\theta,r)=d\left(\frac{1}{2}\left(\int_{\theta+\pi}^{\theta}F_{i,3}(x,\varphi,\epsilon_2)d\varphi \right)dx\right).
\end{align*}
 In conclusion, in the coordinates $(x,\theta,r)$, the $(n-2)$-form $w_i$ has an expression:
 \begin{align}\label{w_i}
     w_i(x,\theta,r)=&\left( \int_{\epsilon_2}^rF^j_{i,1}(x,\theta,u)du\right)\widehat{dx_j}\wedge d\theta +\left( \int_{\epsilon_2}^rF_{i,2}(x,\theta,u)du\right)dx\\
     &+\frac{1}{2}\left(\int_{\theta+\pi}^{\theta}F_{i,3}(x,\varphi,\epsilon_2)d\varphi \right)dx.\nonumber
 \end{align}

Next, we define a new $(n-2)$-form $w_2=(\chi\circ p) w_0+(1-\chi\circ p)w_1$ and consider its exterior derivative $dw_2$. It turns out that $p_\ast dw_2$ is a $2$-valued $(n-1)$-form on $U$ that coincides with $\star_{g_1}v_1$ on $U_{\epsilon_0}$ and with $\star_{g_0}v_0$ outside $U_{\epsilon_1}$. 

Define a $2$-valued $(n-1)$-form $\omega$ on $M$ with singular locus $\Zl_0$ by:
\begin{align}
    \omega(x)=\begin{cases}
        p_\ast dw_2(x),& x\in U_{\epsilon_1},\\
        \star_{g_1}v_1(x),&x\notin U_{\epsilon_1}.
        \end{cases}
\end{align}
Although $v_2$ and $\omega$ are $2$-valued, the $n$-form $v_2\wedge\omega$ is well-defined.
\begin{lemma}\label{wedgepositiveonbandlem}
We can choose $\delta_1$ sufficiently small, such that on the region $U_{0,\epsilon_1}$, we have $v_2\wedge \omega>0$.     
\end{lemma}
\begin{proof}
    When $0<\dist_{g_0}(x,\Zl_0)<\epsilon_0$, $\omega(x)=\star_{g_0}v_0(x)$, and hence $v_2\wedge\omega(x)>0$. 

    On the cambium $U_{\epsilon_0,\epsilon_1}$, we compute: 
    \begin{align}\label{wedgepositive}
        v_2\wedge\omega=&\left(v_0-(1-\chi)df+fd\chi\right)\wedge \\
        &\left(\star_{g_0}v_0+(1-\chi)(\star_{g_1}v_1-\star_{g_0}v_0)+d\chi\wedge p_\ast(w_0-w_1)\right)\nonumber\\
        =&|v_0|^2\mathrm{vol}_{g_0}-(1-\chi)df\wedge\star_{g_0}v_0-(1-\chi)^2df\wedge(\star_{g_1}v_1-\star_{g_0}v_0)\nonumber\\
        &-(1-\chi)df\wedge d\chi\wedge p_\ast(w_0-w_1)+fd\chi\wedge\star_{g_0}v_0\nonumber\\
        &+(1-\chi)fd\chi\wedge(\star_{g_1}v_1-\star_{g_0}v_0).\nonumber
    \end{align}
    Let's analyze the right-hand side of the last equality term by term.

    The first term is positive, and we shall show that all the remaining terms are sufficiently small to be dominated by it.

    For the second and third terms of the right-hand side of the final equality in \eqref{wedgepositive}, we have: 
    \begin{align}\label{wedgepositive1}
        &(1-\chi)df\wedge\star_{g_0}v_0+(1-\chi)^2df\wedge(\star_{g_1}v_1-\star_{g_0}v_0)\\
        =&(1-\chi)\left((df,v_0)_{g_0}-(1-\chi)|df|^2\right)\mathrm{vol}_{g_0}-(1-\chi)^2df\wedge(\star_{g_0}-\star_{g_1})v_1.\nonumber
    \end{align}
    Since $|df|$ can be made arbitrarily small as in Lemma \ref{nonvanishing1form}, we may assume that the first term of the right-hand side of \eqref{wedgepositive1} is dominated by $\frac{1}{12}|v_0|^2\mathrm{vol}_{g_0}$, i.e., \[\left|(1-\chi)\left((df,v_0)_{g_0}-(1-\chi)|df|^2\right)\mathrm{vol}_{g_0}/\frac{1}{12}|v_0|^2\mathrm{vol}_{g_0}\right|<1.\]
    On the other hand, since \[\|g_0-g_1\|_{C^{\ell,\alpha}}\le \frac{C_{\ell}2^{\ell}\delta_1}{1-C_{\ell}2^{\ell}\delta_1},\] we may may choose $\delta_1$ sufficiently small so that $\star_{g_0}-\star_{g_1}$ is pointwise small enough on $U_{\epsilon_0,\epsilon_1}$, ensuring that the second term of the right-hand side of \eqref{wedgepositive1} is dominated by $\frac{1}{12}|v_0|^2\mathrm{vol}_{g_0}$ either.

    A similar argument applies to the fifth and sixth terms of the right-hand side in the final equality of \eqref{wedgepositive}.

    Finally, we estimate the fourth term. It suffices to bound $|w_1-w_0|_{\hat{g}_0}$ on $\hat{U}_{\epsilon_0,\epsilon_1}= p^{-1}U_{\epsilon_0,\epsilon_1}$.

    Recall the local expression \eqref{w_i}, where the coefficients involves first-order derivatives of $\hat{f}_i$.
    For example, \begin{align*}
         \int_{\epsilon_2}^rF^j_{i,1}(x,\theta,u)du=\int_{\epsilon_2}^r\left((-1)^{j-1}\frac{\partial \hat{f}_i}{\partial x_s}\hat{g}_i^{sj}+(-1)^{n-2}\frac{\partial \hat{f}_i}{\partial \theta}\hat{g}_i^{s\theta}+(-1)^{n-1}\frac{\partial \hat{f}_i}{\partial r}\hat{g}_i^{sr}\right)du.
    \end{align*}
    Thus there is an upper bound \[\left|\int_{\epsilon_2}^rF^j_{1,1}(x,\theta,u)du-\int_{\epsilon_2}^rF^j_{2,1}(x,\theta,u)du\right|\le C,\]where $C$ is a constant depends on $\|f_0-f_1\|_{C^{1,\alpha}(U_{\epsilon_0,\epsilon_1})},\|g_0-g_1\|_{C^1},\epsilon_0$ and $\epsilon_1$.  A similar estimate holds for the other coefficients.
    
    It follows that $|w_1-w_0|_{\hat{g}_0}$ is uniformly bounded on $U_{\epsilon_0,\epsilon_1}$. Indeed, as $\delta_1$ becomes sufficiently small, the derivatives of $f=f_0-f_1$ on $U_{\epsilon_0,\epsilon_1}$ can be made arbitrarily small. And since the coefficients of $w_i$ depend continuously
    on the first derivatives of $\hat{f}_i$ and the metric $\hat{g}_i$, we conclude $|w_1-w_0|_{\hat{g}_1}$ tends to zero on $\hat{U}_{\epsilon_0,\epsilon_1}$ as $\delta_1\to 0$. 

    Therefore, since $|df|_{\hat{g}_0}$ and $|w_0 - w_1|_{\hat{g}_0}$ can both be made arbitrarily small,
    and $|d\chi|$ is bounded, we can assume that $(1-\chi)df\wedge d\chi\wedge p_\ast(w_0-w_1)$ is dominated by $\frac{1}{6}|v_0|^2\mathrm{vol}_{g_0}$. 

    In conclusion, $v_2\wedge\omega>\frac{1}{2}|v_0|^2\mathrm{vol}_{g_0}>0$ on $U_{\epsilon_0,\epsilon_1}$.
\end{proof}

\begin{large}
    \textbf{Step 4.}
\end{large}
We now glue the metrics $g_0$ and $g_1$, to obtain a new metric $g_2\in\M_0$, such that $\omega=\star_{g_2}v_2$. Since $d\omega=0$, it will follow that $v_2$ is harmonic with respect to $g_2$.

Moreover, the resulting metric $g_2$ can be made arbitrarily close to $g_0$, provided the metric $g_1$ used in the construction is sufficiently close to $g_0$ in the Fr\'echet topology.


Fix a point $x\in \hat{U}_{\epsilon_0,\epsilon_1}$, let's construct a metric on the vector space $T^\ast_x \hat{M}$. Such a metric determines,
and is determined by, an isomorphism $T^*_x \hat{M} \to \wedge^{n-1} T^*_x \hat{M}$ that is compatible with the Hodge star operator.

We have a vector $\hat{v}_2(x)\ne 0$ in $T^\ast_x \hat{M}$. Extend it to a basis $\{y_1=\hat{v}_2(x),y_2,\cdots,y_n\}$. There is also a nonzero vector $\hat{\omega}(x)\in \wedge^{n-1}T^\ast_x\hat{M}$. 

Using the dual basis $\widehat{y_i}=y_1\wedge\cdots\wedge y_{i-1}\wedge y_{i+1}\wedge\cdots y_n$ for $\wedge^{n-1}T^\ast_x\hat{M}$, we may write $\hat{\omega}(x)=\sum a^{1i}(x)\widehat{y_i}$. It turns out that if we extend $y_1,\cdots,y_n$ to a smooth local frame of cotangent bundle, then the coefficients $a^{1i}$ also extend locally as smooth functions.

Now define a linear isomorphism $L_x: T^\ast_x\hat{M}\to \wedge^{n-1} T^\ast_x\hat{M}$ by mapping $y_1$ to $\hat{\omega}(x)$, and for other vectors by
\begin{align*}
    L_x(y_i)= \left((\chi\circ p(x)) a_0^{ij}(x)+(1-\chi\circ p(x))a_1^{ij}(x)\right)\widehat{y_j},~~2\le i\le n,1\le j\le n,
\end{align*}
where the coefficients $a_s^{ij}(x)~(s=0,1)$ are the entries of the matrix $(\langle y_k,y_l \rangle_{\hat{g}_s(x)})$, which is the metric matrix of the metric induced by $\hat{g}_s~(s=0,1)$ on $T^\ast_x \hat{M}$, with respect to the basis $y_1,\cdots,y_n$. Here, $\chi$ is the cut-off function defined earlier.  

Next we define a symmetric bilinear form on $T^\ast_x\hat{M}$ by setting $\langle y_1,y_i\rangle=\langle y_i,y_1\rangle=a^{1i}(x)$ for $1\le i\le n$, and $\langle y_i,y_j\rangle=(\chi\circ p(x)) a_2^{ij}(x)+(1-\chi\circ p(x))a_1^{ij}(x)$ for $2\le i\le n,1\le j\le n$. 

This bilinear form extends smoothly to the cotangent bundle $T^\ast |_{\hat{U}_{\epsilon_0,\epsilon_1}}\hat{M}$. Moreover, near $\partial \hat{U}_{\epsilon_0}$, this bilinear form coincides with the metric on $T^\ast\hat{M}$ induced by $\hat{g}_0$, and near $\partial \hat{U}_{\epsilon_1}$, it coincides with the metric induced by $\hat{g}_1$.

We now show that this bilinear form is indeed a metric. Since $\chi g_0+(1-\chi)g_1$ defines a Riemannian metric on $U_{\epsilon_0,\epsilon_1}$, provided $\delta_1$ is sufficiently small, it suffices to show that the difference between the linear maps $L_p$ and $\left((\chi\circ p)\star_{\hat{g}_0}+(1-\chi\circ p)\star_{\hat{g}_1}\right)$, can be made uniformly and arbitrarily small for $x\in U_{\epsilon_0,\epsilon_1}$. It will then follow that the bilinear form defined by $L_p$ is positive definite.

For $2\le i\le n$, $L_p(y_i)=\left((\chi\circ p)\star_{\hat{g}_0}+(1-\chi\circ p)\star_{\hat{g}_1}\right)(y_i)$, hence we only need to analyze the action on the vector $y_1=\hat{v}_2(p)$. We compute:
\begin{align}\label{erroofmetric}
    &L_p(y_1)-\left((\chi\circ p)\star_{\hat{g}_0}+(1-\chi\circ p)\star_{\hat{g}_1}\right)(y_1)\\
    =&(\chi\circ p)\star_{\hat{g}_0}\hat{v}_1+(1-\chi\circ p)\star_{\hat{g}_1}\hat{v}_2+d(\chi\circ p)(w_0-w_1)\nonumber\\
     &-\left((\chi\circ p)\star_{\hat{g}_0}+(1-\chi\circ p)\star_{\hat{g}_1}\right)\left((\chi\circ p)\hat{v}_0+(1-\chi\circ p)\hat{v}_1+d(\chi\circ p)(\hat{f}_0-\hat{f}_1)\right)\nonumber\\
    =&-(1-\chi\circ p)\star_{\hat{g}_1}d\hat{f}+d(\chi\circ p)(w_0-w_1)\nonumber\\
    &+\left((\chi\circ p)\star_{\hat{g}_0}+(1-\chi\circ p)\star_{\hat{g}_1}\right)\left((1-\chi\circ p)d\hat{f}+d(\chi\circ p)\hat{f}\right)\nonumber,
\end{align}
where $d\hat{f}=\hat{v}_0-\hat{v}_1$. Each term of the right-hand side of the last equality of \eqref{erroofmetric} tends to zero as $\delta_1\to 0$, due to the tame estimates and the smallness of $|d\hat{f}|$, $|\hat{f}|$, and $|w_0 - w_1|$.

As a result, we have the following:
\begin{lemma}\label{newmetricsmall}
    When $\delta_1$ is sufficiently small, i.e. metrics $g_1$ is sufficiently close to $g_0$ in $\M_0$, there is a metric smooth Riemannian $g_2$ on $M$, such that $\star_{g_2}v_2=\omega$. Moreover, the distance $d(g_2-g_0,0)$ can be made arbitrarily small by choosing $\delta_1$ appropriately.
\end{lemma}

Finally, we summarize the above four steps as follows:
\begin{theorem}\label{normalfixdense}
    Given a metric $g_0$, a singular locus $\Zl_0$ and cohomology class $[h]$, such that $v_0=v_0(g_0,\Zl_0,[h])$ is a nondegenerate $\ZT$ harmonic $1$-form. Then for any neighborhood $\mathcal{U}\subset\M_0$ of $g_0$, there always exists another metric $g\in\mathcal{U}$, such that the $2$-valued harmonic $1$-form $v=v(g,\Zl_0,[h])$ is a transverse and nondegenerate $\ZT$ harmonic $1$-form.
\end{theorem}

Combing Theorem \ref{normalfixdense} with Donaldson's deformation theorem, we obtain our main result:
\begin{theorem}\label{perturbation}
    In a neighborhood $\mathcal{U}$ of $g_0$ in $\M$, there exists a dense open set $\mathcal{U}'\subset\mathcal{U}$, such that for any $g\in\mathcal{U}'$, $v=v(g,\underline{\Zl}(g),[h])$ is a transverse and nondegenerate $\ZT$ harmonic $1$-form.
\end{theorem}
\begin{proof}
    By Theorem \ref{deformation}, we have neighborhoods $\mathcal{U}\subset\M$ of $g_0$ and $\mathcal{V}\subset\mathcal{S}$ of $\Zl_0$, such that for any $g\in\mathcal{U}$, $v=v(g,\underline{\Zl}(g),[h])$ is a nondegenerate $\ZT$ harmonic $1$-form. 

    Next, by Theorem \ref{normalfixdense}, we have a sequence of metrics $\{g_m\}_m$, such that $g_m\to g$ and $v_m(g_m,\underline{\Zl}(g),[h])$ is transeverse, nondegenerate and with discrete ordinary zero set. Note that the singular locus $\underline{\Zl}(g)$ is fixed. This proves that the subset $\mathcal{U}'$ is dense.
    
    We show that $\mathcal{U}'$ is open. Let $g_1\in \mathcal{U}'$, and let $\Zl_1=\underline{\Zl}(g_1)$. $v_1=v(g_1,\Zl_1,[h])$ is transverse and nondegenerate. 

    We have neighborhoods $\mathcal{U}_1\times \mathcal{V}_1$ of $(g_1,\Zl_1)$, and a tame map $\Psi:\mathcal{U}_1\times\mathcal{V}_1\to \mathrm{Diff}(M)$ as in Lemma \ref{tamenessofDiff}, such that for any $(g,\Zl)\in \mathcal{U}_1\times \mathcal{V}_1$, $\Psi(g,\Zl)$ maps $\Zl$ together with its normal structure to $\Zl_1$ and its normal structure induced by $g_1$. We may also assume that the map $\underline{\Zl}$ sends $\mathcal{U}_1$ into $\mathcal{V}_1$.
    
    By tameness of $\Psi$ and $\underline{\Zl}$, and by Lemma \ref{gluingband}, there exists a constant $\epsilon_1>0$, such that for any $g\in\mathcal{U}_1$, the form $v_g=\Psi(g,\underline{\Zl}(g))^\ast v(g,\underline{\Zl}(g),[h])$ is nonvanishing on $U_{0,\epsilon_1}$. Furthermore, $|v_g|_{g_1}$ has a positive uniform lower bound lower bound on $U_{\epsilon_1/2,\epsilon_1}$. 

    Now we have a tame map $\underline{V}:\mathcal{U}_1\to C^{\infty}(T^\ast(\hat{M}\setminus \hat{U}_{\epsilon_1})):g\mapsto v_g$. $\underline{V}(g)$ is nonvanishing on the boundary of $\hat{M}\setminus \hat{U}_{\epsilon_1}$ for any $g\in\mathcal{U}_1$, and $\underline{V}(g_1)$ is transverse with the zero section of $T^\ast(\hat{M}\setminus \hat{U}_{\epsilon_1})$. Since transversailty is an open condition, we know that there is a neighborhood $\mathcal{U}_2\subset\mathcal{U}_1$ of $g_1$, such that $\underline{V}(g)$ is transverse to the zero section for any $g\in\mathcal{U}_2$.

    Therefore, for every $g \in \mathcal{U}_2$, the form $v(g, \underline{\Zl}(g), [h])$ is both transverse and nondegenerate, with discrete ordinary zero set. Hence, $\mathcal{U}_2 \subset \mathcal{U}'$, and this completes the proof. 
\end{proof}

Theorem \ref{main1} is a corollary of Theorem \ref{perturbation} and Lemma \ref{transdegzerodiscrete}.

\section{Simplify the leaf space}\label{Sec4}
In this section, we first recall the definitions of the leaf space associated to a $\ZT$ harmonic $1$-form, and of $\RR$-trees. Then we prove in Section \ref{4.1} that if a nondegenerate $\ZT$ harmonic $1$-form is transverse and its pullback to the $2$-fold branched cover represents a rational cohomology class, then its leaf space on the universal cover is a $\ZZ$-tree (Theorem \ref{simplicialtree}). We also observe that the leaf space of such a $\ZT$ harmonic $1$-form on the underlying manifold, is a finite graph (Theorem \ref{simplicialgraph}).

We then focus on the $3$-dimensional case. On a $3$-dimensional rational homology sphere, the corresponding leaf space is indeed a finite tree (Proposition \ref{finitetreeonRHS}). we show that by perturbing both the $\mathbb{Z}/2$-harmonic $1$-form and the metric,
we can arrange that in the resulting leaf space (still a finite tree), each boundary vertex corresponds exactly to one connected component of the singular locus.

Finally, we describe a pruning process (Theorem \ref{prunning}) on this leaf space: we cut away a large portion of the form and graft the exterior derivative of a harmonic function in its place. The resulting $\mathbb{Z}/2$-harmonic $1$-form has a singular locus consisting of exactly two connected components.

\vspace{2ex}
Consider a $\ZT$-harmonic $1$-form $v=v(g,\Zl,[h])$ on a closed, connected, and oriented manifold $M^n$ $(n>1)$. On the universal cover $ p:\tilde{M}\to M$, we have the pullback $2$-valued $1$-form $\tilde{v}=\pi^\ast v$, which is harmonic with respect to $\tilde{g}=\pi^\ast g$, and has singular locus $\tilde{\Zl}=\pi^{-1}(\Zl)$.

\begin{definition}
   We define a pseudo-metric on $M$ by \[d_{v}(x,y)=\inf_{\gamma}\int_{0}^1|v(\dot{\gamma}(t))|dt,\]
   where $x,y\in M$, and $\gamma$ ranges over all piecewise $C^1$ curves from $x$ to $y$. Similarly, define a pseudo-metric on $\tilde{M}$ by 
   \[d_{\tilde{v}}(x,y)=\inf_{\gamma}\int_{0}^1|\tilde{v}(\dot{\gamma}(t))|dt.\]
\end{definition}

\begin{definition}
    Define a equivalence relation on $M$ by setting: $x\sim y$ if and only if $d_v(x,y)=0$. Denote the resulting quotient space by $\T_{M,v}$. Similarly, define an equivalence relation on $\tilde{M}$, and denote the quotient space by $\T_{\tilde{M},\tilde{v}}$. The pseudo-metrics $d_v$ and $d_{\tilde{v}}$ induce genuine metrics on $\T_{M,v}$ and $\T_{\tilde{M},\tilde{v}}$, which we again denote by $d_v$ and $d_{\tilde{v}}$ for simplicity.
    
    We call the metric space $(\T_{M,v},d_v)$ (resp. $(\T_{\tilde{M},\tilde{v}},d_{\tilde{v}})$) the leaf space of $v$ on $M$ (resp. of $\tilde{v}$ on $\tilde{M}$). 

    There is a natural projection map $p_{\tilde{v}}:\tilde{M}\to \T_{\tilde{M},\tilde{v}}$ (resp. $p_{v}:M\to \T_{M,v}$).
\end{definition}

\begin{lemma}\label{mono}
    For any two points $x,y\in \tilde{M}$, we have $d_{\tilde{v}}(x,y)\ge d_{v}(\pi(x),\pi(y))$.
\end{lemma}
\begin{proof}
    Let $\gamma$ be a piecewise $C^1$ arc between $x$ and $y$, then $\pi\circ\gamma$ is a piecewise $C^1$ arc between $\pi(x)$ and $\pi(y)$ on $M$. By definition we have $d_{\tilde{v}}(x,y)\ge d_{v}(\pi(x),\pi(y))$.
\end{proof}

\begin{definition}\label{Rtrees}
    A metric space $(X,d)$ is called an $\RR$-tree if any two points in $X$ is connected by a unique shortest arc, and this arc is supposed to be isometric with a closed interval in $\RR$.
\end{definition}

\begin{lemma}{\cite[Theorem 4.4]{HeWentworthZhang}}
    $(\T_{\tilde{M},\tilde{v}},d_{\tilde{v}})$ is always an $\RR$-tree. Moreover, the deck transformation of the fundamental group $\pi_1(M)$ on $\tilde{M}$, induces an isometrical $\pi_1(M)$-action on $(\T_{\tilde{M},\tilde{v}},d_{\tilde{v}})$.
\end{lemma}

\subsection{The leaf space of a rational $\ZT$ harmonic $1$-form}\label{4.1}

\begin{definition}
    Consider a point $x$ of an $\RR$-tree $(X,d)$. If there exists a positive number $\epsilon>0$ such that the metric ball $B_{\epsilon}(x)=\{y\in X:d(x,y)<\epsilon\}$ is isometric with an open interval of $\RR$, then we call such a point $x$ an \textbf{edge point} of $(X,d)$. Any point $x$ of $(X,d)$ that is not an edge point is called a \textbf{vertex}. The set of all vertices is denoted by $\mathfrak{V}(X,d)$. 
\end{definition}

\begin{definition}
    An $\RR$-tree $(X,d)$ is called a $\ZZ$-tree if its vertex set $\mathfrak{V}(X,d)$ is a discrete subset of $(X,d)$. On a $\ZZ$-tree, a shortest arc between two vertices that contains no other vertices in the interior, is called an \textbf{edge}.
\end{definition}

Note that on a $\ZZ$-tree, there may be infinitely many edges emerging from one vertex.

The following result is a straightforward corollary of Theorem \ref{deformation} and \ref{perturbation}:
\begin{lemma}\label{rotionalandMorse}
Suppose there is a nondegenerate $\ZT$ harmonic $1$-form $v_0=v_0(g_0,\Zl_0,[h_0])$ on $M$, then there exists a pair $(g,[h])$ in a small neighborhood of $(g_0,[h_0])$ in $\M\times H$, such that $[h]\in H^1_-(\hat{M};\QQ)$, and $v=v(g,\underline{\Zl}(g),[h])$ is a transverse and nondegenerate $\ZT$ harmonic $1$-form on $M$.    
\end{lemma}

\begin{definition}
    A $\mathbb{Z}/2$-harmonic $1$-form $v = v(g, \Zl, [h])$ is called \textbf{rational} if the cohomology class $[h]$ lies in
    $H^1_-(\hat{M}; \mathbb{Q})$.
\end{definition}

Through out this subsection, we shall consider a rational, transverse and nondegenerate $\ZT$ harmonic $1$-form $v=v(g,\Zl,[h])$, and let $p:\hat{M}\to M$ be the $2$-fold branched cover branching along $\Zl$. And $\hat{v}=p^\ast v$. Since $[\hat{v}]=[h]\in H^1_-(\hat{M};\QQ)$, there exists a continuous map $u:\hat{M}\to S^1$, such that $du=u^\ast d\theta=\mu \hat{v}$ for some $\mu\in\QQ$. 
This property implies that the singular foliation induced by $v$ is significantly simpler than that of non-rational $\mathbb{Z}/2$-harmonic $1$-forms.

We now analyze the structure of singular foliations.

The well-defined $1$-form $\hat{v}$ induces a singular foliation, namely $\ker \hat{v}$ on $\hat{M}$. Similarly, $\ker v$ defines a singular foliation on $M$, and $\ker \tilde{v}$ on $\tilde{M}$.

\begin{definition}
    Let $I_v$ denote the set of isolated zeros of $v$ on $M$, and let $\hat{I}_v = p^{-1}(I_v)$ be the corresponding zero set of $\hat{v}$ on $\hat{M}$.
\end{definition}

A point $x\in \hat{M}$ is a critical point of $u:\hat{M}\to S^1$ if and only if $\hat{v}(x)=0$. Hence, the critical point set of $u$ is precisely the zero set of $\hat{v}$, namely $\Zl \cup \hat{I}_v$.

For the purpose of the following discussion, we define:
\begin{definition}
    A \textbf{leaf} $L$ of $\ker \hat{v}$ is a connected component of a level set of $u$, which is a compact subset of $\hat{M}$. $L$ is called \textbf{regular} if every point $y \in L$ is a regular point of $u$; otherwise, $L$ is \textbf{singular}.
\end{definition}
This definition is adapted to the present context and may not align with general conventions in foliation theory.

A regular leaf $L$ of $\ker\hat{v}$ is always a compact codimension-$1$ submanifold of $\hat{M}$. It carries a volume form $\star_{\hat{g}}\hat{v}$ on such a leaf $L$, and is hence oriented, with trivial normal bundle in $\hat{M}$. On the other hand, if $L$ is singular, then $L$ contains at least one component of $\Zl\cup \hat{I}_v$. Since $\Zl \cup \hat{I}_v$ has only finitely many connected components,
there are only finitely many singular leaves. The union of singular leaves is called the \textbf{critical graph} of $\ker\hat{v}$ on $\hat{M}$.

Similarly, we define regular, singular leaves as well as the critical graph for singular foliations $\ker v$ and $\ker\tilde{v}$. A leaf of $\ker v$ is called \textbf{regular} if it is the image under $p:\hat{M} \to M$ of a regular leaf of $\ker \hat{v}$; otherwise, it is \textbf{singular}.
For the foliation $\ker \tilde{v}$ on $\tilde{M}$, a leaf is \textbf{regular} if it is a connected component of $\pi^{-1}(L)$, where $L$ is a regular leaf of $\ker v$.

Let's study the structure of $\ker v$. If $L$ is a regular leaf of $\ker v$, then one can see that $p^{-1}(L)$ has two components, interchanged by the involution $\tau$ on $\hat{M}$, and each is a regular leaf of $\ker \hat{v}$. Thus $L$ is a smooth orientable submanifold of $M$ with codimension $1$ and trivial normal bundle. 

On the other hand, if $L$ is a singular leaf of $\ker v$, then $L\setminus(\Zl\cup I_v)$ is a smooth noncompact submanifold of $M$ with codimension $1$ and trivial normal bundle, but may have multiple connected components.

The local structure of a singular leaf $L$ near an isolated zero $y \in I_v$ or near a component of $\Zl$ is described as follows. 

Let $B$ be a small open ball centered at $y\in I_v$. By definition of transverse $\ZT$ harmonic $1$-form, there is a Morse function $f$ on $B$, such that $v|_{B}=\pm df$. Therefore, we may choose Morse coordinate on $B$. If $y\in L$, then $B\cap L$ is diffeomorphic to the cone $\{(x_1,\cdots,x_n)\in\RR^n:x_1^2=x_2^2+\cdots+x_n^2\}$.

Let $\Sigma$ be a component of $\Zl$, and suppose $\Sigma\subset L$. For any $y\in \Sigma$, there is a neighborhood $B\subset M$ of $y$, and a diffeomorphism $\phi:B\to \CC\times\RR^{n-2}$, such that $\phi(\Sigma\cap U)=\{0\}\times\RR^{n-2}$. We may choose this chart so that $\phi(L\cap B)=\{(z,t)\in \CC\times\RR^{n-2}:\re z^{3/2}=0\}$. See Figure \ref{trivalent}.

\begin{figure}[!h]
    \centering

\tikzset{every picture/.style={line width=0.75pt}} 

\begin{tikzpicture}[x=0.5pt,y=0.5pt,yscale=-1,xscale=1]

\draw  [dash pattern={on 4.5pt off 4.5pt}] (50,141.08) .. controls (50,86.45) and (94.29,42.17) .. (148.92,42.17) .. controls (203.55,42.17) and (247.83,86.45) .. (247.83,141.08) .. controls (247.83,195.71) and (203.55,240) .. (148.92,240) .. controls (94.29,240) and (50,195.71) .. (50,141.08) -- cycle ;
\draw    (28.33,140.67) -- (148.92,141.08) ;
\draw    (148.92,141.08) -- (211.83,36.67) ;
\draw    (148.92,141.08) -- (211.33,243.17) ;
\draw  [dash pattern={on 4.5pt off 4.5pt}] (348.67,140.92) .. controls (348.67,86.29) and (392.95,42) .. (447.58,42) .. controls (502.21,42) and (546.5,86.29) .. (546.5,140.92) .. controls (546.5,195.55) and (502.21,239.83) .. (447.58,239.83) .. controls (392.95,239.83) and (348.67,195.55) .. (348.67,140.92) -- cycle ;
\draw    (327,140.5) -- (447.58,140.92) ;
\draw    (447.58,140.92) -- (510.5,36.5) ;
\draw    (447.58,140.92) -- (510,243) ;
\draw [color={rgb, 255:red, 74; green, 144; blue, 226 }  ,draw opacity=1 ]   (325.83,120.17) .. controls (438.33,120.17) and (437.83,117.17) .. (497.83,22.67) ;
\draw [color={rgb, 255:red, 74; green, 144; blue, 226 }  ,draw opacity=1 ]   (528.33,236.67) .. controls (469.83,141.17) and (470.33,141.67) .. (527.83,41.17) ;
\draw [color={rgb, 255:red, 74; green, 144; blue, 226 }  ,draw opacity=1 ]   (327.33,161.17) .. controls (439.83,161.17) and (437.83,164.67) .. (496.83,257.67) ;

\draw (68,113) node [anchor=north west][inner sep=0.75pt]    {$\mathrm{Re} z^{3/2} =0$};
\draw (368,113) node [anchor=north west][inner sep=0.75pt]    {$\mathrm{\textcolor[rgb]{0,0.45,0.98}{Re}}\textcolor[rgb]{0,0.45,0.98}{z}\textcolor[rgb]{0,0.45,0.98}{^{3/2}}\textcolor[rgb]{0,0.45,0.98}{< \epsilon }$};

\end{tikzpicture}

    \caption{The set $\re z^{3/2}=0$ and the region $|\re z^{3/2}|<\epsilon$}
    \label{trivalent}
\end{figure}
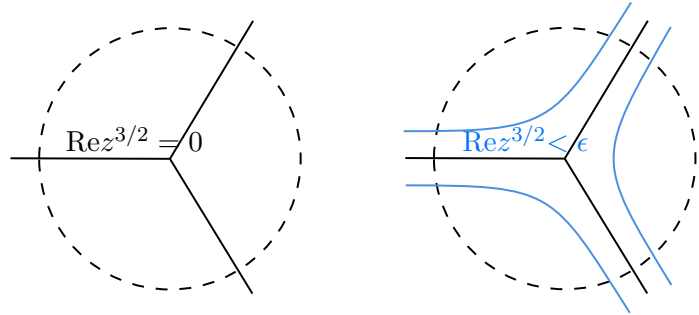

Denote the critical graph of $\ker v$ by $\mathrm{G}_v$, and let $\mathrm{G}_{\hat{v}}:=p^{-1}(\mathrm{G}_v)$ be the critical graph of $\ker \hat{v}$. Suppose the leaf $L$ is a component of $\mathrm{G}_v$ and has nonempty intersection with $\Zl$, then $p^{-1}(L)$ is connected. If $L\cap\Zl=\varnothing$, then $p^{-1}(L)$ has two components. 

Let's define a neighborhood $U_\epsilon$ of a leaf $L$ for $\epsilon>0$ small. 

Firstly, consider a tubular neighborhood $U$ of $\Zl\cap L$, on which $v|_U = df$, where $f$ is a $2$-valued harmonic function, and we have $f^{-1}(0)=\Zl\cap L$. Define an open subset $U^1_\epsilon(L)=\{x\in U:|f(x)|<\epsilon\}$. When $\epsilon>0$ is sufficiently small, the number of connected components equals that of $\Zl\cap L$. If $L\cap\Zl=\varnothing$, we set $U_\epsilon^1(L)=\varnothing$. 

Suppose $L\cap I_v=\{x_1,\cdots,x_j\}$. Let $B_i$ be a small open ball centered at $x_i\in L\cap I_v$. $v|_{B_i}=\pm df_i$ for some Morse function $f_i$ on $B_i$. We may assume $f_i^{-1}(0)=x_i$. Define a neighborhood $U^2_\epsilon(L)$ of $L\cap (\cup_{1\le i\le j}B_i)$ by $\{x\in \cup_{1\le i\le j}B_i:|f_i(x)|<\epsilon\}$. If $L \cap I_v = \varnothing$, set $U^2_\epsilon(L) = \varnothing$.

Let $L':=L\setminus \left(U\cup(\cup_{1\le i\le j}B_i)\right)$, which is a codimension-$1$ submanifold (with boundary, if $L$ is singular) of $M$. Let $U'$ be a neighborhood of $L'$, that deformation retracts onto $L'$ and on which $v$ is nonvanishing.  These conditions imply that $v|_{U'}=\pm df$ for some well-defined function $f$ such that $f^{-1}(0)=L'$. We define $U^3_\epsilon(L)=\{x\in U':|f(x)|<\epsilon\}$.

Then $U_\epsilon^1(L)\cup U_\epsilon^2(L)\cup U^3_\epsilon(L)=:U_\epsilon(L)$ is a connected neighborhood of $L$. Moreover, $U_\epsilon(L)$ deformation retracts onto $L$, and $\partial U_\epsilon(L)$ is a union of regular leaves of $\ker v$. For a regular leaf $L$, $U_\epsilon(L)=U^3_\epsilon(L)\approx L\times(-\epsilon,\epsilon)$.

Consider two distinct leaves $L_1,L_2$, then we can find $\epsilon>0$ so that $U_\epsilon(L_1)\cap U_\epsilon(L_2)=\varnothing$. Therefore, for any arc $\gamma$ connecting $L_1$ and $L_2$, $\int_0^1|v(\dot{\gamma}(t))|dt\ge 2\epsilon$.

Now consider $\tilde{v}$ on $\tilde{M}$ and the $\RR$-tree $(\T_{\tilde{M},\tilde{v}},d_{\tilde{v}})$. 

\begin{lemma}\label{pointandleaf}
    Let $x$ be a point of $\tilde{M}$, then the equivalent class $[x]\in\T_{\tilde{M},\tilde{v}}$ is the leaf $L_x$ of $\ker\tilde{v}$ passing through $x$. 
\end{lemma}
\begin{proof}
    By definition, $[x]=\{x'\in\tilde{M}:d_{\tilde{v}}(x,x')=0\}$. For any point $y\in L_x$, there is an arc $\gamma\subset L_x$ from $x$ to $y$, such that $|\tilde{v}(\dot{\gamma}(t))|\equiv 0$. Thus $L_x\subset [x]$.

    Suppose for contradiction that there exists $y \in [x] \setminus L_x$. Then there is a sequence of arcs $\gamma_i$ connecting $x$ and $y$, such that $\int_0^1|\tilde{v}(\dot{\gamma}_i)(t)|dt\to 0$ as $i\to \infty$.

    These induce a sequence of arcs $\underline{\gamma_i}=\pi\circ\gamma_i$ from $\pi(x)$ to $\pi(y)$. Since $\pi(L_x)$ is a leaf of $\ker v$, for $\epsilon>0$ sufficiently small, we can define a neighborhood $U_\epsilon(\pi(L_x))$ of $\pi(L_x)$ as above. 
    
    If $\underline{\gamma_i}$ is not contained in $U_\epsilon(\pi(L_x))$, then $\int_0^1|\tilde{v}(\dot{\gamma}_i)(t)|dt=\int_0^1|v\left(\underline{\dot{\gamma}_i}\right)(t)|dt\ge \epsilon$. As a result, $\pi(y)$ must be contained in $U_\epsilon(\pi(L_x))$. Since $\epsilon>0$ is arbitrary, we conclude that $\pi(y)\in \pi(L_x)$.

    Consequently, $y$ lies in another component of $\pi^{-1}(\pi(L_x))$. We may assume there is a $\sigma\in\pi_1(M)$ such that $\sigma.x=y$. For any arc $\gamma$ connecting $x$ and $y$, $\underline{\gamma}=\pi\circ\gamma$ represents $\sigma$. And similarly, we may assume $\underline{\gamma}\subset U_\epsilon(\pi(L_x))$. 
    
    However, by construction, $U_\epsilon(\pi(L_x))$ deformation retracts onto $\pi(L_x)$. Thus $\sigma\in \mathrm{im}(\pi_1(L_x))$, a subgroup of $\pi_1(M)$ that is the image of $\pi_1(L_x)$ under the inclusion map. Therefore $y=\sigma.x\in L_x$, leading to a contradiction.
\end{proof}

\begin{lemma}
    A point $[x]$ of $\T_{\tilde{M},\tilde{v}}$ is an edge point if the corresponding leaf $L_x$ is regular. As a result, $[x]$ is a vertex only if $L_x$ is singular. 
\end{lemma}
\begin{proof}
    If $L_x$ is regular, then so is $\pi(L_x)$. For sufficiently small $\epsilon>0$, the neighborhood $U_\epsilon(\pi(L_x))$ is diffeomorphic to $\pi(L_x)\times(-\epsilon,\epsilon)$. And $\pi^{-1}\left(U_\epsilon(\pi(L_x))\right)$ has the same number of components as $\pi^{-1}(\pi(L_x))$, since $U_\epsilon(\pi(L_x))$ deformation retracts onto $\pi(L_x)$.
    
    One of these components contains $L_x$; denote it by $U_\epsilon(L_x)$.
    Then $U_\epsilon(L_x) \approx L_x \times (-\epsilon, \epsilon)$ is a tubular neighborhood of $L_x$.

    Note that $U_\epsilon(L_x)\approx L_x\times(-\epsilon,\epsilon)$, and $p_{\tilde{v}}(U_\epsilon(L_x))=(-\epsilon,\epsilon)$. By definition, we know that the metric open ball $B_{\epsilon}([x])$ can be identified with $(-\epsilon,\epsilon)$, and hence $[x]$ is an edge point of $\T_{\tilde{M},\tilde{v}}$.
\end{proof}

\begin{theorem}\label{simplicialtree}
    The $\RR$-tree $(\T_{\tilde{M},\tilde{v}},d_{\tilde{v}})$ is a $\ZZ$-tree.
\end{theorem}
\begin{proof}
    Let $[x_1]$ and $[x_2]$ be two distinct vertices of $(\T_{\tilde{M},\tilde{v}},d_{\tilde{v}})$. We shall show that there is a positive constant $\epsilon_0$ independent of $[x_1]$ and $[x_2]$, such that $d_{\tilde{v}}([x_1],[x_2])\ge\epsilon_0$.

    Indeed, since the number of components of $\mathrm{G}_v$ is finite, there is a $\epsilon_0>0$ such that, for any two distinct singular leaves $L,L'\subset\mathrm{G}_v$, $U_{\epsilon_0}(L)\cap U_{\epsilon_0}(L')=\varnothing$.
    
    Let $L_1$ and $L_2$ be the two singular leaves of $\ker\tilde{v}$ corresponding to $[x_1]$ and $[x_2]$, respectively. If $d_{\tilde{v}}([x_1],[x_2])<\epsilon_0$, then there is an arc $\gamma$ connecting $x_1$ and $x_2$, such that $\int_0^1|\tilde{v}(\dot{\gamma}(t))|dt<\epsilon_0$. 
    As in the proof of Lemma \ref{pointandleaf}, this will imply that $\pi(L_1)=\pi(L_2)$. Thus we may assume $x_2=\sigma.x_1$ for some $\sigma\in\pi_1(M)$.

    Then similarly, one can show that $\pi\circ\gamma$ represents $\sigma$. Since $\pi\circ\gamma\subset U_{\epsilon_0}(\pi(L_1))$, which deformation retracts to $\pi(L_1)$, we have $x_2=\sigma.x_1\in L_1$. This leads to a contradiction.
\end{proof}

\begin{theorem}\label{simplicialgraph}
    The metric space $(\T_{M,v},d_v)$ is a finite graph.
\end{theorem}
\begin{proof}
    The complement $M\setminus\mathrm{G}_v$ has several components. Let's write $M\setminus\mathrm{G}_v=W_1\sqcup W_2\sqcup\cdots \sqcup W_m$. 
    Since any loop $\gamma$ that $E|_{\gamma}$ is nontrivial must have nontrivial intersection with a singular leaf, we see that on each component $W_i$, $v|_{W_i}=\pm w_i$ for some well-defined and nonvanishing $1$-form $w_i$.

    Therefore, each $W_i$ is foliated by regular leaves of $\ker v$, and admits a natural product structure $W_i \approx L_i \times (0, d_i)$,
    where $L_i$ is any regular leaf in $W_i$, and $d_i = \sup_{x,y \in W_i} d_v(x,y)$ is the diameter of $W_i$ in the leaf space metric.

    We now define a graph $\mathcal{G}$. The vertices of $\mathcal{G}$ are the singular leaves of $\ker v$. There is an edge between two vertices $L_1$ and $L_2$ if and only if there exists a component $W_i$ such that both $L_1$ and $L_2$ are contained in $\overline{W}_i$. We equip $\mathcal{G}$ with a metric by assigning to each edge the length $d_i$ of the corresponding region $W_i$. 
    
    The projection $p_v \colon M \to \T_{M,v}$ collapses each regular leaf to a point, and maps each $W_i$ to an open interval of length $d_i$,
    with endpoints corresponding to the boundary singular leaves. Hence, $\T_{M,v}$ is isometric to $\mathcal{G}$.

    There is a subtlety that the definitions of vertices are different on $\T_{\tilde{M},\tilde{v}}$ and $\T_{M,v}$. Note that for a vertex $[x]$ of $\T_{M,v}\cong\mathcal{G}$, it is possible that there is a small neighborhood $B_{\epsilon}([x])$ that is isometric with the open interval $(-\epsilon,\epsilon)$. 
\end{proof}

\begin{remark}A subtlety arises in Theorem \ref{simplicialgraph}: unlike in $\T_{\tilde{M},\tilde{v}}$, a vertex $[x] \in \T_{M,v}$ may have a neighborhood isometric to an open interval $(-\epsilon, \epsilon)$. Also, note that $(\T_{M,v},d_v)$ is isometric with the quotient space of the  $\ZZ$-tree $(\T_{\tilde{M},\tilde{v}},d_{\tilde{v}})$ under the isometric $\pi_1(M)$-action. Thus a vertex of $\T_{\tilde{M},\tilde{v}}$ can be folded to be a vertex of $\T_{M,v}$ whose neighborhood isometric to some open interval.
\end{remark}

\subsection{The leaf space on $3$-manifolds.}\label{4.2}

In this subsection, we study rational, transverse and nondegenerate $\ZT$ harmonic $1$-forms on closed oriented $3$-dimensional manifolds. In this case, any singular locus $\Zl$ is a disjoint union of $S^1$ in the $3$-manifold. And the local structure of singular foliation near a component of $\Zl$ is simple.

\begin{theorem}\label{distinguish}
    If $v=v(g,\Zl,[h])$ is a rational, transverse and nondegenerate $\ZT$ harmonic $1$-forms on a closed oriented $3$-dimensional manifold $M^3$. Then there exists another $\ZT$ harmonic $1$-form $v'=v'(g',\Zl,[h])$ with the same singular locus $\Zl$ such that: 
    \begin{enumerate}[label=(\roman*)]
        \item $v'$ is still rational, transverse and nondegenerate.
        \item Each singular leaf $L$ of $\ker v'$ contains at most one component of $\Zl$. Moreover, if $L$ contains a component of $\Zl$, we may assume that $L\cap I_{v'}=\varnothing$.
        \item The metric $g'$ can be chosen arbitrarily close to $g$ in the Fr\'echet topology on $\M$, i.e., for any $\varepsilon>0$, we can make $d(g'-g,0)<\epsilon$.
    \end{enumerate}
\end{theorem}
\begin{proof}
    Let's write $\Zl=\Sigma_1\sqcup\cdots\sqcup \Sigma_n$ for some $n\ge 1$, each $\Sigma_j$ is diffeomorphic with $S^1$.

    For each $\Sigma_j~(1\le j\le n)$, choose a tubular neighborhood $U_i$ and a diffeomorphism $\phi:S^1\times D^2\to U_j$, such that $\phi^\ast v=d\re z^{3/2}$. Let $L_j$ be the singular leaf containing $\Sigma_j$, then we have $L_j\cap U_j\approx S^1\times\{z\in D^2:\re z^{3/2}=0\}$. Note that it could happen that $L_j=L_k$ for $j\ne k$. Actually, we will perturb $v$ together with $g$, to make sure that $L_j\ne L_k$ whenever $j\ne k$. The argument is similar to the grafting of transverse $2$-valued harmonic $1$-forms (Theorem \ref{singulartrans}).

    For each $j$, there exists open sets $B_{j,l}$ in $U_j$, such that $\phi^{-1}(B_{j,l})=e^{(2l-1)\pi i/3}=S^1\times\{e^{(2l+1)\pi i/3}(z+1-2\epsilon_1):z\in Q_{\epsilon_1}\}$, where $Q_{\epsilon_1}=\{(x_1,x_2):-\epsilon_1\le x_1,x_2\le\epsilon_1\}$ for $1\le l\le 3$. See Figure \ref{boxes}. 
    
    We can choose coordinates $(\theta,x_1,x_2)$ on each open set $B_{i,l}\approx S^1\times Q_\epsilon$ such that $v|_{B_{i,l}}=\pm dx_2$. 

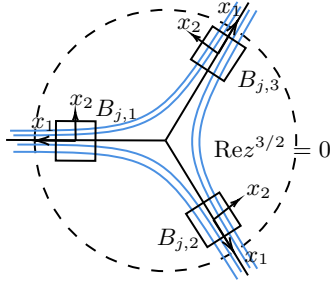
\begin{figure}[!h]
        \centering

\tikzset{every picture/.style={line width=0.75pt}} 

\begin{tikzpicture}[x=0.5pt,y=0.5pt,yscale=-1,xscale=1]

\draw  [dash pattern={on 4.5pt off 4.5pt}] (70,161.08) .. controls (70,106.45) and (114.29,62.17) .. (168.92,62.17) .. controls (223.55,62.17) and (267.83,106.45) .. (267.83,161.08) .. controls (267.83,215.71) and (223.55,260) .. (168.92,260) .. controls (114.29,260) and (70,215.71) .. (70,161.08) -- cycle ;
\draw    (48.33,160.67) -- (168.92,161.08) ;
\draw    (168.92,161.08) -- (231.83,56.67) ;
\draw    (168.92,161.08) -- (231.33,263.17) ;
\draw [color={rgb, 255:red, 74; green, 144; blue, 226 }  ,draw opacity=1 ]   (53.83,157.17) .. controls (166.33,157.17) and (165.83,154.17) .. (225.83,59.67) ;
\draw [color={rgb, 255:red, 74; green, 144; blue, 226 }  ,draw opacity=1 ]   (51.33,153.67) .. controls (163.83,153.67) and (163.33,150.67) .. (223.33,56.17) ;
\draw [color={rgb, 255:red, 74; green, 144; blue, 226 }  ,draw opacity=1 ]   (56.83,164.17) .. controls (169.33,164.17) and (167.33,167.67) .. (226.33,260.67) ;
\draw [color={rgb, 255:red, 74; green, 144; blue, 226 }  ,draw opacity=1 ]   (53.83,169.67) .. controls (166.33,169.67) and (164.33,173.17) .. (223.33,266.17) ;
\draw [color={rgb, 255:red, 74; green, 144; blue, 226 }  ,draw opacity=1 ]   (232.83,257.67) .. controls (174.33,162.17) and (174.83,162.67) .. (232.33,62.17) ;
\draw [color={rgb, 255:red, 74; green, 144; blue, 226 }  ,draw opacity=1 ]   (238.33,258.17) .. controls (179.83,162.67) and (180.33,163.17) .. (237.83,62.67) ;
\draw   (86.08,146.34) -- (116.04,146.54) -- (115.84,176.5) -- (85.89,176.3) -- cycle ;
\draw   (188.14,99.71) -- (204.54,74.64) -- (229.61,91.04) -- (213.21,116.11) -- cycle ;
\draw   (209.67,200.44) -- (225.81,225.67) -- (200.58,241.81) -- (184.44,216.58) -- cycle ;
\draw    (100.96,161.42) -- (100.69,141.33) ;
\draw [shift={(100.67,139.33)}, rotate = 89.23] [color={rgb, 255:red, 0; green, 0; blue, 0 }  ][line width=0.75]    (6.56,-1.97) .. controls (4.17,-0.84) and (1.99,-0.18) .. (0,0) .. controls (1.99,0.18) and (4.17,0.84) .. (6.56,1.97)   ;
\draw    (100.96,161.42) -- (72,161.11) ;
\draw [shift={(70,161.08)}, rotate = 0.62] [color={rgb, 255:red, 0; green, 0; blue, 0 }  ][line width=0.75]    (10.93,-3.29) .. controls (6.95,-1.4) and (3.31,-0.3) .. (0,0) .. controls (3.31,0.3) and (6.95,1.4) .. (10.93,3.29)   ;
\draw    (208.87,95.38) -- (190.61,82.01) ;
\draw [shift={(189,80.83)}, rotate = 36.19] [color={rgb, 255:red, 0; green, 0; blue, 0 }  ][line width=0.75]    (6.56,-1.97) .. controls (4.17,-0.84) and (1.99,-0.18) .. (0,0) .. controls (1.99,0.18) and (4.17,0.84) .. (6.56,1.97)   ;
\draw    (205.12,221.13) -- (222.93,207.07) ;
\draw [shift={(224.5,205.83)}, rotate = 141.72] [color={rgb, 255:red, 0; green, 0; blue, 0 }  ][line width=0.75]    (6.56,-1.97) .. controls (4.17,-0.84) and (1.99,-0.18) .. (0,0) .. controls (1.99,0.18) and (4.17,0.84) .. (6.56,1.97)   ;
\draw    (208.87,95.38) -- (222.49,72.06) ;
\draw [shift={(223.5,70.33)}, rotate = 120.29] [color={rgb, 255:red, 0; green, 0; blue, 0 }  ][line width=0.75]    (10.93,-3.29) .. controls (6.95,-1.4) and (3.31,-0.3) .. (0,0) .. controls (3.31,0.3) and (6.95,1.4) .. (10.93,3.29)   ;
\draw    (205.12,221.13) -- (219.43,243.65) ;
\draw [shift={(220.5,245.33)}, rotate = 237.58] [color={rgb, 255:red, 0; green, 0; blue, 0 }  ][line width=0.75]    (10.93,-3.29) .. controls (6.95,-1.4) and (3.31,-0.3) .. (0,0) .. controls (3.31,0.3) and (6.95,1.4) .. (10.93,3.29)   ;

\draw (203.33,154.57) node [anchor=north west][inner sep=0.75pt]  [font=\footnotesize]  {$\mathrm{Re} z^{3/2} =0$};
\draw (113.5,128) node [anchor=north west][inner sep=0.75pt]  [font=\footnotesize]  {$B_{j,1}$};
\draw (160,225.9) node [anchor=north west][inner sep=0.75pt]  [font=\footnotesize]  {$B_{j,2}$};
\draw (221,106.4) node [anchor=north west][inner sep=0.75pt]  [font=\footnotesize]  {$B_{j,3}$};
\draw (94,124.4) node [anchor=north west][inner sep=0.75pt]  [font=\footnotesize]  {$x_{2}$};
\draw (65,142) node [anchor=north west][inner sep=0.75pt]  [font=\footnotesize]  {$x_{1}$};
\draw (175.5,66.9) node [anchor=north west][inner sep=0.75pt]  [font=\footnotesize]  {$x_{2}$};
\draw (228,192.4) node [anchor=north west][inner sep=0.75pt]  [font=\footnotesize]  {$x_{2}$};
\draw (226.5,242.9) node [anchor=north west][inner sep=0.75pt]  [font=\footnotesize]  {$x_{1}$};
\draw (206,54.9) node [anchor=north west][inner sep=0.75pt]  [font=\footnotesize]  {$x_{1}$};

\end{tikzpicture}

      \caption{The singular foliation near $\Sigma_j$}
        \label{boxes}
    \end{figure}

    Let's define a function on the square $Q_{\epsilon_1}$, as follows. First, choose a constant $\epsilon_2>0$ smaller that $\epsilon_1$, let $\eta_{\epsilon_2}:[-\epsilon_1,\epsilon_1]\to [0,\epsilon_2]$ be a smooth function that $\eta_{\epsilon_2}(x)=0$ for $-\epsilon_1\le x\le -\epsilon_1/2$ and $\eta_{\epsilon_2}(x)={\epsilon_2}$ for $\epsilon_1/2\le x\le \epsilon_1$. Additionally, we assume $|d\eta_{\epsilon_2}|\le 2\epsilon_2/\epsilon_1$. 
    
    Second, let $\chi:[-\epsilon_1,\epsilon_1]\to [0,1]$ be a smooth bump function, such that the support of $\chi$ is contained in $(-\epsilon_1,\epsilon_1)$ and $|d\chi|\le 2/\epsilon_1$. We define $h_{\epsilon_2}:Q_{\epsilon_1}\to [-\epsilon_1,\epsilon_1]:(x_1,x_2)\mapsto x_2+\eta_{\epsilon_2}(x_1)\chi(x_2)$.  

    Then $|dh_{\epsilon_2}|\ge 1-|d\eta_{\epsilon_2}|-\epsilon_2|d\chi|\ge 1-4\epsilon_2/\epsilon_1>0$, provided $\epsilon_2<\epsilon_1/4$. The level sets of $h_{\epsilon_2}$ are described in Figure \ref{hc}. 

\begin{figure}[!h]
    \centering

\tikzset{every picture/.style={line width=0.75pt}} 

\begin{tikzpicture}[x=0.5pt,y=0.5pt,yscale=-1,xscale=1]

\draw    (152.17,132) .. controls (168.17,132.5) and (173.67,132) .. (184.17,135) .. controls (194.67,138) and (232.67,150) .. (240.67,152.5) .. controls (248.67,155) and (258.17,155.5) .. (273.17,155.5) ;
\draw [color={rgb, 255:red, 74; green, 144; blue, 226 }  ,draw opacity=1 ]   (152.33,113) .. controls (168.33,113.5) and (177.83,116) .. (186.83,119) .. controls (195.83,122) and (225.33,129.5) .. (233.33,132) .. controls (241.33,134.5) and (258.33,135.5) .. (273.33,135.5) ;
\draw [color={rgb, 255:red, 74; green, 144; blue, 226 }  ,draw opacity=1 ]   (152.67,97) .. controls (168.67,97.5) and (177.83,98) .. (187.33,100) .. controls (196.83,102) and (217.83,105) .. (229.33,107) .. controls (240.83,109) and (257.33,110) .. (272.33,110) ;
\draw [color={rgb, 255:red, 74; green, 144; blue, 226 }  ,draw opacity=1 ]   (151.67,84) -- (272.17,84) ;
\draw [color={rgb, 255:red, 74; green, 144; blue, 226 }  ,draw opacity=1 ]   (151.33,160.5) .. controls (167.33,161) and (182.33,162) .. (192.33,164.5) .. controls (202.33,167) and (217.83,171) .. (228.83,173) .. controls (239.83,175) and (246.83,177) .. (272.33,177) ;
\draw [color={rgb, 255:red, 74; green, 144; blue, 226 }  ,draw opacity=1 ]   (152.5,190) -- (273.33,190) ;
\draw    (152.17,132) -- (300.17,132.49) ;
\draw [shift={(302.17,132.5)}, rotate = 180.19] [color={rgb, 255:red, 0; green, 0; blue, 0 }  ][line width=0.75]    (10.93,-3.29) .. controls (6.95,-1.4) and (3.31,-0.3) .. (0,0) .. controls (3.31,0.3) and (6.95,1.4) .. (10.93,3.29)   ;
\draw    (213.5,195.5) -- (212.68,42) ;
\draw [shift={(212.67,40)}, rotate = 89.69] [color={rgb, 255:red, 0; green, 0; blue, 0 }  ][line width=0.75]    (10.93,-3.29) .. controls (6.95,-1.4) and (3.31,-0.3) .. (0,0) .. controls (3.31,0.3) and (6.95,1.4) .. (10.93,3.29)   ;
\draw   (152,74.5) -- (273,74.5) -- (273,195.5) -- (152,195.5) -- cycle ;

\draw (308,122.9) node [anchor=north west][inner sep=0.75pt]    {$x_{1}$};
\draw (221.5,31.9) node [anchor=north west][inner sep=0.75pt]    {$x_{2}$};
\draw (277,144.9) node [anchor=north west][inner sep=0.75pt]    {$\epsilon_2$};

\end{tikzpicture}

    \caption{The level set of $h_{\epsilon_2}$.}
    \label{hc}
\end{figure}
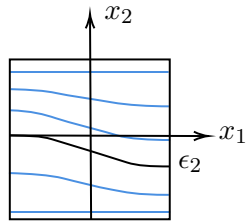

    Note that near the boundary of $Q_{\epsilon_1}$, $h_{\epsilon_2}$ is independent of $x_1$, and hence $dh_{\epsilon_2}=dx_2$.
    Let $p_2:S^1\times \{e^{(2l+1)\pi i/3}(z+1-2\epsilon_1):z\in Q_{\epsilon_1}\}\to Q_\epsilon$ be the projection that maps $(\theta,e^{(2l+1)\pi i/3}(z+1-2\epsilon_1))$ to $z$. Define a nonvanishing $1$-form $(\phi^{-1})^\ast p_2^\ast dh_{\epsilon_2}$ on each $B_{i,l}$. We need to choose different $\epsilon_2$ for each $\Sigma_j~(1\le j\le n)$ in the following, hence we write $\epsilon_2=\epsilon_2(j)$.

    The $1$-form $\sum_j(\phi^{-1})^\ast p_2^\ast dh_{\epsilon_2(j)}$ is defined on the open subset $\cup_{j,l}B_{j,l}$ of $M$, we can extend it to be a $2$-valued closed $1$-form $v'$, that coincides with $v$ outside $\cup_{j,l}B_{j,l}$. For each singular leaf $L_j~(j\ne 1)$, $p_2\circ\phi^{-1} (L_j\cap B_{k,l})=:C_{j,k,l}$ is a disjoint union of curves in $Q_{\epsilon_1}$. Moreover, each component of $C_{j,k,l}$ is a level set of $h_{\epsilon_2(k)}$ on $Q_{\epsilon_1}$. 

    If $\epsilon_2(j)=0$ for $j\ne1$, then $v'$ coincides $v$ outside the open sets $B_{1,l}~(1\le l\le 3)$. Note that $h_{\epsilon_2(1)}|_{C_{1,1,l}}$ is invariant as $\epsilon(2)$ varies. On the other hand, for $j\ne 1$, if $L_j=L_1$, $h_{\epsilon_2(1)}|_{C_{j,1,l}}$ will change for $\epsilon_2(1)>0$. 
    
    Since the set $\{h_{\epsilon_2(1)}|_{C_{j,1,l}}:2\le j\le n\}$ is finite, we can choose $\epsilon_2(1)>0$ ($\epsilon_2(j)=0$ for $j\ne 1$) such that the values of $h_{\epsilon_2(1)}$ on $C_{j,1,l}~(j\ne 1)$ differ from those on $C_{1,1,l}$. 
    
    Therefore, with this choice of $\epsilon_1(k)~(1\le k\le n)$, for the new $2$-valued closed $1$-form $v'$,the singular leaf $L_j'$ of $\ker v'$ containing $\Sigma_j$ is distinct from $L_1'$ for all $j \ne 1$.
    
    By induction on $j = 2, \dots, n$, we can choose $\epsilon_2(j) > 0$ sequentially so that in the final $v'$, each singular leaf of $\ker v'$ contains at most one component of $\Zl$. By a similar argument, we can choose $\epsilon_2(j)~(1\le j\le n)$ appropriately, such that if a singular leaf $L$ of $\ker v'$ contains a component of $\Zl$, then $L\cap I_{v'}=\varnothing$.

    Next we construct a metric $g'$ on $M$, depending on the parameters $\epsilon_2(j)~(1\le j\le n)$, such that $v'$ is harmonic with respect to $g'$. As in Theorem \ref{singulartrans}, we find a candidate for $\star_{g'}v'$, and then use it to determine the metric $g'$.

    Our candidate is $\star_{g}v$. Near the boundary of $\cup_{j,l}B_{j,l}$, $v'=v$, and hence $v'\wedge \star_{g}v=v\wedge star_{g}v>0$. And in the interior of $B_{j,l}$, if $\epsilon_2(j)$ is sufficiently small, then $v'$ is close to $v$ for any $C^{\ell,\alpha}$-norm (with $\ell\ge 1$ fixed), and hence $v'\wedge \star_{g}v$ close to $v\wedge \star_{g}v$. As a result, $v'\wedge \star_{g}v>0$ on $\cup_{j,l}B_{j,l}$.

    Similar to the Step $6$ before Theorem \ref{normalfixdense}, when $\epsilon_2(j)~(1\le j\le n)$ are sufficiently small, we can define a metric $g'$ that coincides with $g$ outside $\cup_{j,l}B_{j,l}$, such that $\star_{g'}v'=\star_{g}v$. Moreover, $g-g'$ is supported in $\cup_{j,l}B_{j,l}$, and $||g-g'||_{C^{\ell,\alpha}}$ can be made arbitrarily small as $\epsilon_2(j)\to 0$. 
\end{proof}

Next, we consider closed $3$-manifolds with vanishing first Betti number, i.e., rational homology spheres.
\begin{proposition}\label{finitetreeonRHS}
    Let $M^3$ be a rational homology sphere, and let $v = v(g, \Zl, [h])$ be a rational, transverse, and nondegenerate $\mathbb{Z}/2$-harmonic $1$-form on $M$.
    Then the leaf space $\T_{M,v}$ on $M$ is a finite tree.
\end{proposition}
\begin{proof}
    By Theorem \ref{simplicialgraph}, $\T_{M,v}$ is a finite graph, with edges corresponding to the connected components of $M \setminus \mathrm{G}_v$. We construct an embedding $\iota:\T_{M,v} \to M$ as follows.
    
    First, map each vertex $[x] \in \T_{M,v}$ to a point $x \in [x] \subset M$. Second, recall that an edge between vertices $[x_1]$ and $[x_2]$
    corresponds to a component $W_i \subset M \setminus \mathrm{G}_v$. Map this edge to an arc $\gamma \subset \overline{W}_i$ such that $\gamma(0) = x_1$, $\gamma(1) = x_2$, and $|v(\dot{\gamma}(t))| > 0$ for all $t$.

    This defines a continuous embedding $\iota:\T_{M,v} \to M$ such that the composition $p_v \circ \iota: \T_{M,v} \to \T_{M,v}$
    is homotopic to the identity map. Therefore, the induced map on homology $\iota_* \colon H_1(\T_{M,v}; \mathbb{Q}) \to H_1(M; \mathbb{Q})$
    is injective.

    Since $M$ is a rational homology sphere, $H_1(M; \mathbb{Q}) \cong 0$, so $H_1(\T_{M,v}; \mathbb{Q}) = 0$. Hence, $\T_{M,v}$ is a finite graph with no cycles, i.e., a finite tree.
\end{proof}

\begin{lemma}\label{1bdry1knot}
    Let $v = v(g, \Zl, [h])$ be a nondegenerate $\mathbb{Z}/2$-harmonic $1$-form on a closed $3$-manifold $M$.
    Then there exists a rational, transverse, and nondegenerate $\mathbb{Z}/2$-harmonic $1$-form $v'=v'(g',\Zl',[h])$
    such that each boundary vertex of the leaf space $\T_{M,v'}$, viewed as a singular leaf of $\ker v'$, contains exactly one connected component of $\Zl'$ and no isolated zeros of $v'$.
\end{lemma}
\begin{proof}
    The existence of a rational, transverse, and nondegenerate $\mathbb{Z}/2$-harmonic $1$-form $v'=v'(g',\Zl',[h])$, such that each vertex of $\T_{M,v'}$, when viewed as a singular leaf of $\ker v'$, contains at most one connected component of $\Zl'$, was established in Theorem~\ref{distinguish}. It remains to show that each boundary vertex of the tree $\T_{M,v'}$ contains \emph{exactly} one component of $\Zl'$.

    Let $[x]$ be a boundary point of $\T_{M,v'}$. For $\epsilon>0$ sufficiently small, the preimage of the closed neighborhood $\overline{B}_{\epsilon}([x])=\{[y]\in\T_{M,v'}:d_{v'}([y],[x])\le \epsilon\}\subset \T_{M,v'}$, under the projection $p_{v'}:M\to \T_{M,v'}$, is a submanifold $W_{\epsilon}=p_{v'}^{-1}(\overline{B}_{\epsilon}([x]))\subset M$ with connected boundary $\partial W_{\epsilon}=p^{-1}_{v'}(\partial B_{\epsilon}([x]))$. Let $L_x$ be the singular leaf corresponding to $[x]$.
    
    Suppose, for contradiction, that $\Zl' \cap L_x = \varnothing$. Consider the $2$-fold branched covering $p':\hat{M}'\to M$, branched along $\Zl'$. Since $\Zl' \cap L_x = \varnothing$, for $\epsilon>0$ sufficiently small, we have $W_\epsilon\cap \Zl'=\varnothing$, and hence $(p')^{-1}(W_{\epsilon})$ has two connected components in $\hat{M}'$, say $(p')^{-1}(W_{\epsilon})=W_\epsilon^+\sqcup W_\epsilon^-$. 

    The pullback $1$-form $\hat{v}'=(p')^\ast v'$ is well-defined. Thus $v|_{W_\epsilon}=\pm \hat{v}'|_{W^+_\epsilon}=\hat{v}'|_{W^\pm_\epsilon}$.
    On the other hand, by assumption $\hat{v}'$ is rational, there exists a map $u:\hat{M}'\to S^1$ such that $u^\ast d\theta=\mu\hat{v}'$ for some $\mu\in\QQ$. Therefore, the image of $W_\epsilon^+$ under the map $u$ has length $|u(W_{\epsilon}^+)|=|\mu|\cdot\epsilon$ on $S^1$. Consequently, for $\epsilon$ sufficiently small, $\hat{v}'|_{W^\pm_\epsilon}$ is exact.  

    It follows that $u$ defines a Morse function $\underline{u}:(W_{\epsilon},\partial W_{\epsilon})\to ([0,\epsilon],\{\epsilon\})$, and $\underline{u}^{-1}(0)=L_x$. 
    Since $\partial W_\epsilon$ is connected, and $\underline{u}|_{\partial W_\epsilon}\equiv \epsilon$, there must exist a critical point of $\underline{u}$ with Morse index $0$. However, critical points of $\underline{u}$ are exactly the isolated zeros of $v'$, none of which has Morse index $0$ or $3$, leading to a contradiction.
\end{proof}

Before proving the pruning result for nondegenerate $\ZT$ harmonic $1$-form, let's recall the intrinsic characterization of harmonic (single-valued) $1$-forms on closed oriented manifolds. 

Let $\omega$ be a closed (single-valued) $1$-form on some closed oriented manifold $N$. Denote the zero set of $\omega$ by $S_\omega$.

\begin{definition}[\cite{intrinsic}]
    A closed $1$-form $\omega$ is called \textbf{transitive} if for any point $x\in N\setminus S_\omega$ there is a closed $\omega$-positive smooth path $\gamma:S^1\to N$ through $x$. Here, ``$\omega$-positive" means that $\omega(\dot{\gamma}(t))>0$ for all $t\in S^1$.
\end{definition}

\begin{definition}[\cite{intrinsic}]\label{localharmonic}
    A closed $1$-form is called \textbf{locally intrinsically harmonic} if there exists an open neighborhood $U$ of its zero set $S_{\omega}$ and a Riemannian metric $g_U$ on $U$ which makes the restriction $\omega|_U$ being harmonic, i.e. $d\star_{g|_U}\omega=0$. 
\end{definition}

\begin{lemma}\label{omegapositive}
    Let $x,y\in N\setminus S_{\omega}$ and suppose that there is an $\omega$-positive arc $\gamma$ from $x$ to $y$, i.e., $\gamma(0)=x$, $\gamma(1)=y$ ad $\omega(\dot{\gamma}(t))>0$ for all $0\le t\le 1$. Let $L_x$ and $L_y$ denote the leaves of $\ker \omega$ passing through $x$ and $y$, respectively. Then, for any $x' \in L_x$ and $y' \in L_y$, there exists an $\omega$-positive arc $\gamma'$ from $x'$ to $y'$.
\end{lemma}
\begin{proof}
    Let $\gamma_1$ be an arc from $x'$ to $x$ contained in $L_x$, let $\gamma_2$ be a portion of $\gamma$ from $\gamma(\epsilon)$ to $\gamma(1-\epsilon)$ for some $\epsilon>0$ sufficiently small, and let $\gamma_3$ be an arc from $y$ to $y'$ contained in $L_y$.

    There exists a closed neighborhood $U_\epsilon$ of $\gamma_1$ that is diffeomorphic to $Q_{\epsilon}\times[0,1]$, where $Q_\epsilon\{(s_1,s_2):-\epsilon\le s_1,s_2\le \epsilon\}$ is a square. Under this diffeomorphism, $\omega$ is identified with $ds_2$. We can assume that $\gamma_1$ lies in $\{(s_1,0):-\epsilon\le s_1\le\epsilon\}\times [0,1]$. Moreover, we may assume $\gamma_1(\epsilon)$ is contained in the segment $\{(s_1,\epsilon):-\epsilon\le s_1\le\epsilon\}\times\{1\}$. 

    Since $\omega = ds_2$ in this chart, and $\gamma_1$ moves from $s_2 = 0$ to $s_2 = \epsilon$, we can perturb $\gamma_1$ slightly within $U_\epsilon$ to obtain a smooth arc $\gamma_1'$ from $x'$ to $x$ such that $\omega(\dot{\gamma}_1'(t)) > 0$ for all $t$.

    Similar for $\gamma_2$. Then the arc $\gamma_1' \ast \gamma_2 \ast \gamma_3'$
    is a piecewise smooth $\omega$-positive arc from $x'$ to $y'$.

    A small smoothing near the corners preserves the strict positivity of $\omega(\dot{\gamma})$,
    yielding a smooth $\omega$-positive arc $\gamma'$ from $x'$ to $y'$.
\end{proof}

\begin{lemma}\label{omegapositiveloopforleaf}
    Let $x\in N\setminus S_{\omega}$ and suppose that there is an $\omega$-positive loop $\gamma$ passing through $x$. Then for any other point $x'\in L_x$, there exists an $\omega$-positive loop $\gamma'$ passing through $x'$. 
\end{lemma}
\begin{proof}
    This is a direct consequence of Lemma \ref{omegapositive}. One only need to modify $\gamma$ near a small neighborhood of $L_x$, to obtain the desired loop $\gamma'$.
\end{proof}

\begin{definition}
    A closed $1$-form $\omega$ is called \textbf{intrinsically harmonic} if there exists a Riemannian metric $g$ on $N$ such that $\omega$ is harmonic under $g$, i.e. $d\star_g\omega=0$. 
\end{definition}

\begin{proposition}{\cite[Theorem 4]{intrinsic}}\label{intrinsicallyharmoniccriterion}
    For a closed $1$-form $\omega$ to be intrinsically harmonic it is necessary and sufficient that:
    \begin{enumerate}[label=(\roman*)]
        \item $\omega$ is locally intrinsically harmonic.
        \item $\omega$ is transitive.
    \end{enumerate}
\end{proposition}

Consider a singular locus $\Zl'\subset M$. Let $p':\hat{M}'\to M$ be the $2$-fold branched cover, branched along $\Zl'$. Suppose $b_1(\hat{M}')>0$, and let $\hat{\omega}$ be a closed $1$-form on $\hat{M}'$, that is anti-invariant under the involution $\tau'$ on $\hat{M}'$. 

In the proof of the sufficiency part of Proposition \ref{intrinsicallyharmoniccriterion}, we do not require the metric to be smooth,
as provided by the locally intrinsically harmonic condition (Definition~\ref{localharmonic}). Thus, by a similar argument, we can prove the following:
\begin{proposition}\label{Z2intrinsic}
    Suppose $\hat{\omega}$ satisfying:
    \begin{enumerate}[label=(\roman*)]
        \item Every zero $x$ of $\hat{\omega}$ on $\hat{M}'\setminus\Zl'$ is of Morse type, i.e. $\hat{\omega}=df$ for some Morse function $f$ near $x$. Additionally, we assume every zero $x$ has Morse index $1$ or $2$
        \item There is an open neighborhood $U$ of $\Zl'$ in $M$ and a smooth Reimannian metric $g|_{U}$, such that with respect to the pullback metric $(p')^\ast g|_U$, the $1$-form $\hat{\omega}$ is harmonic on $p^{-1}(U)$.
        \item $\hat{\omega}$ is transitive.
    \end{enumerate}
    Then there is a (non-smooth) Riemannian metric $\hat{g}$ on $\hat{M}'$ such that $d\star_{\hat{g}}\hat{\omega}=0$. Moreover, $\hat{g}$ can be chosen to be $\tau'$-invariant, and to coincide with $(p')^\ast g|_U$ near $\Zl'$. Thus under the pushforward (smooth) metric $g:=p'_\ast\hat{g}$, the pushforward $2$-valued $1$-form $\omega=p'_\ast\hat{\omega}$ is a $\ZT$ harmonic $1$-form on $M$.
\end{proposition}
\begin{proof}
    The conditions (i) and (ii) correspond to the locally intrinsically harmonic condition in Proposition \ref{intrinsicallyharmoniccriterion}. Thus, as proved in \cite[Theorem 4]{intrinsic}, we can construct a closed $2$-form $\psi$ on $\hat{M}'$, such that $\hat{\omega}\wedge \psi>0$ on $\hat{M}'\setminus\Zl'$, and $\psi$ coincides with $\star_{(p')^\ast g|_U}\hat{\omega}$ near $\Zl'$.

    As a result, $(\tau')^\ast \psi=-\psi$ near $\Zl'$. Consider the closed $2$-form $\psi'=\frac{1}{2}(\psi-(\tau')^\ast\psi)$, which still coincides with $\star_{(p')^\ast g|_U}\hat{\omega}$ near $\Zl'$. Note that \[\hat{\omega}\wedge\frac{1}{2}\left(\psi-(\tau')^\ast\psi\right)=\frac{1}{2}\left(\hat{\omega}\wedge\psi+\tau^\ast(\hat{\omega}\wedge \psi)\right)>0,\]since $\tau'$ preserves orientation.

    This allows us to extend $(p')^\ast g|_U$ to a $\tau'$-invariant metric $\hat{g}$, such that $\star_{\hat{g}}\hat{\omega}=\frac{1}{2}\psi'$. This completes the proof.
\end{proof}

Finally, we prove the following pruning result of singular locus:
\begin{theorem}[Pruning]\label{prunning}
    If there is a nondegenerate $\ZT$ harmonic $1$-form $v=v(g,\Zl,[h])$ on $M$, then we can find another rational, transverse and nondegenerate $\ZT$ harmonic $1$-form $v'=v'(g',\Zl',[h'])$ such that $\Zl'$ is a sublink of $\Zl$ having exactly two connected components. 
\end{theorem}
\begin{proof}
    By Theorem \ref{distinguish} and Lemma \ref{1bdry1knot}, we may assume that $v=v(g,\Zl,[h])$ is rational, transverse and nondegenerate, and that each boundary vertex of $\T_{M,v}$ contains exactly one component of $\Zl$ and no isolated zeros.

    Since the $2$-fold branched covering $p:\hat{M}\to M$ branching along $\Zl$, has a nontrivial harmonic $1$-form $\hat{v}$, it has positive first Betti number. Therefore, the number of components of $\Zl$ is at least $2$. See \cite[Theorem 5]{Haydys}.

    Consequently, there are at least two boundary vertices of $\T_{M,v}$. Let $[x_1]$ and $[x_2]$ be two such boundary vertices of $\T_{M,v}$. Then for sufficiently small $\epsilon>0$, the open balls $B_{\epsilon}([x_1])$ and $B_\epsilon([x_2])$ contain no other boundary vertices.
    
    Denote $W_i=p_v^{-1}B_\epsilon([x_i])$ for $i=1,2$. Note that for $\epsilon>0$ sufficiently small, there are no isolated zeros of $v$ in $W_i$. Consider the submanifold $W:=M\setminus (W_1\cup W_2)$ of $M$. Both $W_i$ and $W$ are connected $3$-manifolds. The boundary of $W$ has two connected components, namely $\partial W=L_1\sqcup L_2$. Moreover, $L_1,L_2$ are regular leaves of $\ker v$, such that $d_v(L_i,[x_i])=\epsilon$. See Figure \ref{treeexample}.

\begin{figure}[!h]
    \centering

\tikzset{every picture/.style={line width=0.75pt}} 

\begin{tikzpicture}[x=0.5pt,y=0.5pt,yscale=-1,xscale=1]

\draw    (141.33,152.17) -- (185.83,163.17) ;
\draw    (185.83,163.17) -- (232.83,119.17) ;
\draw    (141.33,152.17) -- (121.33,170.17) ;
\draw    (185.83,163.17) -- (195.83,192.67) ;
\draw    (114.33,106.67) -- (141.33,152.17) ;
\draw [color={rgb, 255:red, 65; green, 117; blue, 5 }  ,draw opacity=1 ] [dash pattern={on 4.5pt off 4.5pt}]  (107.83,133.17) -- (139.52,114.15) -- (145.33,110.67) ;
\draw [color={rgb, 255:red, 65; green, 117; blue, 5 }  ,draw opacity=1 ] [dash pattern={on 4.5pt off 4.5pt}]  (171.83,183.67) -- (209.67,171.58) ;
\draw    (356.33,95.67) .. controls (365.83,102.67) and (382.83,92.67) .. (381.33,83.17) ;
\draw    (360.33,97.67) .. controls (360.33,87.67) and (373.33,83.17) .. (380.5,88.5) ;
\draw    (394.33,78.17) .. controls (403.83,85.17) and (420.83,75.17) .. (419.33,65.67) ;
\draw    (398.33,80.17) .. controls (398.33,70.17) and (411.33,65.67) .. (418.5,71) ;
\draw    (343.33,105.17) .. controls (312.83,45.67) and (381.33,4.67) .. (436,59.5) ;
\draw   (471.41,248.06) .. controls (468.45,237.42) and (487.66,222.78) .. (514.32,215.36) .. controls (540.97,207.94) and (564.98,210.56) .. (567.94,221.2) .. controls (570.91,231.84) and (551.7,246.48) .. (525.04,253.9) .. controls (498.38,261.31) and (474.37,258.7) .. (471.41,248.06) -- cycle ;
\draw    (502.7,236.16) .. controls (511,244.55) and (529.34,237.31) .. (529.33,227.69) ;
\draw    (506.34,238.76) .. controls (507.89,228.88) and (521.43,226.45) .. (527.68,232.83) ;
\draw    (471.41,248.06) .. controls (482.83,305.17) and (588.83,290.17) .. (567.94,221.2) ;
\draw [color={rgb, 255:red, 74; green, 144; blue, 226 }  ,draw opacity=1 ][line width=1.5]    (369.41,63.67) .. controls (371.04,49.45) and (385.59,43.26) .. (389.65,48.01) .. controls (393.71,52.76) and (386.12,61.14) .. (380.42,62.94) ;
\draw [color={rgb, 255:red, 74; green, 144; blue, 226 }  ,draw opacity=1 ][line width=1.5]    (372.63,58.6) .. controls (382.29,61.67) and (388.35,75.05) .. (382.85,78.72) .. controls (377.35,82.39) and (371.51,74.29) .. (370.07,67.61) ;
\draw [color={rgb, 255:red, 74; green, 144; blue, 226 }  ,draw opacity=1 ][line width=1.5]    (378.06,64.02) .. controls (365.43,69.41) and (354.11,67.34) .. (353.66,61.02) .. controls (353.21,54.69) and (363.1,54.56) .. (369.01,56.75) ;
\draw [color={rgb, 255:red, 74; green, 144; blue, 226 }  ,draw opacity=1 ][line width=1.5]    (524.22,266.41) .. controls (516.59,254.3) and (524.06,240.37) .. (530.2,241.53) .. controls (536.33,242.69) and (535.67,253.97) .. (532.35,258.95) ;
\draw [color={rgb, 255:red, 74; green, 144; blue, 226 }  ,draw opacity=1 ][line width=1.5]    (523.56,260.44) .. controls (533.02,256.79) and (546.12,263.43) .. (544.13,269.73) .. controls (542.14,276.04) and (532.52,273.38) .. (527.21,269.07) ;
\draw [color={rgb, 255:red, 74; green, 144; blue, 226 }  ,draw opacity=1 ][line width=1.5]    (531.19,261.27) .. controls (524.72,273.38) and (514.6,278.86) .. (510.29,274.21) .. controls (505.97,269.57) and (513.61,263.26) .. (519.58,261.27) ;
\draw   (369.44,146.63) .. controls (364.35,136.83) and (380.36,118.42) .. (405.2,105.52) .. controls (430.04,92.63) and (454.31,90.12) .. (459.4,99.92) .. controls (464.49,109.72) and (448.48,128.12) .. (423.63,141.02) .. controls (398.79,153.92) and (374.53,156.43) .. (369.44,146.63) -- cycle ;
\draw    (443.91,211.06) .. controls (456.83,235.17) and (542.33,212.67) .. (540.44,184.2) ;
\draw  [dash pattern={on 4.5pt off 4.5pt}]  (443.91,211.06) .. controls (437.83,180.17) and (528.83,158.17) .. (540.44,184.2) ;
\draw    (369.44,146.63) .. controls (374.33,166.67) and (376.33,171.67) .. (386.33,180.17) .. controls (396.33,188.67) and (411.33,187.67) .. (443.91,211.06) ;
\draw    (459.4,99.92) .. controls (476.83,125.67) and (487.33,116.17) .. (506.33,125.17) .. controls (525.33,134.17) and (537.33,151.17) .. (540.44,184.2) ;
\draw    (417.33,116.17) .. controls (426.83,123.17) and (443.83,113.17) .. (442.33,103.67) ;
\draw    (421.33,118.17) .. controls (421.33,108.17) and (434.33,103.67) .. (441.5,109) ;
\draw    (383.33,133.17) .. controls (392.83,140.17) and (409.83,130.17) .. (408.33,120.67) ;
\draw    (387.33,135.17) .. controls (387.33,125.17) and (400.33,120.67) .. (407.5,126) ;
\draw  [dash pattern={on 4.5pt off 4.5pt}]  (476.84,197.71) .. controls (485.48,205.74) and (503.51,197.74) .. (503.1,188.13) ;
\draw  [dash pattern={on 4.5pt off 4.5pt}]  (480.58,200.15) .. controls (481.72,190.22) and (495.15,187.22) .. (501.66,193.33) ;
\draw    (345.5,150.5) -- (253.33,150.66) ;
\draw [shift={(251.33,150.67)}, rotate = 359.9] [color={rgb, 255:red, 0; green, 0; blue, 0 }  ][line width=0.75]    (10.93,-3.29) .. controls (6.95,-1.4) and (3.31,-0.3) .. (0,0) .. controls (3.31,0.3) and (6.95,1.4) .. (10.93,3.29)   ;
\draw    (343.33,105.17) .. controls (363.33,128.17) and (445.83,84.67) .. (436,59.5) ;
\draw  [dash pattern={on 4.5pt off 4.5pt}]  (343.33,105.17) .. controls (335.33,80.17) and (422.83,40.17) .. (436,59.5) ;

\draw (105,89.4) node [anchor=north west][inner sep=0.75pt]    {$\textcolor[rgb]{0.29,0.56,0.89}{x_{1}}$};
\draw (194,191.9) node [anchor=north west][inner sep=0.75pt]    {$\textcolor[rgb]{0.29,0.56,0.89}{x_{2}}$};
\draw (228,103.4) node [anchor=north west][inner sep=0.75pt]    {$x_{4}$};
\draw (107.5,165.9) node [anchor=north west][inner sep=0.75pt]    {$x_{3}$};
\draw (107,112.9) node [anchor=north west][inner sep=0.75pt]  [font=\footnotesize]  {$\textcolor[rgb]{0.25,0.46,0.02}{\epsilon }$};
\draw (183,183.4) node [anchor=north west][inner sep=0.75pt]  [font=\footnotesize]  {$\textcolor[rgb]{0.25,0.46,0.02}{\epsilon }$};
\draw (310,49.9) node [anchor=north west][inner sep=0.75pt]    {$\textcolor[rgb]{0.29,0.56,0.89}{\Sigma _{1}}$};
\draw (461.5,272.4) node [anchor=north west][inner sep=0.75pt]    {$\textcolor[rgb]{0.29,0.56,0.89}{\Sigma }\textcolor[rgb]{0.29,0.56,0.89}{_{2}}$};
\draw (454,148.9) node [anchor=north west][inner sep=0.75pt]    {$W$};
\draw (459,80.4) node [anchor=north west][inner sep=0.75pt]    {$L_{1}$};
\draw (545.5,172.9) node [anchor=north west][inner sep=0.75pt]    {$L_{2}$};
\draw (286.5,127.4) node [anchor=north west][inner sep=0.75pt]    {$p_{v}$};
\draw (430,28.4) node [anchor=north west][inner sep=0.75pt]    {$W_{1}$};
\draw (578,230.9) node [anchor=north west][inner sep=0.75pt]    {$W_{2}$};

\end{tikzpicture}

    \caption{The finite tree $(\T_{M,v},d_v)$ and the projection $p_v:M\to \T_{M,v}$.}
    \label{treeexample}
\end{figure}
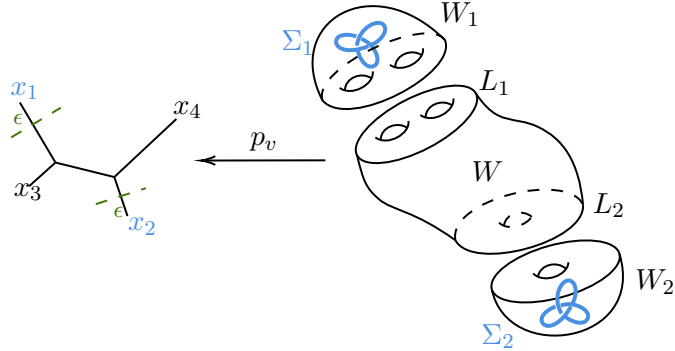

    For $\delta>0$, let $U_\delta(L_i)=\{x\in W:\dist_{g}(x,L_i)<\delta\}$ be a collar neighborhood of $L_i$, and denote $U_\delta=U_\delta(L_1)\cup U_\delta(L_2)$. Let $f_0$ be a Morse function on $W$ satisfying:
    \begin{enumerate}[label=(\roman*)]
        \item $f_0|_{L_i}\equiv i$ for $i=1,2$.
        \item There is a $\delta>0$ sufficiently small, such that $\pm df_0=v|_{U_\delta}$.
        \item All critical points of $f_0$ are contained in $W\setminus U_\delta$, and they have Morse index $1$ or $2$.
    \end{enumerate}
    The existence of such a Morse function follows from classical Morse theory. See, for example, \cite{MilnorHcob}. 
    
    There exists a piecewise smooth curve $\gamma_0$ composed of gradient flow lines of $f_0$. Then, in a Morse coordinate chart around a critical point of $f_0$ that contained in $\gamma_0$, we can modify $\gamma_0$ to obtain a smooth arc $\gamma$ such that $\gamma(0)\in L_1$, $\gamma(1)\in L_2$ and that $f_0\circ\gamma(t)$ is strictly increasing in $t$.

    We now define a $2$-valued closed $1$-form $v'$ on $M$ by:
    \begin{align}\label{finalv1}
        v'(x)=\begin{cases}
            v(x),& x\notin W,\\
            \pm df_0(x),&x\in W.
        \end{cases}
    \end{align}
    To be more precise, we need to pass to the $2$-fold branched cover $p':\hat{M}'\to M$, branched along $\Sigma_1\sqcup\Sigma_2=:\Zl'\subset\Zl$. For each $i=1,2$, the preimage $(p')^{-1}(W_i)=:\hat{W}_i$ is connected in $\hat{M}'$. The pullback $(p')^\ast v$ is a well-defined odd $1$-form on $\hat{W}_i$, coinciding with $\hat{v}=p^\ast v$. On the other hand, $(p')^{-1}(W)=:\hat{W}$ has two components, namely $\hat{W}=W_+\sqcup W_-$. Denote the boundary components of $W_\pm$ by $L^\pm_i~(i=1,2)$.

    By possibly changing the sign of $(p')^\ast v$ on $\hat{W}_i$, the following is well-defined on $\hat{W}'$:
    \begin{align}\label{finalv1}
        \hat{v}'(x)=\begin{cases}
            \hat{v}(x),& x\in \hat{W}_i,\\
            df_0(x),&x\in W_+,\\
            -df_0(x),&x\in W_-.
        \end{cases}
    \end{align}
    We assume that $\hat{v}'|_{\hat{W}_1}$ is pointing outward on $L_1^+$, and pointing inward on $L_1^-$. For $\hat{v}'$ to be well-defined, we need to set $\hat{v}'|_{\hat{W}_2}$ to be pointing outward on $L_2^-$, and pointing inward on $L_2^+$. See Figure \ref{branchedhatMprime}.

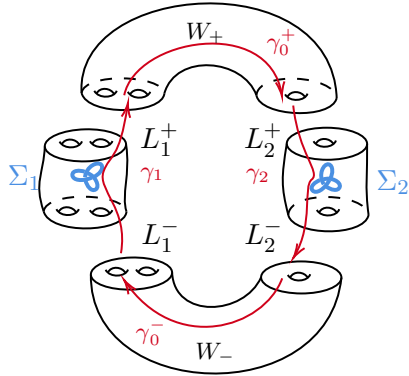
\begin{figure}[!h]
        \centering
\tikzset{every picture/.style={line width=0.75pt}} 

\begin{tikzpicture}[x=0.5pt,y=0.5pt,yscale=-1,xscale=1]

\draw    (113.67,167.13) .. controls (116.94,173.49) and (128.87,172.65) .. (130.62,167.08) ;
\draw    (115.3,169.29) .. controls (118,163.86) and (126.28,164.93) .. (128.73,169.76) ;
\draw    (139.03,167.88) .. controls (142.3,174.24) and (154.23,173.4) .. (155.98,167.83) ;
\draw    (140.66,170.04) .. controls (143.36,164.61) and (151.64,165.68) .. (154.09,170.51) ;
\draw [color={rgb, 255:red, 74; green, 144; blue, 226 }  ,draw opacity=1 ][line width=1.5]    (135.32,143) .. controls (136.31,134.37) and (145.14,130.62) .. (147.6,133.5) .. controls (150.06,136.38) and (145.46,141.47) .. (142,142.56) ;
\draw [color={rgb, 255:red, 74; green, 144; blue, 226 }  ,draw opacity=1 ][line width=1.5]    (137.28,139.92) .. controls (143.13,141.79) and (146.81,149.91) .. (143.47,152.13) .. controls (140.14,154.35) and (136.6,149.44) .. (135.72,145.39) ;
\draw [color={rgb, 255:red, 74; green, 144; blue, 226 }  ,draw opacity=1 ][line width=1.5]    (140.57,143.21) .. controls (132.91,146.48) and (126.04,145.23) .. (125.77,141.39) .. controls (125.5,137.56) and (131.5,137.48) .. (135.08,138.8) ;
\draw    (104.04,168.78) .. controls (108.7,186.67) and (165.25,185.3) .. (166.7,168.98) ;
\draw  [dash pattern={on 4.5pt off 4.5pt}]  (104.04,168.78) .. controls (106.45,153.04) and (164.77,154.93) .. (166.7,168.98) ;
\draw   (104.9,119.63) .. controls (104.72,112.93) and (118.34,107.13) .. (135.31,106.68) .. controls (152.28,106.23) and (166.18,111.29) .. (166.36,117.98) .. controls (166.54,124.68) and (152.92,130.48) .. (135.95,130.93) .. controls (118.98,131.38) and (105.08,126.32) .. (104.9,119.63) -- cycle ;
\draw    (139.1,115.7) .. controls (142.43,122.04) and (154.36,121.09) .. (156.06,115.51) ;
\draw    (140.76,117.85) .. controls (143.41,112.4) and (151.69,113.39) .. (154.19,118.2) ;
\draw    (116.05,115.96) .. controls (119.38,122.3) and (131.3,121.35) .. (133,115.77) ;
\draw    (117.7,118.12) .. controls (120.35,112.66) and (128.64,113.65) .. (131.13,118.46) ;
\draw    (104.9,119.63) .. controls (104.04,140.67) and (96.46,146.44) .. (104.04,168.78) ;
\draw    (166.36,117.98) .. controls (169.85,140.67) and (162.57,156.14) .. (166.7,168.98) ;
\draw    (142.35,77.51) .. controls (145.62,83.88) and (157.55,83.03) .. (159.3,77.47) ;
\draw    (143.98,79.68) .. controls (146.68,74.24) and (154.96,75.31) .. (157.41,80.14) ;
\draw    (167.71,78.26) .. controls (170.98,84.63) and (182.91,83.78) .. (184.66,78.22) ;
\draw    (169.34,80.43) .. controls (172.04,74.99) and (180.32,76.06) .. (182.77,80.89) ;
\draw    (132.72,79.16) .. controls (137.38,97.05) and (193.93,95.69) .. (195.38,79.37) ;
\draw  [dash pattern={on 4.5pt off 4.5pt}]  (132.72,79.16) .. controls (135.13,63.42) and (193.45,65.31) .. (195.38,79.37) ;
\draw   (268.52,218.7) .. controls (268.55,212) and (282.18,206.63) .. (298.96,206.7) .. controls (315.74,206.77) and (329.32,212.26) .. (329.29,218.96) .. controls (329.26,225.66) and (315.63,231.03) .. (298.85,230.96) .. controls (282.07,230.88) and (268.49,225.39) .. (268.52,218.7) -- cycle ;
\draw    (288.74,216.92) .. controls (292.2,223.18) and (304.1,221.99) .. (305.68,216.37) ;
\draw    (290.44,219.03) .. controls (292.97,213.53) and (301.28,214.34) .. (303.87,219.1) ;
\draw   (287.96,117.46) .. controls (287.99,110.76) and (301.61,105.39) .. (318.39,105.47) .. controls (335.18,105.54) and (348.76,111.03) .. (348.73,117.73) .. controls (348.7,124.43) and (335.07,129.8) .. (318.29,129.73) .. controls (301.51,129.65) and (287.93,124.16) .. (287.96,117.46) -- cycle ;
\draw    (308.18,115.69) .. controls (311.64,121.95) and (323.54,120.76) .. (325.12,115.14) ;
\draw    (309.88,117.8) .. controls (312.41,112.29) and (320.72,113.11) .. (323.31,117.87) ;
\draw    (264.98,80.88) .. controls (268.05,97.18) and (321.6,99.8) .. (325.71,82.99) ;
\draw  [dash pattern={on 4.5pt off 4.5pt}]  (264.98,80.88) .. controls (267.11,61.9) and (323.76,65.82) .. (325.71,82.99) ;
\draw    (289.08,171.45) .. controls (292.7,187.64) and (346.31,188.44) .. (349.85,171.5) ;
\draw  [dash pattern={on 4.5pt off 4.5pt}]  (289.08,171.45) .. controls (290.57,152.42) and (347.31,154.41) .. (349.85,171.5) ;
\draw [color={rgb, 255:red, 74; green, 144; blue, 226 }  ,draw opacity=1 ][line width=1.5]    (315.94,148.28) .. controls (311.31,140.93) and (315.84,132.48) .. (319.57,133.19) .. controls (323.29,133.89) and (322.89,140.73) .. (320.87,143.75) ;
\draw [color={rgb, 255:red, 74; green, 144; blue, 226 }  ,draw opacity=1 ][line width=1.5]    (315.54,144.66) .. controls (321.28,142.44) and (329.22,146.47) .. (328.02,150.29) .. controls (326.81,154.11) and (320.97,152.5) .. (317.75,149.89) ;
\draw [color={rgb, 255:red, 74; green, 144; blue, 226 }  ,draw opacity=1 ][line width=1.5]    (320.17,145.16) .. controls (316.25,152.5) and (310.11,155.82) .. (307.49,153.01) .. controls (304.88,150.19) and (309.5,146.37) .. (313.13,145.16) ;
\draw    (287.96,117.46) .. controls (283.87,143.1) and (283.56,149.77) .. (289.08,171.45) ;
\draw    (348.73,117.73) .. controls (349.06,137.95) and (346.94,150.68) .. (349.85,171.5) ;
\draw   (139.75,216.71) .. controls (139.57,210.01) and (153.18,204.21) .. (170.15,203.76) .. controls (187.12,203.31) and (201.02,208.37) .. (201.2,215.06) .. controls (201.38,221.76) and (187.77,227.56) .. (170.8,228.01) .. controls (153.83,228.46) and (139.92,223.4) .. (139.75,216.71) -- cycle ;
\draw    (173.95,212.78) .. controls (177.27,219.12) and (189.2,218.17) .. (190.9,212.59) ;
\draw    (175.6,214.93) .. controls (178.25,209.48) and (186.53,210.47) .. (189.03,215.28) ;
\draw    (150.89,213.04) .. controls (154.22,219.38) and (166.15,218.43) .. (167.84,212.85) ;
\draw    (152.55,215.19) .. controls (155.2,209.74) and (163.48,210.73) .. (165.98,215.54) ;
\draw    (195.38,79.37) .. controls (204.29,48.63) and (254.93,53.18) .. (264.98,80.88) ;
\draw    (132.72,79.16) .. controls (127.4,-4.79) and (321.83,-18.89) .. (325.71,82.99) ;
\draw    (201.2,215.06) .. controls (202.87,250.79) and (255.63,247.76) .. (268.52,218.7) ;
\draw    (139.75,216.71) .. controls (138.48,312.55) and (318.81,319.6) .. (329.29,218.96) ;
\draw    (310.68,167.69) .. controls (314.14,173.95) and (326.04,172.76) .. (327.62,167.14) ;
\draw    (312.38,169.8) .. controls (314.91,164.29) and (323.22,165.11) .. (325.81,169.87) ;
\draw    (286.68,80.19) .. controls (290.14,86.45) and (302.04,85.26) .. (303.62,79.64) ;
\draw    (288.38,82.3) .. controls (290.91,76.79) and (299.22,77.61) .. (301.81,82.37) ;
\draw [color={rgb, 255:red, 208; green, 2; blue, 27 }  ,draw opacity=1 ]   (162.5,200.5) .. controls (168.33,166.67) and (144.33,152.67) .. (149.33,138.67) .. controls (154.26,124.88) and (161.12,130.98) .. (167.54,88.63) ;
\draw [shift={(167.83,86.67)}, rotate = 98.31] [color={rgb, 255:red, 208; green, 2; blue, 27 }  ,draw opacity=1 ][line width=0.75]    (10.93,-3.29) .. controls (6.95,-1.4) and (3.31,-0.3) .. (0,0) .. controls (3.31,0.3) and (6.95,1.4) .. (10.93,3.29)   ;
\draw [color={rgb, 255:red, 208; green, 2; blue, 27 }  ,draw opacity=1 ]   (292.83,88.67) .. controls (296.83,125.17) and (313.83,139.17) .. (306.83,144.67) .. controls (299.97,150.06) and (308.48,177.54) .. (293.76,202.63) ;
\draw [shift={(292.83,204.17)}, rotate = 302.11] [color={rgb, 255:red, 208; green, 2; blue, 27 }  ,draw opacity=1 ][line width=0.75]    (10.93,-3.29) .. controls (6.95,-1.4) and (3.31,-0.3) .. (0,0) .. controls (3.31,0.3) and (6.95,1.4) .. (10.93,3.29)   ;
\draw [color={rgb, 255:red, 208; green, 2; blue, 27 }  ,draw opacity=1 ]   (164.98,80.88) .. controls (173.74,38.1) and (274.29,20.52) .. (284.07,79.85) ;
\draw [shift={(284.33,81.67)}, rotate = 262.59] [color={rgb, 255:red, 208; green, 2; blue, 27 }  ,draw opacity=1 ][line width=0.75]    (10.93,-3.29) .. controls (6.95,-1.4) and (3.31,-0.3) .. (0,0) .. controls (3.31,0.3) and (6.95,1.4) .. (10.93,3.29)   ;
\draw [color={rgb, 255:red, 208; green, 2; blue, 27 }  ,draw opacity=1 ]   (284.5,219.5) .. controls (270.4,257.48) and (203.51,277.47) .. (165.41,221.52) ;
\draw [shift={(164.83,220.67)}, rotate = 56.31] [color={rgb, 255:red, 208; green, 2; blue, 27 }  ,draw opacity=1 ][line width=0.75]    (10.93,-3.29) .. controls (6.95,-1.4) and (3.31,-0.3) .. (0,0) .. controls (3.31,0.3) and (6.95,1.4) .. (10.93,3.29)   ;

\draw (74.77,131.89) node [anchor=north west][inner sep=0.75pt]    {$\textcolor[rgb]{0.29,0.56,0.89}{\Sigma }\textcolor[rgb]{0.29,0.56,0.89}{_{1}}$};
\draw (353.63,135.3) node [anchor=north west][inner sep=0.75pt]    {$\textcolor[rgb]{0.29,0.56,0.89}{\Sigma }\textcolor[rgb]{0.29,0.56,0.89}{_{2}}$};
\draw (209.26,23.48) node [anchor=north west][inner sep=0.75pt]  [font=\footnotesize]  {$W_{+}$};
\draw (216.62,265.74) node [anchor=north west][inner sep=0.75pt]  [font=\footnotesize]  {$W_{-}$};
\draw (174.9,99.64) node [anchor=north west][inner sep=0.75pt]    {$L_{1}^{+}$};
\draw (175.7,172.5) node [anchor=north west][inner sep=0.75pt]    {$L_{1}^{-}$};
\draw (253.73,172.41) node [anchor=north west][inner sep=0.75pt]    {$L_{2}^{-}$};
\draw (253.56,100.05) node [anchor=north west][inner sep=0.75pt]    {$L_{2}^{+}$};
\draw (175,133.4) node [anchor=north west][inner sep=0.75pt]  [font=\footnotesize]  {$\textcolor[rgb]{0.82,0.01,0.11}{\gamma _{1}}$};
\draw (255,134.4) node [anchor=north west][inner sep=0.75pt]  [font=\footnotesize]  {$\textcolor[rgb]{0.82,0.01,0.11}{\gamma _{2}}$};
\draw (270.5,28.4) node [anchor=north west][inner sep=0.75pt]  [font=\footnotesize]  {$\textcolor[rgb]{0.82,0.01,0.11}{\gamma _{0}^{+}}$};
\draw (171.5,247.9) node [anchor=north west][inner sep=0.75pt]  [font=\footnotesize]  {$\textcolor[rgb]{0.82,0.01,0.11}{\gamma }\textcolor[rgb]{0.82,0.01,0.11}{_{0}^{-}}$};

\end{tikzpicture}

     \caption{A decomposition of the branched cover $\hat{M}'$.}
        \label{branchedhatMprime}
    \end{figure}
    
    Now we have a closed $1$-form $\hat{v}'$ that is anti-invariant under the involution $\tau'$ on $\hat{M}'$. And by construction, $\hat{v}'$ satisfies conditions (i) and (ii) in Proposition \ref{Z2intrinsic}. We will prove the transitive condition (iii) for $\hat{v}'$ in the following. It will then follow that there exists a smooth metric $g'$ on $M$, that coincides with $g$ near $\Zl'=\Sigma_1\sqcup \Sigma_2$, such that $v'$ is a transverse and nondegenerate $\ZT$ harmonic $1$-form on $(M,g')$. 

    First, let's show that in the manifold $\hat{W}_1$, there is a $\hat{v}'$-positive arc $\gamma_1$ from $L_1^-$ to $L_1^+$. 
    
    By the nondegeneracy of $v$, there is a tubular neighborhood $U_1\subset W_1$ of $\Sigma_1$, and a diffeomorphism $\phi_1:S^1\times D^2\to U_1$, such that $\phi^\ast v=d\re z^{3/2}$.

    The set $\{z\in D^2:\re z^{3/2}=0\}$ divides $D^2$ into three parts. Thus $\underline{U}_1:=U_1\setminus \phi\left(\{(\theta,z)\in S^1\times D^2:\re z^{3/2}=0\}\right)$ has three connected components, namely $U_1^1,U_1^2$ and $U_1^3$. 
    
    In the $2$-fold branched cover $\hat{M}'$, $\hat{\underline{U}}_1=(p')^{-1}(\underline{U}_1)$ has six components. On the slice of $S^1\times D^2$, we arrange them in the following order (Figure \ref{slice6}):

    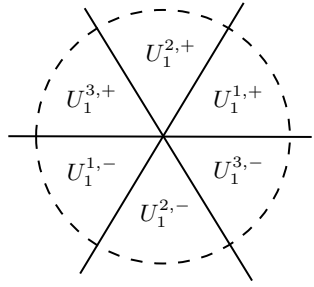
\begin{figure}[!h]
        \centering
\tikzset{every picture/.style={line width=0.75pt}} 

\begin{tikzpicture}[x=0.75pt,y=0.75pt,yscale=-1,xscale=1]

\draw  [dash pattern={on 4.5pt off 4.5pt}] (109.35,133.22) .. controls (109.35,97.87) and (138,69.22) .. (173.34,69.22) .. controls (208.68,69.22) and (237.33,97.87) .. (237.33,133.22) .. controls (237.33,168.56) and (208.68,197.21) .. (173.34,197.21) .. controls (138,197.21) and (109.35,168.56) .. (109.35,133.22) -- cycle ;
\draw    (95.33,132.95) -- (248.33,133.33) ;
\draw    (131.56,202.28) -- (214.04,65.67) ;
\draw    (133.83,68.36) -- (218.25,205.83) ;

\draw (197.06,105.82) node [anchor=north west][inner sep=0.75pt]  [font=\footnotesize]  {$U_{1}^{1,+}$};
\draw (162.95,82.65) node [anchor=north west][inner sep=0.75pt]  [font=\footnotesize]  {$U_{1}^{2,+}$};
\draw (123.02,105.7) node [anchor=north west][inner sep=0.75pt]  [font=\footnotesize]  {$U_{1}^{3,+}$};
\draw (123.31,142.49) node [anchor=north west][inner sep=0.75pt]  [font=\footnotesize]  {$U_{1}^{1,-}$};
\draw (160.22,163.54) node [anchor=north west][inner sep=0.75pt]  [font=\footnotesize]  {$U_{1}^{2,-}$};
\draw (197.15,141.81) node [anchor=north west][inner sep=0.75pt]  [font=\footnotesize]  {$U_{1}^{3,-}$};

\end{tikzpicture}

        \caption{Components of $\underline{\hat{U}}_1$.}
        \label{slice6}
    \end{figure}

    Recall that $L_1$ is a regular leaf of $\ker v$ on $M$ and that $W_1$ contains no isolated zeros of $v$. Therefore, each leaf of $\ker v$ on $W_1$ is diffeomorphic to $L_1$. 
    
    Consider the dual vector field $\hat{V}$ of $\hat{v}$ on $\hat{W}_1$, then there exists three simple closed curves $\Sigma_1^{i,-}~(i=1,2,3)$ on $L_1^-$, such that under the flow generated by $\hat{V}$, each $\Sigma_1^{i,-}$ is sent into $U_1^{i,-}$, and as the time tends to infinity, they all converge to $\Sigma_1$. Analogously, we define three simple closed curves $\Sigma_1^{i,+}~(i=1,2,3)$ on $L_1^+$, so that they are sent into $U_1^{i,+}$ under the flow generated by $-\hat{V}$.

    Now, there is a piecewise smooth arc $\gamma'$ in $\hat{W}_1$, composed of flow lines of $\hat{V}$, that runs from $\Sigma_1^{3,-}$ to $\Sigma_1$, and then from $\Sigma_1$ to $\Sigma_1^{1,+}$. Then we can modify $\gamma'$ near $\Sigma_1$ to obtain a $\hat{v}$-positive (and hence $\hat{v}'$-positive) smooth arc $\gamma_1$ from $L_1^-$ to $L_1^+$. See Figure \ref{tranarcnearSigma1}, in where the blue arcs represent the regular leaf of $\ker\hat{v}$ near $\Sigma_1$.

\begin{figure}[!h]
        \centering
\tikzset{every picture/.style={line width=0.75pt}} 

\begin{tikzpicture}[x=0.75pt,y=0.75pt,yscale=-1,xscale=1]

\draw  [dash pattern={on 4.5pt off 4.5pt}] (49.35,123.72) .. controls (49.35,88.37) and (78,59.72) .. (113.34,59.72) .. controls (148.68,59.72) and (177.33,88.37) .. (177.33,123.72) .. controls (177.33,159.06) and (148.68,187.71) .. (113.34,187.71) .. controls (78,187.71) and (49.35,159.06) .. (49.35,123.72) -- cycle ;
\draw    (35.33,123.45) -- (188.33,123.83) ;
\draw    (71.56,192.78) -- (154.04,56.17) ;
\draw    (73.83,58.86) -- (158.25,196.33) ;
\draw   (221,47.5) .. controls (221,45.01) and (223.01,43) .. (225.5,43) .. controls (227.99,43) and (230,45.01) .. (230,47.5) .. controls (230,49.99) and (227.99,52) .. (225.5,52) .. controls (223.01,52) and (221,49.99) .. (221,47.5) -- cycle ; \draw   (222.32,44.32) -- (228.68,50.68) ; \draw   (228.68,44.32) -- (222.32,50.68) ;
\draw [color={rgb, 255:red, 208; green, 2; blue, 27 }  ,draw opacity=1 ]   (113.34,123.72) .. controls (196.49,115.09) and (171.84,87.56) .. (221.8,51.11) ;
\draw [shift={(223.33,50)}, rotate = 144.57] [color={rgb, 255:red, 208; green, 2; blue, 27 }  ,draw opacity=1 ][line width=0.75]    (10.93,-3.29) .. controls (6.95,-1.4) and (3.31,-0.3) .. (0,0) .. controls (3.31,0.3) and (6.95,1.4) .. (10.93,3.29)   ;
\draw [color={rgb, 255:red, 208; green, 2; blue, 27 }  ,draw opacity=1 ]   (214.33,194.5) .. controls (170.77,190.54) and (160.54,154.73) .. (117.64,124.9) ;
\draw [shift={(116.33,124)}, rotate = 34.29] [color={rgb, 255:red, 208; green, 2; blue, 27 }  ,draw opacity=1 ][line width=0.75]    (10.93,-3.29) .. controls (6.95,-1.4) and (3.31,-0.3) .. (0,0) .. controls (3.31,0.3) and (6.95,1.4) .. (10.93,3.29)   ;
\draw [color={rgb, 255:red, 74; green, 144; blue, 226 }  ,draw opacity=0.84 ]   (163.83,190) .. controls (126.83,131) and (128.83,130) .. (200.33,130) ;
\draw [color={rgb, 255:red, 74; green, 144; blue, 226 }  ,draw opacity=1 ]   (160.5,60.33) .. controls (117.33,124.5) and (124.33,117) .. (195.33,116) ;
\draw   (214.33,194.5) .. controls (214.33,192.01) and (216.35,190) .. (218.83,190) .. controls (221.32,190) and (223.33,192.01) .. (223.33,194.5) .. controls (223.33,196.99) and (221.32,199) .. (218.83,199) .. controls (216.35,199) and (214.33,196.99) .. (214.33,194.5) -- cycle ; \draw   (215.65,191.32) -- (222.02,197.68) ; \draw   (222.02,191.32) -- (215.65,197.68) ;
\draw  [dash pattern={on 4.5pt off 4.5pt}] (318.85,121.72) .. controls (318.85,86.37) and (347.5,57.72) .. (382.84,57.72) .. controls (418.18,57.72) and (446.83,86.37) .. (446.83,121.72) .. controls (446.83,157.06) and (418.18,185.71) .. (382.84,185.71) .. controls (347.5,185.71) and (318.85,157.06) .. (318.85,121.72) -- cycle ;
\draw    (304.83,121.45) -- (457.83,121.83) ;
\draw    (341.06,190.78) -- (423.54,54.17) ;
\draw    (343.33,56.86) -- (427.75,194.33) ;
\draw   (490.5,45.5) .. controls (490.5,43.01) and (492.51,41) .. (495,41) .. controls (497.49,41) and (499.5,43.01) .. (499.5,45.5) .. controls (499.5,47.99) and (497.49,50) .. (495,50) .. controls (492.51,50) and (490.5,47.99) .. (490.5,45.5) -- cycle ; \draw   (491.82,42.32) -- (498.18,48.68) ; \draw   (498.18,42.32) -- (491.82,48.68) ;
\draw [color={rgb, 255:red, 208; green, 2; blue, 27 }  ,draw opacity=1 ]   (483.83,192.5) .. controls (386.33,128) and (368.33,125) .. (491.33,48) ;
\draw [shift={(491.33,48)}, rotate = 147.95] [color={rgb, 255:red, 208; green, 2; blue, 27 }  ,draw opacity=1 ][line width=0.75]    (10.93,-3.29) .. controls (6.95,-1.4) and (3.31,-0.3) .. (0,0) .. controls (3.31,0.3) and (6.95,1.4) .. (10.93,3.29)   ;
\draw [color={rgb, 255:red, 74; green, 144; blue, 226 }  ,draw opacity=0.84 ]   (433.33,188) .. controls (396.33,129) and (398.33,128) .. (469.83,128) ;
\draw [color={rgb, 255:red, 74; green, 144; blue, 226 }  ,draw opacity=1 ]   (430,58.33) .. controls (386.83,122.5) and (393.83,115) .. (464.83,114) ;
\draw   (483.83,192.5) .. controls (483.83,190.01) and (485.85,188) .. (488.33,188) .. controls (490.82,188) and (492.83,190.01) .. (492.83,192.5) .. controls (492.83,194.99) and (490.82,197) .. (488.33,197) .. controls (485.85,197) and (483.83,194.99) .. (483.83,192.5) -- cycle ; \draw   (485.15,189.32) -- (491.52,195.68) ; \draw   (491.52,189.32) -- (485.15,195.68) ;
\draw [line width=1.5]    (213,121.5) -- (286.83,121.5) ;
\draw [shift={(289.83,121.5)}, rotate = 180] [color={rgb, 255:red, 0; green, 0; blue, 0 }  ][line width=1.5]    (14.21,-4.28) .. controls (9.04,-1.82) and (4.3,-0.39) .. (0,0) .. controls (4.3,0.39) and (9.04,1.82) .. (14.21,4.28)   ;

\draw (142.06,93.32) node [anchor=north west][inner sep=0.75pt]  [font=\footnotesize]  {$U_{1}^{1,+}$};
\draw (102.95,73.15) node [anchor=north west][inner sep=0.75pt]  [font=\footnotesize]  {$U_{1}^{2,+}$};
\draw (63.02,96.2) node [anchor=north west][inner sep=0.75pt]  [font=\footnotesize]  {$U_{1}^{3,+}$};
\draw (63.31,132.99) node [anchor=north west][inner sep=0.75pt]  [font=\footnotesize]  {$U_{1}^{1,-}$};
\draw (100.22,154.04) node [anchor=north west][inner sep=0.75pt]  [font=\footnotesize]  {$U_{1}^{2,-}$};
\draw (143.65,126.81) node [anchor=north west][inner sep=0.75pt]  [font=\footnotesize]  {$U_{1}^{3,-}$};
\draw (239.5,41.4) node [anchor=north west][inner sep=0.75pt]  [font=\footnotesize]  {$\Sigma _{1}^{1,+}$};
\draw (231.5,190.9) node [anchor=north west][inner sep=0.75pt]  [font=\footnotesize]  {$\Sigma _{1}^{3,-}$};
\draw (411.56,91.32) node [anchor=north west][inner sep=0.75pt]  [font=\footnotesize]  {$U_{1}^{1,+}$};
\draw (372.45,71.15) node [anchor=north west][inner sep=0.75pt]  [font=\footnotesize]  {$U_{1}^{2,+}$};
\draw (332.52,94.2) node [anchor=north west][inner sep=0.75pt]  [font=\footnotesize]  {$U_{1}^{3,+}$};
\draw (332.81,130.99) node [anchor=north west][inner sep=0.75pt]  [font=\footnotesize]  {$U_{1}^{1,-}$};
\draw (369.72,152.04) node [anchor=north west][inner sep=0.75pt]  [font=\footnotesize]  {$U_{1}^{2,-}$};
\draw (413.15,124.81) node [anchor=north west][inner sep=0.75pt]  [font=\footnotesize]  {$U_{1}^{3,-}$};
\draw (509,39.4) node [anchor=north west][inner sep=0.75pt]  [font=\footnotesize]  {$\Sigma _{1}^{1,+}$};
\draw (501,188.9) node [anchor=north west][inner sep=0.75pt]  [font=\footnotesize]  {$\Sigma _{1}^{3,-}$};
\draw (225.5,99.5) node [anchor=north west][inner sep=0.75pt]   [align=left] {Modify};

\end{tikzpicture}

        \caption{Modify flow lines of $\hat{V}$.}
        \label{tranarcnearSigma1}
    \end{figure}
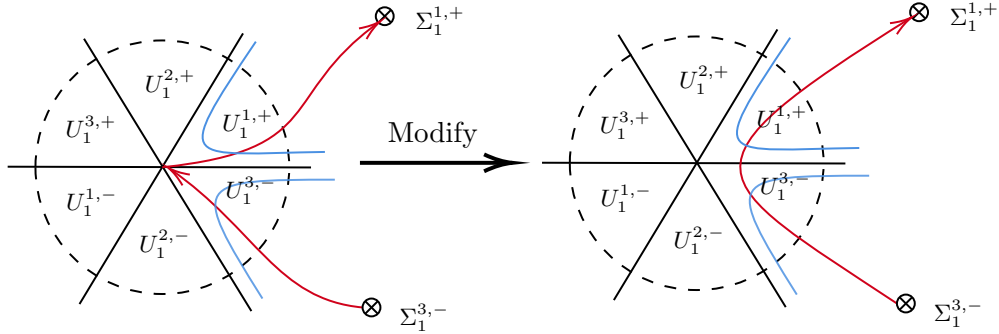

Then by Lemma \ref{omegapositive}, for any $x\in L_1^-$ and any $y\in L_1^+$, there always exists a smooth $\hat{v}'$-positive arc $\gamma_1$  from $x$ to $y$, that is contained in $\hat{W}_1$. 

Similarly, for any $x\in L_2^+$ and any $y\in L_2^-$, there always exists a smooth $\hat{v}'$-positive arc $\gamma_2$  from $x$ to $y$, that is contained in $\hat{W}_2$. 

Recall that after define the Morse function $f_0$, we found a $df_0$-positive arc $\gamma_0$ in $W$ that connecting from $L_1$ to $L_2$. This defines a $\hat{v}'$-positive arc $\gamma_0^+$ in $W_+$, from $L_1^+$ to $L_2^+$, and defines another $\hat{v}'$-positive arc $\gamma_0^+$ in $W_-$, from $L_2^-$ to $L_1^-$.

Now by Lemma \ref{omegapositive}, we can combine $\gamma_1$, $\gamma_0^+$, $\gamma_2$ and $\gamma_0^-$ together, to obtain a $\hat{v}'$-positive loop $\bar{\gamma}$ on $\hat{M}'$. See Figure \ref{branchedhatMprime}.

 Similar to Proposition \ref{finitetreeonRHS} and \ref{1bdry1knot}, we can prove that $\T_{M,v'}$ must be a finite tree, and each boundary vertex corresponds to a component of $\Zl'$. As a result, $T_{M,v'}$ must be a closed interval, and hence level sets of $f_0$ are connected.  
 Therefore, $\bar{\gamma}$ passes through every leaf of $\ker\hat{v}'$. By Lemma \ref{omegapositiveloopforleaf}, it follows that $\hat{v}'$ satisfies the transitive condition.

This proves that there exists a Riemannian metric $g'$ on $M$, such that $v'$ is a transverse and nondegenerate $\ZT$ harmonic $1$-form, with singular locus $\Zl'=\Sigma_1\sqcup\Sigma_2$. Moreover, the $\hat{v}'$-positive loop $\bar{\gamma}$ indeed defines a map $u:\hat{M}'\to S^1$, such $u^\ast d\theta=\mu\hat{v}'$ for some $\mu\in\QQ$. Thus $v'=v'(g',\Zl,[h]')$ for some $[h]'\in H^1(\hat{M}';\QQ)$, i.e., $v'$ is rational.
\end{proof}

\begin{remark}
    There are various examples of nondegenerate $\ZT$ harmonic $1$-forms on rational homology spheres, whose singular loci have more than two connected components. See \cite{Z3He,HeParker}.
\end{remark}

	\bibliographystyle{alpha}

\input{main.bbl}
\end{document}

%% file: package.tex
\usepackage{amsmath, amssymb, amscd, mathrsfs, slashed, url,color,extsizes,xcolor,multicol,indentfirst,latexsym,bm,graphicx,subfigure,esint,float,verbatim,comment,epsfig,dsfont,amsthm,amsfonts,amsbsy,tikz-cd,amsthm,thmtools,bbm,adjustbox,enumitem,calligra,fancyhdr,accents}

\usepackage[all,cmtip]{xy}
\usepackage[top=1in, bottom=1in, left=1.25in, right=1.25in]{geometry}
\usepackage[colorlinks]{hyperref}
\usepackage[toc,page]{appendix}
\setlist{nosep}
\setlist[itemize]{leftmargin=*}
\setlist[enumerate]{leftmargin=*,align=left,nolistsep}

\DeclareMathAlphabet{\mathcalligra}{T1}{calligra}{m}{n}

\declaretheoremstyle[
headfont=\color{blue}\normalfont\bfseries,
bodyfont=\color{blue}\normalfont\itshape,
]{colored}

\usepackage[utf8]{inputenc}
\usepackage[english]{babel}

%% file: newtheorem.tex
\newtheorem{theorem}{Theorem}[section]

\newtheorem{corollary}[theorem]{Corollary}

\newtheorem{definition}[theorem]{Definition}
\newtheorem{question}[theorem]{Question}

\newtheorem{lemma}[theorem]{Lemma}

\newtheorem{proposition}[theorem]{Proposition}

\theoremstyle{definition}

\newtheorem*{remark}{Remark}

%% file: convention.tex
\newcommand{\RR}{\mathbb{R}}
\newcommand{\NN}{\mathbb{N}}
\newcommand{\QQ}{\mathbb{Q}}
\newcommand{\ZZ}{\mathbb{Z}}

\newcommand{\re}{\mathrm{Re}}
\newcommand{\CC}{\mathbb{C}}
\newcommand{\SL}{\mathrm{SL}_2(\CC)}

\newcommand{\dist}{\mathrm{dist}}

\newcommand{\ZT}{\ZZ/2}
\newcommand{\vv}{\mathbf{v}}
\newcommand{\T}{\mathcal{T}}

\newcommand{\bA}{\mathbf{A}}
\newcommand{\bB}{\mathbf{B}}

\newcommand{\D}{\mathcal{D}}
\newcommand{\E}{\mathcal{E}}
\newcommand{\M}{\mathcal{M}}
\newcommand{\Zl}{\mathcal{Z}}

%% file: main.bbl
\begin{thebibliography}{DDW00}

\bibitem[Aub82]{Aubin}
Thierry Aubin.
\newblock {\em Nonlinear analysis on manifolds. {M}onge-{A}mp\`ere equations}, volume 252 of {\em Grundlehren der mathematischen Wissenschaften [Fundamental Principles of Mathematical Sciences]}.
\newblock Springer-Verlag, New York, 1982.

\bibitem[BT82]{BottTu}
Raoul Bott and Loring~W. Tu.
\newblock {\em Differential forms in algebraic topology}, volume~82 of {\em Graduate Texts in Mathematics}.
\newblock Springer-Verlag, New York-Berlin, 1982.

\bibitem[BW25]{Walpuski-Bera}
Gorapada Bera and Thomas Walpuski.
\newblock Dirac operators twisted by ramified {E}uclidean line bundles.
\newblock {\em arXiv preprint arXiv:2503.01392}, 2025.

\bibitem[Cal69]{Calabi}
Eugenio Calabi.
\newblock An intrinsic characterization of harmonic one-forms.
\newblock In {\em Global {A}nalysis ({P}apers in {H}onor of {K}. {K}odaira)}, pages 101--117. Univ. Tokyo Press, Tokyo, 1969.

\bibitem[CH24]{ChenHe}
Jiahuang Chen and Siqi He.
\newblock On the existence and rigidity of critical {$\mathbb{Z}/2$} eigenvalues.
\newblock {\em arXiv preprint arxiv:2404.05387}, 2024.

\bibitem[CHY]{ChenHeYan}
Jiahuang Chen, Siqi He, and Dashen Yan.
\newblock Calabi surgery and {W}hitehead move for {$\mathbb{Z}/2$} harmonic {$1$}-forms on {$3$}-manifolds, in progress.

\bibitem[DDW00]{DDW00}
G.~Daskalopoulos, S.~Dostoglou, and R.~Wentworth.
\newblock On the {M}organ-{S}halen compactification of the {${\rm SL}(2,{\bf C})$} character varieties of surface groups.
\newblock {\em Duke Math. J.}, 101(2):189--207, 2000.

\bibitem[Don12]{kahlercone}
S.~K. Donaldson.
\newblock K\"ahler metrics with cone singularities along a divisor.
\newblock In {\em Essays in mathematics and its applications}, pages 49--79. Springer, Heidelberg, 2012.

\bibitem[Don21]{donaldsondeformation2019}
Simon Donaldson.
\newblock Deformations of multivalued harmonic functions.
\newblock {\em Q. J. Math.}, 72(1-2):199--235, 2021.

\bibitem[EL24]{SamanLi}
Saman~Habibi Esfahani and Yang Li.
\newblock Fueter sections and {$\mathbb{Z}_2$}-harmonic 1-forms.
\newblock {\em arXiv preprint arxiv:2410.06367}, 2024.

\bibitem[Ham82]{Hamilton}
Richard~S. Hamilton.
\newblock The inverse function theorem of {N}ash and {M}oser.
\newblock {\em Bull. Amer. Math. Soc. (N.S.)}, 7(1):65--222, 1982.

\bibitem[Hay22]{Haydys}
Andriy Haydys.
\newblock Seiberg-{W}itten monopoles and flat {${\rm PSL}(2,\Bbb R)$}-connections.
\newblock {\em Adv. Math.}, 409:Paper No. 108686, 18, 2022.

\bibitem[He23]{Hebranchdefor}
Siqi He.
\newblock The branched deformations of the special {L}agrangian submanifolds.
\newblock {\em Geom. Funct. Anal.}, 33(5):1266--1321, 2023.

\bibitem[He25]{Z3He}
Siqi He.
\newblock Existence of nondegenerate {$\Bbb Z/2$} harmonic 1-forms via {$\Bbb Z_3$} symmetry.
\newblock {\em Geom. Dedicata}, 219(2):Paper No. 28, 13, 2025.

\bibitem[HMT23]{IndexHMT}
Andriy Haydys, Rafe Mazzeo, and Ryosuke Takahashi.
\newblock An index theorem for {$\mathbb{Z}/2$}-harmonic spinors branching along a graph.
\newblock {\em arXiv preprint arXiv:2310.15295}, 2023.

\bibitem[Hon04]{hondatrans}
Ko~Honda.
\newblock Transversality theorems for harmonic forms.
\newblock {\em Rocky Mountain J. Math.}, 34(2):629--664, 2004.

\bibitem[HP24]{HeParker}
Siqi He and Gregory~J. Parker.
\newblock {$\mathbb{Z}_2$}-harmonic spinors and 1-forms on connected sums and torus sums of 3-manifolds.
\newblock {\em arXiv preprint arxiv:2407.10922}, 2024.

\bibitem[HW15]{HaydysWalpuski}
Andriy Haydys and Thomas Walpuski.
\newblock A compactness theorem for the {S}eiberg-{W}itten equation with multiple spinors in dimension three.
\newblock {\em Geom. Funct. Anal.}, 25(6):1799--1821, 2015.

\bibitem[HWZ24]{HeWentworthZhang}
Siqi He, Richard Wentworth, and Boyu Zhang.
\newblock {$\mathbb{Z}$/2} harmonic 1-forms, {$\mathbb{R}$}-trees, and the {M}organ-{S}halen compactification.
\newblock {\em arXiv preprint arXiv:2409.04956}, 2024.

\bibitem[Mil65]{MilnorHcob}
John Milnor.
\newblock {\em Lectures on the {$h$}-cobordism theorem}.
\newblock Princeton University Press, Princeton, NJ, 1965.
\newblock Notes by L. Siebenmann and J. Sondow.

\bibitem[MS84]{MorganShalenI}
John~W. Morgan and Peter~B. Shalen.
\newblock Valuations, trees, and degenerations of hyperbolic structures. {I}.
\newblock {\em Ann. of Math. (2)}, 120(3):401--476, 1984.

\bibitem[Par24]{parker2024gluing}
Gregory~J. Parker.
\newblock Gluing $\mathbb{Z}_2$-harmonic spinors and {S}eiberg-{W}itten monopoles on 3-manifolds.
\newblock {\em arXiv preprint arXiv:2402.03682}, 2024.

\bibitem[Par25]{parker2025deformations}
Gregory~J. Parker.
\newblock Deformations of {$\mathbb{Z}_2$}-harmonic spinors on 3-manifolds.
\newblock {\em arXiv preprint arXiv:2301.06245}, 2025.

\bibitem[Sal24]{salm}
Willem~Adriaan Salm.
\newblock Construction of $\mathbb{Z}_2$-harmonic 1-forms on closed 3-manifolds with long cylindrical necks.
\newblock {\em arXiv preprint arxiv:2410.07015}, 2024.

\bibitem[Tak15]{Takahashi}
Ryosuke Takahashi.
\newblock {\em The moduli space of {$S^1$}-type zero loci for {$\mathbb{Z}/2$}-harmonic spinors in dimension 3}.
\newblock ProQuest LLC, Ann Arbor, MI, 2015.
\newblock Thesis (Ph.D.)--Harvard University.

\bibitem[Tak18]{TakahashiIndex}
Ryosuke Takahashi.
\newblock Index theorem for {$\mathbb{Z}/2$}-harmonic spinors.
\newblock {\em Math. Res. Lett.}, 25(5):1645--1671, 2018.

\bibitem[Tau13]{TaubesPSL2C}
Clifford~Henry Taubes.
\newblock {${\rm PSL}(2;\Bbb C)$} connections on 3-manifolds with {${\rm L}^2$} bounds on curvature.
\newblock {\em Camb. J. Math.}, 1(2):239--397, 2013.

\bibitem[Tau14]{taubes2014}
Clifford~Henry Taubes.
\newblock The zero loci of {$\mathbb{Z}/2$} harmonic spinors in dimension 2, 3 and 4.
\newblock {\em arXiv preprint arXiv:1407.6206}, 2014.

\bibitem[Tau20]{taubes20}
Clifford~Henry Taubes.
\newblock The {$\mathbb{R}$} invariant solutions to the {K}apustin-{W}itten equations on {$(0,\infty) \times \mathbb{R}^2 \times \mathbb{R}$} with generalized {N}ahm pole asymptotics.
\newblock {\em arXiv preprint arXiv:1903.03539}, 2020.

\bibitem[Tau22]{Taubes22}
Clifford~Henry Taubes.
\newblock Sequences of {N}ahm pole solutions to the {$\mathrm{SU}(2)$} {K}apustin-{W}itten equations.
\newblock {\em arXiv preprint arXiv:1805.02773}, 2022.

\bibitem[Tel83]{Teleman}
Nicolae Teleman.
\newblock The index of signature operators on {L}ipschitz manifolds.
\newblock {\em Inst. Hautes \'Etudes Sci. Publ. Math.}, (58):39--78, 1983.

\bibitem[TW20]{TWI}
C.~H. Taubes and Y.~Wu.
\newblock Examples of singularity models for {$\mathbb{Z}/2$} harmonic 1-forms and spinors in dimensional three.
\newblock In {\em Proceedings of the {G}\"okova {G}eometry-{T}opology {C}onferences 2018/2019}, pages 37--66. Int. Press, Somerville, MA, 2020.

\bibitem[TW24]{TWII}
C.~H. Taubes and Y.~Wu.
\newblock Topological aspects of {$\Bbb{Z} / 2 \Bbb{Z}$} eigenfunctions for the {L}aplacian on {$S^2$}.
\newblock {\em J. Differential Geom.}, 128(1):379--462, 2024.

\bibitem[Vol08]{intrinsic}
Evgeny Volkov.
\newblock Characterization of intrinsically harmonic forms.
\newblock {\em J. Topol.}, 1(3):643--650, 2008.

\bibitem[Wan93]{ShuguangWang}
Shuguang Wang.
\newblock Moduli spaces over manifolds with involutions.
\newblock {\em Math. Ann.}, 296(1):119--138, 1993.

\bibitem[WZ21]{WalpuskiZhang}
Thomas Walpuski and Boyu Zhang.
\newblock On the compactness problem for a family of generalized {S}eiberg-{W}itten equations in dimension 3.
\newblock {\em Duke Math. J.}, 170(17):3891--3934, 2021.

\bibitem[Yan]{Yan}
Dashen Yan.
\newblock On {$\mathbb{Z}/2$} harmonic {$1$}-forms with shrinking branching sets, in preparation.

\bibitem[Zha22]{Boyu17}
Boyu Zhang.
\newblock Rectifiability and {M}inkowski bounds for the zero loci of {$\Bbb Z/2$} harmonic spinors in dimension 4.
\newblock {\em Comm. Anal. Geom.}, 30(7):1633--1681, 2022.

\end{thebibliography}
